\documentclass[11pt]{article}

\usepackage{amsmath, amsfonts, amssymb, amsthm,  graphicx, mathtools, enumerate,amsbsy, dsfont}

\usepackage[blocks, affil-it]{authblk}

\usepackage{mathrsfs}

\usepackage[numbers, square, sort]{natbib}
\usepackage[CJKbookmarks=true,
			bookmarksopen=true,
			colorlinks=true,
			citecolor=red,
			linkcolor=blue,
			anchorcolor=red,
			urlcolor=blue]{hyperref}
\usepackage[usenames]{color}

\usepackage[letterpaper, left=1.2truein, right=1.2truein, top = 1.2truein, bottom = 1.2truein]{geometry}
\usepackage[ruled, vlined, lined, commentsnumbered]{algorithm2e}
\usepackage[font=small]{caption}
\usepackage{prettyref,soul}

\newtheorem{lemma}{Lemma}[section]

\newtheorem{thm}{Theorem}[section]

\newrefformat{eq}{(\ref{#1})}
\newrefformat{chap}{Chapter~\ref{#1}}
\newrefformat{sec}{Section~\ref{#1}}
\newrefformat{algo}{Algorithm~\ref{#1}}
\newrefformat{fig}{Fig.~\ref{#1}}
\newrefformat{tab}{Table~\ref{#1}}
\newrefformat{rmk}{Remark~\ref{#1}}
\newrefformat{clm}{Claim~\ref{#1}}
\newrefformat{def}{Definition~\ref{#1}}
\newrefformat{cor}{Corollary~\ref{#1}}
\newrefformat{lmm}{Lemma~\ref{#1}}
\newrefformat{lemma}{Lemma~\ref{#1}}
\newrefformat{prop}{Proposition~\ref{#1}}
\newrefformat{app}{Appendix~\ref{#1}}
\newrefformat{ex}{Example~\ref{#1}}
\newrefformat{exer}{Exercise~\ref{#1}}
\newrefformat{soln}{Solution~\ref{#1}}
\newrefformat{cond}{Condition~\ref{#1}}

\def\Var{\textsf{Var}} 

\def\h{\textsf{H}}

\def\text#1{\mbox{\rm #1}}

\def\sgn{\text{sign}}
\def\h{\textsf{H}}
\def\k{\textsf{K}}
\def\F{\textsf{F}}
\def\n{{ \mathcal{N} }}
\def\S{\mathfrak{S}}

\DeclarePairedDelimiter{\ceil}{\lceil}{\rceil}

\newcommand{\argmin}{\mathop{\rm argmin}}

\newcommand{\indc}[1]{{\mathbb{I}\left\{{#1}\right\}}}
\newcommand{\norm}[1]{\left\|{#1} \right\|}

\newcommand{\wh}{\widehat}
\newcommand{\wt}{\widetilde}

\newcommand{\opnorm}[1]{\|#1\|_{\rm op}}

\newtheorem*{condb'}{Condition B'}

\newcommand{\br}[1]{\left( #1 \right)}
\newcommand{\sbr}[1]{\left[ #1 \right]}
\newcommand{\cbr}[1]{\left\{ #1 \right\}}

\newcommand{\pbr}[1]{\p\left( #1 \right)}
\newcommand{\mathr}{\mathbb{R}}
\newcommand{\mathn}{\mathcal{N}}
\newcommand{\abs}[1]{\left| #1 \right|}

\newcommand{\p}{\mathbb{P}}
\newcommand{\E}{\mathbb{E}}

\title{Optimal Full Ranking from Pairwise Comparisons
}

\author[1]{Pinhan Chen}
\author[1]{Chao Gao}
\author[2]{Anderson Y. Zhang}
\affil[1]{
University of Chicago
}

\affil[2]{
University of Pennsylvania
}

\begin{document}
\maketitle

\begin{abstract}

We consider the problem of ranking $n$ players from partial pairwise comparison data under the Bradley-Terry-Luce model. For the first time in the literature, the minimax rate of this ranking problem is derived with respect to the Kendall's tau distance that measures the difference between two rank vectors by counting the number of inversions. The minimax rate of ranking exhibits a transition between an exponential rate and a polynomial rate depending on the magnitude of the signal-to-noise ratio of the problem. To the best of our knowledge, this phenomenon is unique to full ranking and has not been seen in any other statistical estimation problem. To achieve the minimax rate, we propose a divide-and-conquer ranking algorithm that first divides the $n$ players into groups of similar skills and then computes local MLE within each group. The optimality of the proposed algorithm is established by a careful approximate independence argument between the two steps.

\smallskip

\end{abstract}



\section{Introduction}\label{sec:intro}

Given partially observed pairwise comparison data from $n$ players, we are interested in ranking the players according to their skills by aggregating the comparison results. This high-dimensional statistical estimation problem has important applications in many areas such as recommendation systems \cite{baltrunas2010group, cao2018attentive}, sports and gaming \cite{beaudoin2018computationally, csato2013ranking, minka2018trueskill, herbrich2007trueskill, motegi2012network,sha2016chalkboarding}, web search \citep{dwork2001rank,cossock2006subset}, social choices \cite{mcfadden1973conditional, mcfadden2000mixed, saaty1990decision, louviere2000stated, manski1977structure}, psychology \cite{choo2004common, thurstone1927law, luce1977choice},  information retrieval \cite{liu2011learning,cao2007learning},  etc. In this paper, we focus on arguably one of the most widely used parametric models, the Bradley-Terry-Luce (BTL) model \citep{bradley1952rank,luce2012individual}. That is, we observe $L$ games played between $i$ and $j$, and the outcome is modeled by
\begin{equation}
y_{ijl}\stackrel{ind}{\sim}\text{Bernoulli}\left(\frac{w_i^*}{w_i^*+w_j^*}\right),\quad l=1,\cdots, L. \label{eq:BTL-w}
\end{equation}
We only observe outcomes from a small subset of pairs. This subset $E$ is modeled by edges generated by an Erd\H{o}s-R\'{e}nyi random graph \citep{erdHos1960evolution} with connection probability $p$ on the $n$ players. More details of the model will be given in Section \ref{sec:ml}. With the observations $\{y_{ijl}\}_{(i,j)\in E,l\in[L]}$, our goal is to optimally recover the ranks of the skill parameters $w_i^*$'s.

The literature on ranking under the BTL model has been mainly focused on the so-called top-$k$ ranking problem. Let $r^*$ be the rank vector of the $n$ players. In other words, $r^*$ is a permutation such that $r_i^*=j$ if $w_i^*$ is the $j$th largest number among $\{w_i^*\}_{i\in[n]}$. The goal of top-$k$ ranking is to recover the set $\{i\in[n]: r_i^*\leq k\}$ from the pairwise comparison data. Theoretical properties of the top-$k$ ranking problem have been studied by \cite{chen2015spectral,jang2016top,chen2017competitive,jang2017optimal,chen2019spectral} and references therein. Recently, it was shown by \cite{chen2019spectral} that both the MLE and the spectral ranking algorithm proposed by \cite{negahban2017rank} can exactly recover the set of top-$k$ players with high probability under optimal sample complexity up to some constant factor. In terms of partial recovery, the minimax rate of the problem under a normalized Hamming distance was derived by \cite{chen2020partial}.

In this paper, we study the problem of \textit{full ranking}, the estimation of the entire rank vector $r^*$. To the best of our knowledge, theoretical analysis of full ranking under the BTL model has not been considered in the literature yet. We rigorously formulate the full ranking problem from a decision-theoretic perspective, and derive the minimax rate with respect to a loss function that measures the difference between two permutation vectors. To be specific, our main result of the paper shows that
\begin{equation}
\inf_{\wh{r}\in\S_n}\sup_{r^*\in\S_n}\mathbb{E}\k(\wh{r},r^*)\asymp \begin{cases}
\exp\left(-\Theta(Lp\beta)\right), & Lp\beta > 1, \\
n\wedge\sqrt{\frac{1}{Lp\beta}}, & Lp\beta \leq 1,\\
\end{cases}\label{eq:minimax-BTL-intro}
\end{equation}
where $\S_n$ is the set of all rank vectors of size $n$, $\k(\wh{r},r^*)$ is the \textit{Kendall's tau distance} that counts the number of inversions between two ranks, and $\beta$ is the minimal gap between skill parameters of different players. The precise definitions of these quantities will be given in Section \ref{sec:ml}. The minimax rate (\ref{eq:minimax-BTL-intro}) exhibits a transition between an exponential rate and a polynomial rate. This is a unique phenomenon in the estimation of a full rank vector. In contrast, under the same BTL model, the minimax rate of estimating the skill parameters is always polynomial \citep{negahban2017rank,chen2019spectral}, and the minimax rate of top-$k$ ranking is always exponential \citep{chen2020partial}. Whether (\ref{eq:minimax-BTL-intro}) is exponential or polynomial depends on the value of $Lp\beta$ that plays the role of signal-to-noise ratio. When $Lp\beta > 1$, the exponential minimax rate is a consequence of the discreteness of a rank vector. On the other hand, when $Lp\beta \leq 1$, the discrete nature of ranking is blurred by the noise, and thus estimating the rank vector is effectively estimating a continuous parameter, which leads to a polynomial rate. A more detailed statement of the minimax rate (\ref{eq:minimax-BTL-intro}) with an explicit exponent in the regime of exponential rate will be given in Section \ref{sec:main}.

Achieving the minimax rate (\ref{eq:minimax-BTL-intro}) is a nontrivial problem. To this end, we propose a \textit{divide-and-conquer} algorithm that first partitions the $n$ players into several leagues and then computes a local MLE using games in each league. Finally, a full rank vector is obtained by aggregating local ranking results from all leagues. The divide-and-conquer technique is the basis of efficient algorithms for all kinds of sorting problems \citep{sedgewick1978implementing,katajainen1997meticulous,knuth1997art}. Our adaption of this classical technique in the optimal full ranking is motivated by both information-theoretic and computational considerations. From an information-theoretic perspective, games between players whose skill parameters are significantly different from each other have little effect on the final ranking result. This phenomenon can be revealed by a simple local Fisher information calculation of each player. The league partition step groups players with similar skill parameters together, thus maximizing information in the follow-up step of local MLE. From a computational perspective, the local MLE computed within each league involves an objective function whose Hessian matrix is well conditioned, a property that is crucial for efficient convex optimization. The description and the analysis of our algorithm are given in Section \ref{sec:dnc}.

Before the end of the introduction section, let us also remark that the more general problem of permutation estimation has also been considered in various other settings in the literature \citep{braverman2008noisy,braverman2009sorting,collier2013permutation,collier2016minimax,pananjady2016linear,gao2017phase,mao2018minimax,gao2019iterative,pananjady2020worst}. For instance, in the problem of noisy sorting \citep{braverman2009sorting,mao2018minimax}, one assumes a data generating process that satisfies $\mathbb{P}(y_{ijl}=1)>\frac{1}{2}+\gamma$ when $r_i^*<r_j^*$. In the feature matching problem \citep{collier2016minimax,collier2013permutation}, it is assumed that $X_i-Y_{r_i^*}\sim\n(0,\sigma^2_i)$ for some permutation $r^*$, and the goal is to match the two data sequences $X$ and $Y$ by recovering the unknown permutation. An extension of this problem, called shuffled regression, assumes that the response variable $y_{i}$ and regression function $x_{r_i^*}^T\beta$ are linked by an unknown permutation. Estimation of the unknown permutation in shuffled regression has been considered by \cite{pananjady2016linear}.

The rest of the paper is organized as follows. We introduce the problem setting in Section \ref{sec:ml}. The minimax rate of the  full ranking is presented in Section \ref{sec:main}. In Section \ref{sec:dnc}, we introduce and analyze a divide-and-conquer algorithm that achieves the minimax rate. Numerical studies of the algorithm are given in Section \ref{sec:simulation}. In Section \ref{sec:disc}, we discuss a few extensions and future projects that are related to the paper. Finally, Section \ref{sec:pf} collects technical proofs of the results of the paper.

We close this section by introducing some notation that will be used in the paper. For an integer $d$, we use $[d]$ to denote the set $\{1,2,...,d\}$. Given two numbers $a,b\in\mathbb{R}$, we use $a\vee b=\max(a,b)$ and $a\wedge b=\min(a,b)$. For any $x\in\mathbb{R}$, $\lfloor x\rfloor$ stands for the largest integer that is no greater than $x$ and $\lceil x\rceil$ is the smallest integer that is no less than $x$.  For two positive sequences $\{a_n\},\{b_n\}$, $a_n\lesssim b_n$ or $a_n=O(b_n)$ means $a_n\leq Cb_n$ for some constant $C>0$ independent of $n$, $a_n=\Omega(b_n)$ means $b_n=O(a_n)$, and we use $a_n\asymp b_n$ or $a_n=\Theta(b_n)$ when both $a_n\lesssim b_n$ and $b_n\lesssim a_n$ hold. We also write $a_n=o(b_n)$ when $\limsup_n\frac{a_n}{b_n}=0$. For a set $S$, we use $\indc{S}$ to denote its indicator function and $|S|$ to denote its cardinality. We use the notation $S =S_1 \uplus S_2$ to denote a partition of $S$ such that $S_1\cap S_2 = \varnothing$ and $S= S_1 \cup S_2$.
For a vector $v\in\mathbb{R}^d$, its norms are defined by $\norm{v}_1=\sum_{i=1}^d|v_i|$, $\norm{v}^2=\sum_{i=1}^dv_i^2$ and $\norm{v}_{\infty}=\max_{1\leq i\leq d}|v_i|$. For a matrix $A\in\mathbb{R}^{n\times m}$, we use $\opnorm{A}$ for its operator norm, which is the largest singular value. The notation $\mathds{1}_{d}$ means a $d$-dimensional column vector of all ones. Given $p,q\in(0,1)$, the Kullback-Leibler divergence is defined by $D(p\|q)=p\log\frac{p}{q}+(1-p)\log\frac{1-p}{1-q}$. For a natural number $n$, $\S_n$ is the set of permutations on $[n]$. The notation $\mathbb{P}$ and $\mathbb{E}$ are used for generic probability and expectation whose distribution is determined from the context.

\section{A Decision-Theoretic Framework of Full Ranking}\label{sec:ml}

\paragraph{The BTL Model.} Consider $n$ players, each associated with a positive latent skill parameter $w_i^*$ for $i\in[n]$. The games played among the $n$ players are modeled by an Erd\H{o}s-R\'{e}nyi random graph $A\sim\mathcal{G}(n,p)$. To be specific, we have $A_{ij}\stackrel{iid}{\sim}\text{Bernoulli}(p)$ for all $1\leq i<j\leq n$. For any pair $(i,j)$ such that $A_{ij}=1$, we observe the outcomes of $L$ games played between $i$ and $j$, modeled by the  Bradley-Terry-Luce (BTL) model (\ref{eq:BTL-w}). Our goal is to estimate the ranks of the $n$ players.

To formulate the problem of full ranking from a decision-theoretic perspective, we can reparametrize the BTL model (\ref{eq:BTL-w}) by a sorted vector $\theta^*$ and a rank vector $r^*$. A sorted vector $\theta^*$ satisfies $\theta_1^*\geq \theta_2^* \geq \cdots \geq \theta_n^*$, and a rank vector $r^*$ is an element of the permutation set $\S_n$. We have
\begin{equation}
y_{ijl}\stackrel{ind}{\sim}\text{Bernoulli}(\psi(\theta^*_{r^*_i}-\theta^*_{r^*_j})),\quad l=1,\cdots, L, \label{eq:BTL-theta}
\end{equation}
where $\psi(\cdot)$ is the sigmoid function $\psi(t)=\frac{1}{1+e^{-t}}$. In the original representation (\ref{eq:BTL-w}), we have $w_i^*=\exp(\theta_{r_i^*}^*)$ for all $i\in[n]$. With (\ref{eq:BTL-theta}), the full ranking problem is to estimate the rank vector $r^*$ from the random comparison data.

\paragraph{Loss Function for Full Ranking.}

To measure the difference between an estimator $\wh{r}\in\S_n$ and the true $r^*\in\S_n$, we introduce the \textit{Kendall's tau distance}, defined by
\begin{equation}
\k(\wh{r},r^*)=\frac{1}{n}\sum_{1\leq i<j\leq n}\indc{\sgn(\wh{r}_i-\wh{r}_j)\sgn(r_i^*-r_j^*)<0}, \label{eq:kendall}
\end{equation}
where $\text{sign}(x)$ represents the sign of $x$ and $n\k(\wh{r},r^*)$ counts the number of inversions between $\wh{r}$ and $r^*$. Another distance is  the normalized $\ell_1$ loss, defined as
\begin{equation}
\F(\wh{r},r^*)=\frac{1}{n}\sum_{i=1}^n\abs{\wh{r}_i-r_i^*}, \label{eq:footrule}
\end{equation}
also known as the \textit{Spearman's footrule}. 
%
The two loss functions can be related by the following inequality,
\begin{equation}
\frac{1}{2}\F(\wh{r},r^*) \leq \k(\wh{r},r^*) \leq \F(\wh{r},r^*). \label{eq:FandK}
\end{equation}
See \cite{diaconis1977spearman} for the derivation of (\ref{eq:FandK}). The inequality (\ref{eq:FandK}) establishes an equivalence between the estimation of the vector $r^*$ and that of the matrix of pairwise relation $\mathbb{I}\{r_i^*<r_j^*\}$, a key fact that we will explore in constructing an optimal algorithm.

A problem that is closely related to full ranking is called top-$k$ ranking. The goal of top-$k$ ranking is to identify the subset $\{i\in[n]:r_i^*\leq k\}$ from the random comparison data. In \cite{chen2020partial}, the minimax rate of top-$k$ ranking was studied under the loss function of normalized Hamming distance,
\begin{equation}
\h_k(\wh{r},r^*)=\frac{1}{2k}\left(\sum_{i=1}^n\indc{\wh{r}_i>k, r^*_i\leq k}+\sum_{i=1}^n\indc{\wh{r}_i\leq k, r^*_i> k}\right). \label{eq:h-loss}
\end{equation}
The comparison between (\ref{eq:kendall}) and (\ref{eq:h-loss}) reveals the key difference between the two problems. While top-$k$ ranking only requires a correct classification of the two groups, the quality of  the  full ranking depends on the accuracy of each individual $|\wh{r}_i-r_i^*|$. It is easy to see that $\k(\wh{r},r)=0$ implies $\h_k(\wh{r},r^*)=0$, but the opposite direction is not true.

\paragraph{Regularity of Skill Parameters.}

For the nuisance parameter $\theta^*$ of the model (\ref{eq:BTL-theta}), it is necessary that the skill parameters of neighboring players $\theta_i^*$ and $\theta_{i+1}^*$ are separated so that the identification of the ranks is possible. We introduce a parameter space that serves for this purpose. For any $\beta>0$ and any $C_0\geq 1$, define
$$\Theta_n(\beta,C_0)=\left\{\theta\in\mathbb{R}^n: \theta_1\geq \cdots \geq\theta_n, 1 \leq \frac{|\theta_i-\theta_j|}{\beta|i-j|}\leq C_0\text{ for any }i\neq j\right\}.$$
In other words, neighboring $\theta_i^*$ and $\theta_{i+1}^*$  are required to be separated by at least $\beta$.  The magnitude of $\beta$ then characterizes the difficulty of  full ranking. The number $C_0$ characterizes the regularity of the space of sorted vectors $\Theta_n(\beta,C_0)$. The special case $\Theta_n(\beta,1)$ only consists of fully regular $\theta$'s that can be written as $\theta_i=\alpha-\beta i$. Throughout the paper, we assume that $C_0\geq 1$ is an absolute constant, but allow $\beta$ to be a function of the sample size $n$, with the possibility that $\beta\rightarrow 0$.

The assumption $\theta^*\in\Theta_n(\beta,C_0)$ implies that the numbers $\theta_1^*,\cdots,\theta_n^*$ to be roughly evenly spaced. This assumption, which can be certainly relaxed, allows us to obtain relatively clean formulas of the minimax rate of  full ranking. By restricting our focus to the space $\Theta_n(\beta,C_0)$, we will develop a clear but nontrivial understanding of the full ranking problem in this paper. The extension of our results beyond $\theta^*\in\Theta_n(\beta,C_0)$ will be briefly discussed in Section \ref{sec:disc}.

\section{Minimax Rates of Full Ranking}\label{sec:main}

In this section, we present the minimax rate of  full ranking under the BTL model. To better understand the results, we first derive the minimax rate of  full ranking under a Gaussian pairwise comparison model in Section \ref{sec:Gaussian-result}. This allows us to highlight some of the unique and nontrivial features of the BTL model by comparing the minimax rates of the two different distributions. Readers who are already familiar with the BTL model can directly start with Section \ref{sec:intuition}.

\subsection{Results for a Gaussian Model}\label{sec:Gaussian-result}

Consider the same comparison scheme modeled by the Erd\H{o}s-R\'{e}nyi random graph $A\sim\mathcal{G}(n,p)$. For any pair $(i,j)$ such that $A_{ij}=1$, we independently observe
\begin{equation}
y_{ij}\sim\n(\theta_{r_i^*}^*-\theta_{r_j^*}^*,\sigma^2). \label{eq:dgp-Gaussian}
\end{equation}
The joint distribution of $\{A_{ij}\}$ and $\{y_{ij}\}$, under the above generating process, is denoted by $\mathbb{P}_{(\theta^*,\sigma^2,r^*)}$.
Estimation of the rank vector $r^*\in\S_n$ under the Gaussian model (\ref{eq:dgp-Gaussian}) is much less complicated than the same problem under (\ref{eq:BTL-theta}), because of the separate parametrization of mean and variance.

\begin{thm}\label{thm:Gaussian-minimax}
Assume $\theta^*\in\Theta_n(\beta,C_0)$ for some constant $C_0\geq 1$ and $\frac{np}{\log n}\rightarrow\infty$. Then, for any constant $\delta$ that can be arbitrarily small, we have
$$\inf_{\wh{r}\in\S_n}\sup_{r^*\in\S_n}\mathbb{E}_{(\theta^*,\sigma^2,r^*)}\k(\wh{r},r^*)\gtrsim \begin{cases}
\frac{1}{n-1}\sum_{i=1}^{n-1}\exp\left(-\frac{(1+\delta)np(\theta_i^*-\theta_{i+1}^*)^2}{4\sigma^2}\right), & \frac{np\beta^2}{\sigma^2} > 1, \\
n\wedge\sqrt{\frac{\sigma^2}{np\beta^2}}, & \frac{np\beta^2}{\sigma^2} \leq 1.\\
\end{cases}$$
Moreover, let $\wh{r}$ be the rank obtained by sorting the MLE $\wh{\theta}$, and then
$$\sup_{r^*\in\S_n}\mathbb{E}_{(\theta^*,\sigma^2,r^*)}\k(\wh{r},r^*)\lesssim \begin{cases}
\frac{1}{n-1}\sum_{i=1}^{n-1}\exp\left(-\frac{(1-\delta)np(\theta_i^*-\theta_{i+1}^*)^2}{4\sigma^2}\right)+n^{-5}, & \frac{np\beta^2}{\sigma^2} > 1, \\
n\wedge\sqrt{\frac{\sigma^2}{np\beta^2}}, & \frac{np\beta^2}{\sigma^2} \leq 1.\\
\end{cases}$$
Both inequalities are up to constant factors only depending on $C_0$ and $\delta$.
\end{thm}

Theorem \ref{thm:Gaussian-minimax} characterizes the statistical fundamental limit of  full ranking under the Gaussian comparison model. The result holds for each individual $\theta^*\in\Theta_n(\beta,C_0)$. It is interesting to note that the minimax rate exhibits a transition between an exponential rate and a polynomial rate. By scrutinizing the proof, the constant $\delta$ can be replaced by some sequence $\delta_n=o(1)$. Therefore, consider a special example $\theta^*\in\Theta_n(\beta,1)$, and the minimax rate (ignoring the $n^{-5}$ term) can be simplified as
\begin{equation}
\begin{cases}
\exp\left(-\frac{(1+o(1))np\beta^2}{4\sigma^2}\right), & \frac{np\beta^2}{\sigma^2} > 1, \\
n\wedge\sqrt{\frac{\sigma^2}{np\beta^2}}, & \frac{np\beta^2}{\sigma^2} \leq 1.\\
\end{cases}\label{eq:G-minimax-simp}
\end{equation}
The behavior of (\ref{eq:G-minimax-simp}) is illustrated in Figure \ref{fig:phase}.
\begin{figure}[h]
	\centering
	\includegraphics[width=0.7\textwidth]{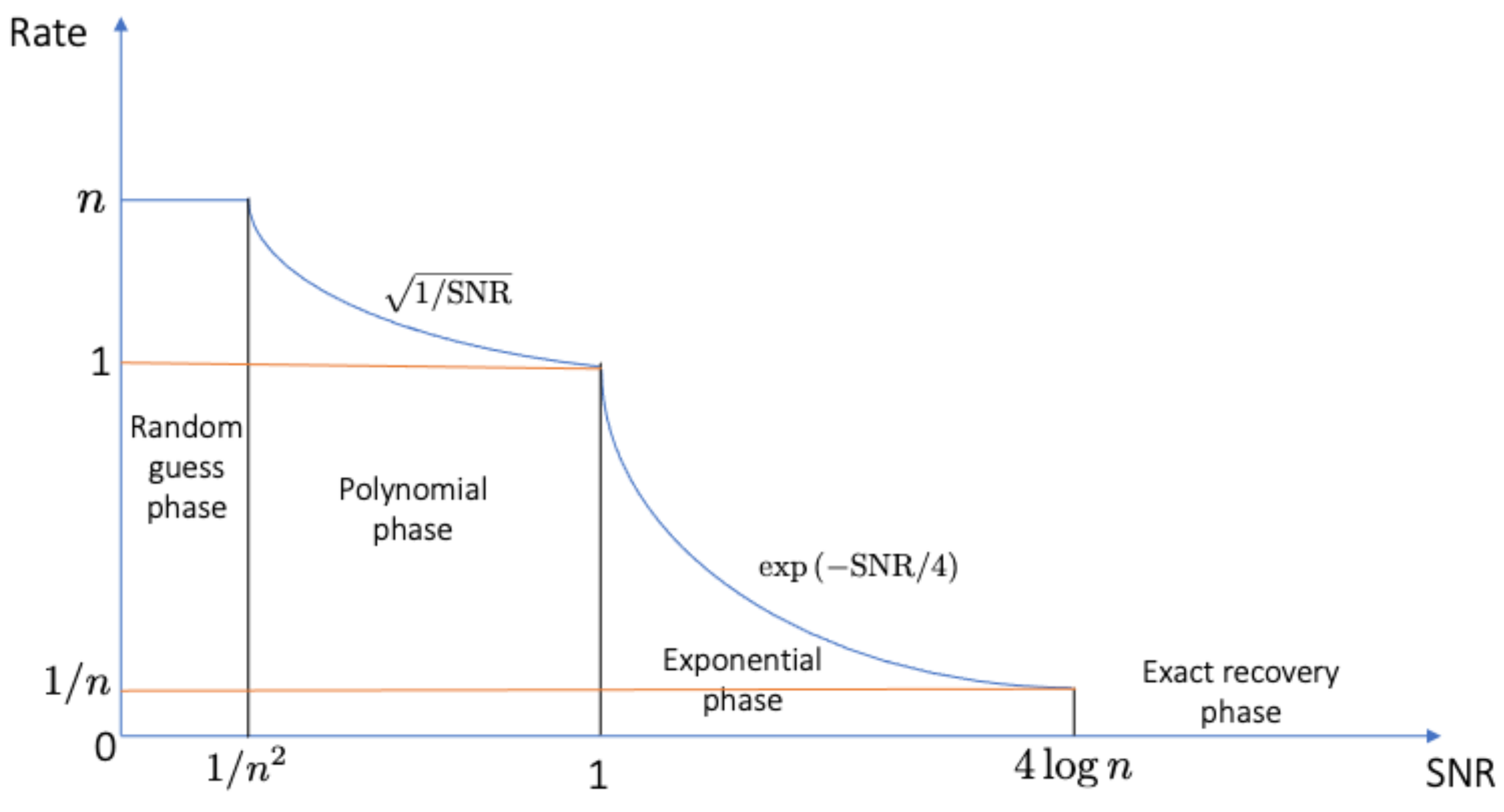}
	\caption{\textsl{Illustration of the the minimax rate of  full ranking.}}
	\label{fig:phase}
\end{figure}
The quantity $\frac{np\beta^2}{\sigma^2}$ plays the role of the signal-to-noise ratio of the ranking problem. In the high SNR regime $\frac{np\beta^2}{\sigma^2} > 1$, the difficulty of the ranking problem is dominated by whether the data can distinguish each $r_i^*$ from its neighboring values. Therefore, ranking is essentially a \textit{hypothesis testing} problem, which leads to an exponential rate. In the low SNR regime $\frac{np\beta^2}{\sigma^2} \leq 1$, the discrete nature of ranking is absent because of the noise level. The recovery of $r^*$ is equivalent to the estimation of a continuous vector in $\mathbb{R}^n$, which is essentially a \textit{parameter estimation} problem. The polynomial rate $n\wedge\sqrt{\frac{\sigma^2}{np\beta^2}}$ is the usual minimax rate for estimating an $n$-dimensional parameter under the $\ell_1$ loss. It is also worthing noting that the rate (\ref{eq:G-minimax-simp}) implies that the rank vector can be exactly recovered when $\frac{np\beta^2}{\sigma^2} > C\log n$ for any constant $C>4$. This is because in this regime, we have $\k(\wh{r},r^*)=o(n^{-1})$ with high probability by a direct application of Markov's inequality. According to the definition of $\k(\wh{r},r^*)$, we know $\k(\wh{r},r^*)=o(n^{-1})$ implies $\k(\wh{r},r^*)=0$.

The upper bound of Theorem \ref{thm:Gaussian-minimax} involves an extra $n^{-5}$ term in the high SNR regime. According to the proof, the number $5$ in the exponent can actually be replaced by an arbitrarily large constant. The $n^{-5}$ term does not contribute to the high-probability bound. By a direct application of Markov's inequality, when $\frac{np\beta^2}{\sigma^2}\rightarrow\infty$, we have
\begin{equation}
\k(\wh{r},r^*) \lesssim \frac{1}{n-1}\sum_{i=1}^{n-1}\exp\left(-\frac{(1-\delta)np(\theta_i^*-\theta_{i+1}^*)^2}{4\sigma^2}\right), \label{eq:no-n-5}
\end{equation}
with probability $1-o(1)$. Notice that the high-probability bound (\ref{eq:no-n-5}) does not involve the $n^{-5}$. This is because when $\k(\wh{r},r^*)$ is nonzero, it must be at least $n^{-1}$ by the definition of the loss function. Therefore, $n^{-5}$ can always be absorbed into the other term of the upper bound.

We also remark that the condition $\frac{np}{\log n}\rightarrow\infty$ guarantees that the random graph $A$ is connected with high probability. It is well known that when $p\leq c\frac{\log n}{n}$ for some sufficiently small constant $c>0$, the random graph has several disjoint components, which makes the comparisons between different components impossible.

An optimal estimator that achieves the minimax rate is the rank vector induced by the MLE, which is defined by
\begin{equation}
\wh{\theta}\in\argmin_{\theta}\sum_{1\leq i<j\leq n}A_{ij}\Big(y_{ij}-(\theta_i-\theta_j)\Big)^2. \label{eq:Gaussian-ls}
\end{equation}
We note that the parameter $\theta^*$ in (\ref{eq:dgp-Gaussian}) is identifiable up to a global shift. We may put an extra constraint $\mathds{1}_n^T\theta=0$ in the least-squares estimator above, so that $\wh{\theta}$ is uniquely defined. However, this constraint is actually not essential, since even without it, the rank vector $\wh{r}$ induced by $\wh{\theta}$ is still uniquely defined. To study the property of $\wh{\theta}$, we introduce a diagonal matrix $D\in\mathbb{R}^{n\times n}$ whose entries are given by $D_{ii}=\sum_{j\in[n]\backslash\{i\}}A_{ij}$. Then, $\mathcal{L}_A=D-A$ is the graph Laplacian of $A$. A standard least-squares analysis of (\ref{eq:Gaussian-ls}) leads to the fact that up to some global shift,
\begin{equation}
\wh{\theta}\sim \n\left(\theta^*,\sigma^2\mathcal{L}_A^{\dagger}\right), \label{eq:MLE-distribution}
\end{equation}
where $\mathcal{L}_A^{\dagger}$ is the generalized inverse of $\mathcal{L}_A$. The covariance matrix of (\ref{eq:MLE-distribution}) is optimal by achieving the intrinsic Cram{\'e}r-Rao lower bound of the problem \citep{boumal2013intrinsic}. Without loss of generality, we can assume $r_i^*=i$ for each $i\in[n]$. Then, by the definition of the loss function (\ref{eq:kendall}), we have
$$\mathbb{E}\k(\wh{r},r^*) = \frac{1}{n}\sum_{1\leq i<j\leq n}\mathbb{P}\left(\wh{r}_i>\wh{r}_j\right) = \frac{1}{n}\sum_{1\leq i<j\leq n}\mathbb{P}\left(\wh{\theta}_i>\wh{\theta}_j\right),$$
and each $\mathbb{P}\left(\wh{\theta}_i>\wh{\theta}_j\right)$ can be accurately estimated by a Gaussian tail bound under the distribution (\ref{eq:MLE-distribution}), which then leads to the upper bound result of Theorem \ref{thm:Gaussian-minimax}. A detailed proof of Theorem \ref{thm:Gaussian-minimax}, including a lower bound analysis, is given in Section \ref{sec:pf-G}.

\subsection{Some Intuitions for the BTL Model} \label{sec:intuition}

Before stating the minimax rate for the BTL model, we discuss a few key differences that one can expect from the result. Without of loss of generality, we assume $r_i^*=i$ for all $i\in[n]$ throughout the discussion to simplify the notation. Let us consider a problem of oracle estimation of the skill parameter of the first player $\theta_1^*$. To be specific, we would like to estimate $\theta_1^*$ by assuming that $\theta_2^*,\cdots,\theta_n^*$ are known. The Fisher information of this problem can be shown as
\begin{equation}
I^{\rm oracle}(\theta_1^*)=Lp\sum_{j=2}^n\psi'(\theta_1^*-\theta_j^*). \label{eq:oracle-fisher}
\end{equation}
The formula (\ref{eq:oracle-fisher}) characterizes the individual contribution of each player to the overall information in estimating $\theta_1^*$. That is, the information from the games between $1$ and $j$ is quantified by $Lp\psi'(\theta_1^*-\theta_j^*)$. Since $\psi'(t)=\frac{e^{t}}{(1+e^{t})^2}\leq e^{-|t|}$, we have
$$\psi'(\theta_1^*-\theta_j^*)\leq \exp\left(-|\theta_1^*-\theta_j^*|\right).$$
In other words, $\psi'(\theta_1^*-\theta_j^*)$ is an exponentially small function of the skill difference $|\theta_1^*-\theta_j^*|$. This means for players whose skills are significantly different from $\theta_1^*$, their games with Player 1 offers little information in the inference of $\theta_1^*$.

This phenomenon can be intuitively understood from the following simple example illustrated in Figure \ref{fig:example}.
\begin{figure}[h]
	\centering
	\includegraphics[width=0.3\textwidth]{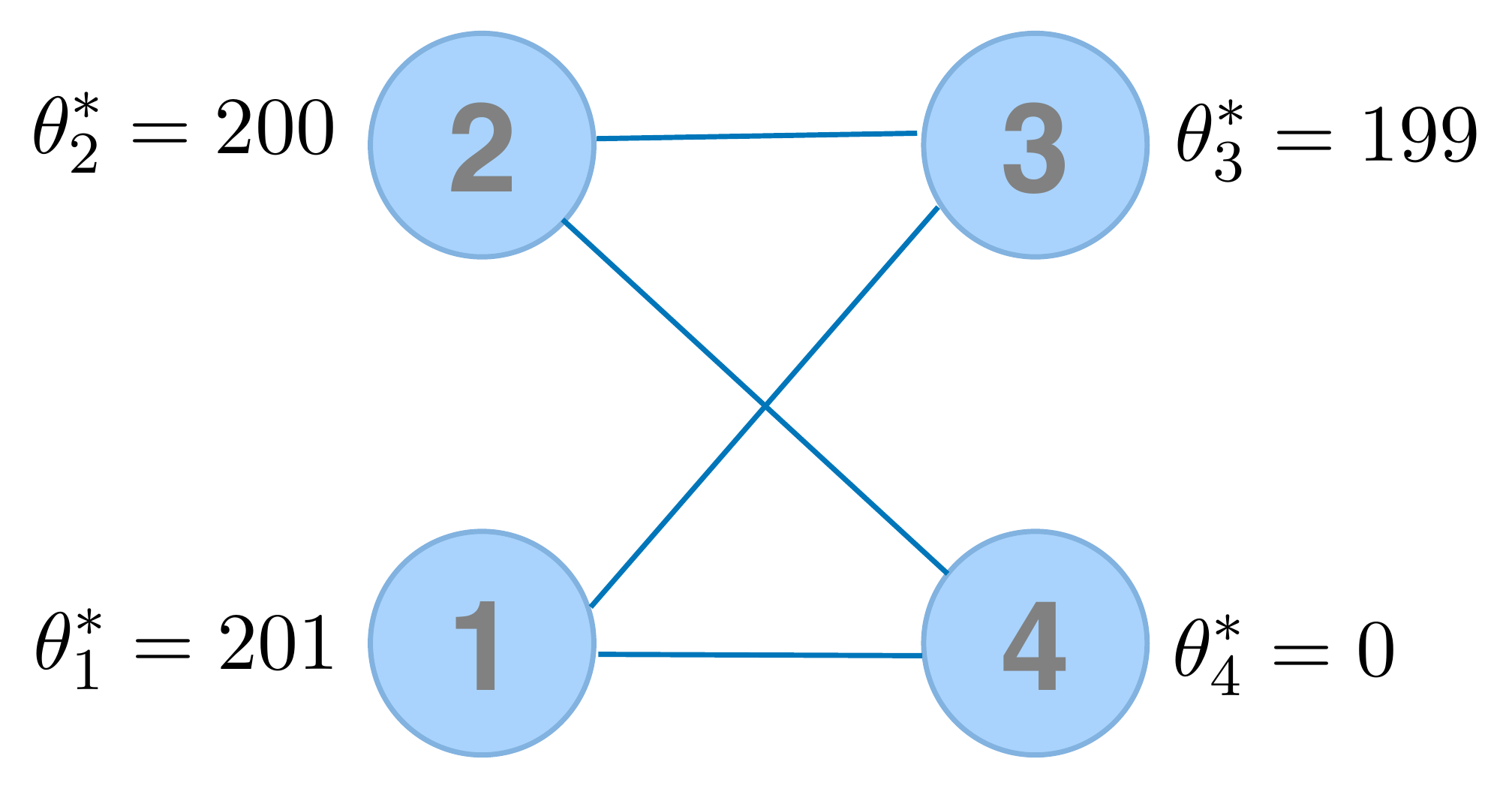}
	\caption{\textsl{A comparison graph of four players.}}
	\label{fig:example}
\end{figure}
Consider four players with skill parameters $(\theta_1^*,\theta_2^*,\theta_3^*,\theta_4^*)=(201,200,199,0)$, and we would like to compare the first two players. With the direct link between $1$ and $2$ missing, the only way to compare Players 1 and 2 is through their performances against Players 3 and 4. Since both $\theta_1^*-\theta_4^*=201$ and $\theta_2^*-\theta_4^*=200$ are very large numbers, it is very likely that Player 4 will lose all games against Players 1 and 2. On the other hand, we have $\theta_1^*-\theta_3^*=2$ and $\theta_2^*-\theta_3^*=1$, and thus Player 3 is likely to lose more games against Player 1 than against Player 2. Therefore, we can conclude that Player 1 is stronger than Player 2 based on their performances against Player 3, and the games against Player 4 offer no information for this purpose. This example clearly illustrates that closer opponents are more informative.

Mathematically, for any $\theta^*\in\Theta_n(\beta,C_0)$ and any $M>0$, it can be easily shown that
\begin{equation}
I^{\rm oracle}(\theta_1^*)\leq (1+O(e^{-M}))Lp\sum_{j\leq M/\beta}\psi'(\theta_1^*-\theta_j^*). \label{eq:oracle-fisher-trunc}
\end{equation}
Therefore, (\ref{eq:oracle-fisher}) and (\ref{eq:oracle-fisher-trunc}) imply that
\begin{equation}
I^{\rm oracle}(\theta_1^*)=(1+O(e^{-M}))Lp\sum_{j\leq M/\beta}\psi'(\theta_1^*-\theta_j^*). \label{eq:oracle-fisher-asymp}
\end{equation}
There is no need to consider the games against players with $j>M/\beta$. Moreover, we also observe from (\ref{eq:oracle-fisher-asymp}) that the parameter $\beta$ plays two different roles in the BTL model:
\begin{enumerate}
\item The parameter $\beta$ is the minimal gap between different players, and it quantifies the signal strength of the BTL model.
\item The number $1/\beta$ quantifies the number of close opponents of each player, and thus $p/\beta$ can be understood as the effective sample size of the BTL model.
\end{enumerate}
While the first role is also shared by the $\beta$ in the Gaussian comparison model (\ref{eq:dgp-Gaussian}), the second role dramatically distinguishes the BTL model from its Gaussian counterpart. The effective sample size of the Gaussian model is $np$, compared with $p/\beta$ of the BTL model. This critical difference is a consequence of the nonlinearity of the logistic function. Increasing $\beta$ magnifies the signal but reduces the effective sample size at the same time. The precise role of $\beta$ in  full ranking under the BTL model will be clarified by the formula of the minimax rate.

\subsection{Results for the BTL Model}

To present the minimax rate of  full ranking under the BTL model, we first introduce some new quantities. For any $i\in[n]$, define
\begin{equation}
V_i(\theta^*)=\frac{n}{\sum_{j\in[n]\backslash\{i\}}\psi^\prime(\theta_i^*-\theta_j^*)}. \label{eq:BTL-var}
\end{equation}
The quantity (\ref{eq:BTL-var}) is interpreted as the variance function of the $i$th best player. With a slight abuse of notation, the expectation associated with the BTL model is denoted as $\mathbb{E}_{(\theta^*,r^*)}$.

\begin{thm}\label{thm:BTL-minimax}
Assume $\theta^*\in\Theta_n(\beta,C_0)$ for some constant $C_0\geq 1$ and $\frac{p}{(\beta\vee n^{-1})\log n}\rightarrow\infty$. Then, for any constant $\delta$ that can be arbitrarily small, we have
$$\inf_{\wh{r}\in\S_n}\sup_{r^*\in\S_n}\mathbb{E}_{(\theta^*,r^*)}\k(\wh{r},r^*)\gtrsim \begin{cases}
\frac{1}{n-1}\sum_{i=1}^{n-1}\exp\left(-\frac{(1+\delta)npL(\theta_i^*-\theta_{i+1}^*)^2}{4V_{i}(\theta^*)}\right), & \frac{Lp\beta^2}{\beta\vee n^{-1}} > 1, \\
n\wedge\sqrt{\frac{\beta\vee n^{-1}}{Lp\beta^2}}, & \frac{Lp\beta^2}{\beta\vee n^{-1}} \leq 1.\\
\end{cases}$$
Moreover, let $\wh{r}$ be the rank computed by Algorithm \ref{alg:whole}, and then if additionally $\frac{L}{\log n}\rightarrow\infty$, we have
$$\sup_{r^*\in\S_n}\mathbb{E}_{(\theta^*,r^*)}\k(\wh{r},r^*)\lesssim \begin{cases}
\frac{1}{n-1}\sum_{i=1}^{n-1}\exp\left(-\frac{(1-\delta)npL(\theta_i^*-\theta_{i+1}^*)^2}{4V_{i}(\theta^*)}\right)+n^{-5}, & \frac{Lp\beta^2}{\beta\vee n^{-1}} > 1, \\
n\wedge\sqrt{\frac{\beta\vee n^{-1}}{Lp\beta^2}}, & \frac{Lp\beta^2}{\beta\vee n^{-1}} \leq 1.\\
\end{cases}$$
Both inequalities are up to constant factors only depending on $C_0$ and $\delta$.
\end{thm}

Similar to Theorem \ref{thm:Gaussian-minimax}, the result of Theorem \ref{thm:BTL-minimax} holds for each individual $\theta^*\in\Theta_n(\beta,C_0)$, and the minimax rate also exhibits a transition between an exponential rate and a polynomial rate. To better understand the minimax rate formula, we use Lemma \ref{lem:sum-psi-prime} to quantify the order of the variance function $V_{i}(\theta^*)$. There exist constants $C_1,C_2>0$, such that
$$C_1\left(\beta\vee \frac{1}{n}\right)\leq \frac{V_{i}(\theta^*)}{n}\leq C_2\left(\beta\vee \frac{1}{n}\right).$$
Therefore, when $\beta\gtrsim n^{-1}$, the minimax rate (ignoring the $n^{-5}$ term) can be simplified as
\begin{equation}
\begin{cases}
\exp\left(-\Theta(Lp\beta)\right), & Lp\beta > 1, \\
n\wedge\sqrt{\frac{1}{Lp\beta}}, & Lp\beta \leq 1.\\
\end{cases} \label{eq:BTL-minimax-simp}
\end{equation}
The formula (\ref{eq:BTL-minimax-simp}) also exhibits a transition between a polynomial rate and an exponential rate. Its behavior can be illustrated by Figure \ref{fig:phase} with SNR being $\Theta(LP\beta)$.
Compared with the minimax rate (\ref{eq:G-minimax-simp}) for the Gaussian comparison model, the dependence of (\ref{eq:BTL-minimax-simp}) on $\beta$ is weaker. This is a consequence of the dual roles of $\beta$  discussed in Section \ref{sec:intuition}. In fact, by writing
$$Lp\beta=L\beta^{-1}p \beta^2,$$
we can directly observe the effects of $\beta^{-1}p$ and $\beta^2$ as the effective sample size and the signal strength, respectively. On the other hand, the number of total players $n$ has very little effect on the minimax rate formula.

The condition $\frac{p}{(\beta\vee n^{-1})\log n}\rightarrow\infty$ required by Theorem \ref{thm:BTL-minimax} can be equivalently written as $\frac{np}{\log n}\rightarrow\infty$ and $\frac{p}{\beta\log n}\rightarrow\infty$. Compared with the setting of Theorem \ref{thm:Gaussian-minimax}, an additional condition $\frac{p}{\beta\log n}\rightarrow\infty$ is assumed for the BTL model. This condition can be seen as a consequence of the Fisher information formula (\ref{eq:oracle-fisher-asymp}) that statistical inference on the skill parameter of each player only depends on the player's close opponents. In other words, for each $\theta_i^*$, the information is available in the games on the local graph
\begin{equation}
\mathcal{A}_i=\left\{A_{jk}: |r_j^*-r_i^*|\leq \frac{M}{\beta}, |r_k^*-r_i^*|\leq \frac{M}{\beta}\right\}.\label{eq:nei-i}
\end{equation}
All the other games have little information in the statistical inference of $\theta_i^*$. Therefore, it is required that the local graph $\mathcal{A}_i$ is connected. The condition $\frac{p}{\beta\log n}\rightarrow\infty$ guarantees the connectivity of $\mathcal{A}_i$ for all $i\in[n]$. Note that the size of the local graph is $O(\beta^{-1})$, which again justifies that the effective sample size of the BTL model is $p/\beta$ instead of $pn$ in the Gaussian case. Since the local graph $\mathcal{A}_i$ is unknown, 
the additional $\frac{L}{\log n}\rightarrow\infty$ assumption is needed in the upper bound to estimate it or its surrogate.

\section{A Divide-and-Conquer Algorithm}\label{sec:dnc}

We introduce a fully adaptive and computationally efficient algorithm for ranking under the BTL model in this section. We first outline the main idea in Section \ref{sec:overview}. Details of the algorithm are presented in Section \ref{sec:BTL-algo}, and the statistical properties are analyzed in Section \ref{sec:analysis}.

\subsection{An Overview}\label{sec:overview}

In the Gaussian comparison model, we first compute the global MLE for the skill parameters via the least-squares optimization (\ref{eq:Gaussian-ls}), and then rank the players according to the estimators of the skills. This simple idea does not generalize to the BTL model, since the statistical information of each player concentrates on its close opponents, a phenomenon that is discussed in Section \ref{sec:intuition}. Therefore, instead of using the global MLE, we should maximize likelihood functions that are only defined by players whose abilities are close. This modification not only addresses the information-theoretic issue of the BTL model that we just mentioned, but it also leads to Hessian matrices that are well conditioned, a property that is critical for efficient convex optimization.

For Player $i$, the set of close opponents that are sufficient for optimal statistical inference is given by $\mathcal{A}_i$ defined in (\ref{eq:nei-i}). Suppose the knowledge of $\mathcal{A}_i$ was available, we could compute the local MLE using games only against players in $\mathcal{A}_i$. This idea is roughly correct, but there are several nontrivial issues that we need to solve before making it actually work. The first issue lies in the identifiability of the BTL model that $\theta_i^*$ can only be estimated up to a translation, which makes the comparison between $\wh{\theta}_i$ obtained from $\mathcal{A}_i$ and $\wh{\theta}_j$ obtained from $\mathcal{A}_j$  meaningless. The second issue is that the set $\mathcal{A}_i$ is unknown, and we need a data-driven procedure to identify the close opponents of each player.

We propose an algorithm that first partitions the $n$ players into several leagues and then use local MLE to compare the skills of players within the same league. The league partition is data-driven, and serves as a surrogate for the local graphs $\mathcal{A}_i$'s. Moreover, for two players $i$ and $j$ in the same league, the MLEs of their skill parameters are computed using the same set of opponents, and thus $\wh{\theta}_i-\wh{\theta}_j$ is a well-defined estimator of $\theta_i^*-\theta_j^*$.

Another key idea we use in our proposed algorithm is that the estimation of $r^*$ is closely related to the estimation of the \emph{pairwise relation matrix} $R^*$ defined as
\begin{align}\label{eqn:pair_relation}
R^*_{ij} =\mathbb{I}\{r_i^* < r^*_j\}\quad \text{for all } 1\leq i\neq j\leq n.
\end{align}
For any estimator of $R^*$, it can be converted into an estimator of the rank vector $r^*$ according to Lemma \ref{lem:anderson-ineq}. As a result, we shall focus on constructing a good estimator for all the pairwise relations $\{\mathbb{I}\{r_i^* < r^*_j\}\}_{i<j}$.

This divide-and-conquer algorithm, which will be described in Section \ref{sec:BTL-algo}, resembles typical strategies adopted in professional sports such as European football leagues. It is computationally efficient and we will show the algorithm achieves the minimax rate of  full ranking.

\subsection{Details of The Proposed Algorithm}\label{sec:BTL-algo}

We first decompose the set $[L]$ by $\{1,\cdots,L_1\}$ and $\{L_1+1,\cdots,L\}$. Games in the first set are used as preliminary games for league partition, and games in the second set are used for computing the MLE. Under the condition $\frac{L}{\log n}\rightarrow\infty$, we can set the number $L_1$ as $L_1=\ceil{\sqrt{L\log n}}$. Define
$$\bar{y}_{ij}^{(1)}=\frac{1}{L_1}\sum_{l=1}^{L_1}y_{ijl}\quad\text{and}\quad\bar{y}_{ij}^{(2)}=\frac{1}{L-L_1}\sum_{l=L_1+1}^Ly_{ijl}$$
as the summary statistics in $\{1,\cdots,L_1\}$ and $\{L_1+1,\cdots,L\}$, respectively. 

The proposed algorithm consists of four steps, which we describe in detail below before presenting the whole procedure in Algorithm \ref{alg:whole}.

\paragraph{Step 1: League Partition.} For each $i\in[n]$, we define
\begin{equation}
w_i^{(1)}=\sum_{j\in[n]}A_{ij}\mathbb{I}\{\bar{y}_{ij}^{(1)}\leq\psi(-2M)\}, \label{eq:def-w1}
\end{equation}
where $M$ is some sufficiently large constant. The indicator $\mathbb{I}\{\bar{y}_{ij}^{(1)}\leq\psi(-2M)\}$ describes the event that Player $i$ is completely dominated by Player $j$ in the preliminary games. The quantity $w_i^{(1)}$ then counts the number of players who have dominated Player $i$. If $w_i^{(1)}$ is sufficiently small, Player $i$ should belong to the top league since only few or no players could dominate Player $i$. Indeed, the first league is defined by
\begin{equation}
S_1=\left\{i\in[n]: w_i^{(1)}\leq h\right\}, \label{eq:def-S1}
\end{equation}
where $h$ is chosen as $h=\frac{pM}{\beta}$. A data-driven $h$ will be described in the Section \ref{sec:h}. Similarly, $w_i^{(2)}$ and the second league $S_2$ can be defined by replacing $[n]$ with $[n]\backslash\{S_1\}$ in (\ref{eq:def-w1}) and (\ref{eq:def-S1}). Sequentially, we compute $w_i^{(k+1)}$ and $S_{k+1}$ based on players in $[n]\backslash\left(S_1\cup\cdots\cup S_k\right)$ for all $k\geq 1$.
This procedure will terminate as soon as the number of the players who are yet to be classified is small enough, at which point all of the remaining players will be grouped together into the last league. The entire procedure of league partition is described in Algorithm \ref{alg:partition}.
\begin{algorithm}
\DontPrintSemicolon
\SetKwInOut{Input}{Input}\SetKwInOut{Output}{Output}
\Input{$\{A_{ij}\bar y_{ij}^{(1)}\}_{1\leq i<j\leq n}$ and $\{A_{ij}\}_{1\leq i<j\leq n}$; $M$ and $h$}
\Output{A partition of $[n]$: $S_1,\cdots, S_K$ such that $[n]=\uplus_{k=1}^KS_k$} 
\nl For $i$ in $[n]$, compute $w_i^{(1)}\leftarrow\sum_{j\in[n]}A_{ij}\mathbb{I}\{\bar{y}_{ij}^{(1)}\leq\psi(-2M)\}$.\;
Set $S_1\leftarrow\left\{i\in[n]: w_i^{(1)}\leq h\right\}$ and $k=1$.\;

\nl While $n-\left(|S_1|+\cdots+|S_k|\right)>|S_k|/2$,\;
\qquad For each $i\in[n]\backslash\left(S_1\cup\cdots\cup S_k\right)$, \;
\qquad\qquad compute $w_i^{(k+1)}\leftarrow\sum_{j\in[n]\backslash\left(S_1\cup\cdots\cup S_k\right)}A_{ij}\mathbb{I}\{\bar{y}_{ij}^{(1)}\leq\psi(-2M)\}$. \;
\qquad Set $S_{k+1}\leftarrow\left\{i\in[n]\backslash\left(S_1\cup\cdots\cup S_k\right): w_i^{(k+1)}\leq h\right\}$ and $k\leftarrow k+1$. \;

\nl Set $K\leftarrow k-1$ and $S_K\leftarrow S_K \cup\left([n]\backslash\left(S_1\cup\cdots\cup S_{K-1}\right)\right)$.
\caption{A league partition algorithm}
\label{alg:partition}
\end{algorithm}

\paragraph{Step 2: Local MLEs and Within-League Pairwise Relation Estimation.} Having obtained the league partition $S_1,\cdots,S_K$, we need to compare players in the same league in the next step. Given the ambiguity between neighboring leagues, we shall also compare players if the leagues they belong to are next to each other. Therefore, for each $k\in[K-1]$, we need to compute the MLE for $\{\theta_{r_i^*}^*\}_{i\in S_k\cup S_{k+1}}$. This leads to the comparison between any two players in $S_k\cup S_{k+1}$.
Define
\begin{equation}
\mathcal{E}=\left\{(i,j):1\leq i<j\leq n, \psi(-M)\leq \bar{y}_{ij}^{(1)}\leq \psi(M)\right\}. \label{eq:close-edges}
\end{equation}
For each $k\in[K-1]$, the local negative log likelihood function is given by
\begin{equation}
\ell^{(k)}(\theta)=\sum_{\substack{(i,j)\in\mathcal{E}\\i,j\in S_{k-1}\cup S_k\cup S_{k+1}\cup S_{k+2}}}A_{ij}\left[\bar{y}_{ij}^{(2)}\log\frac{1}{\psi(\theta_i-\theta_j)}+(1-\bar{y}_{ij}^{(2)})\log\frac{1}{1-\psi(\theta_i-\theta_j)}\right].\label{eq:league-loss}
\end{equation}
When $k=1$ or $k=K-1$, we use the notation $S_0=S_{K+1}=\varnothing$.
Note that the negative log likelihood function is only defined for edges in $\mathcal{E}$. In other words, only games between close opponents are considered. Moreover, some of the top players in $S_k$ may have close opponents in the previous league $S_{k-1}$, and some of the bottom players in $S_{k+1}$ may have close opponents in the next league $S_{k+2}$. The likelihood should include these games as well for optimal inference of the parameters $\{\theta_{r_i^*}^*\}_{i\in S_k\cup S_{k+1}}$. The MLE is defined by
\begin{equation}
\wh{\theta}^{(k)}\in \argmin\ell^{(k)}(\theta), \label{eq:league-MLE}
\end{equation}
which is any vector that minimizes $\ell^{(k)}(\theta)$. Then, for any $i\in S_k$ and any $j\in S_k\cup S_{k+1}$, set
$$R_{ij}=\mathbb{I}\{\wh{\theta}_i^{(k)}>\wh{\theta}_j^{(k)}\}.$$
Note that $\{\wh{\theta}_i^{(k)}\}_{i\in S_k\cup S_{k+1}}$ is defined only up to a common translation, but even with such ambiguity, the comparison indicator $R_{ij}$ is uniquely defined.

We also remark that the computation of the MLE (\ref{eq:league-MLE}) is a straightforward convex optimization. It can be shown that the Hessian matrix of the objective function is well conditioned (Lemma \ref{lem:bound-hessian}), and thus a standard gradient descent algorithm converges to the optimum with a linear rate \citep{chen2019spectral,chen2020partial}.

\paragraph{Step 3: Cross-League Pairwise Relation Estimation.} Consider $i$ and $j$ that belong to $S_k$ and $S_l$ respectively with $|k-l|\geq 2$. This is a pair of players that are separated by at least an entire league between them. For all such pairs, we set
$$R_{ij}=\mathbb{I}\{k<l\}.$$
Combined with the entries that are computed in Step 2, all upper triangular entries of the matrix $R$ have been filled. The remaining entries of $R$ can be filled according to the rule $R_{ij}+R_{ji}=1$.

~\\
\indent  Step 2 and Step 3 together serve the purpose of estimating the pairwise relation matrix $R^*$ defined in (\ref{eqn:pair_relation}). Illustrated in Figure \ref{fig:relation}, the matrix $R^*$ can be decomposed into blocks $\{R^*_{S_k \times S_l}\}_{k<l}$ according to the league partition $\{S_k\}_{k\in[K]}$. The yellow blocks close to the diagonal are estimated by the procedure described in Step 2. In Figure \ref{fig:relation}, the data used in the two local MLEs ($k=1$ and $k=4$) are marked by different patterns for illustration. For example, when $k=4$, we obtain estimators for $R^*_{S_4 \times S_4}$ and $R^*_{S_4 \times S_5}$ based on the local MLE that involves observations from $\{(i,j)\in \mathcal{E}:i,j\in S_3\cup S_4 \cup S_5 \cup S_6\}$. The blue blocks are away from the diagonal and are estimated in Step 3. The remaining blocks in the lower triangular part are estimated according to $R_{ij}+R_{ji}=1$.

\begin{figure}[h]
	\centering
	\includegraphics[width=0.6\textwidth]{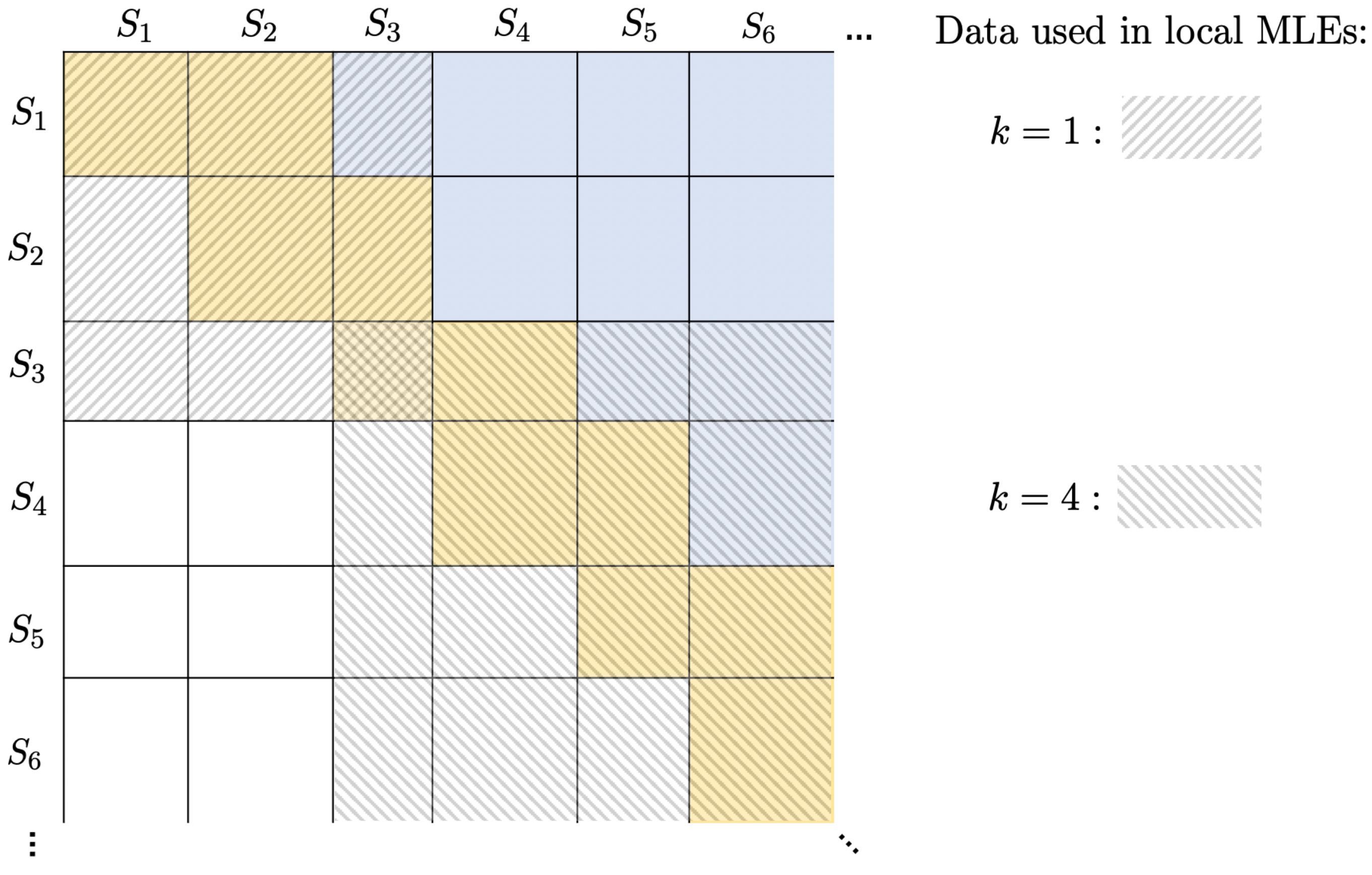}
	\caption{\textsl{Illustration of Step 2 and Step 3.}}
	\label{fig:relation}
\end{figure}

\paragraph{Step 4: Full Rank Estimation.} In the last step, we convert the pairwise relations estimator $R$ into a rank estimator. First, compute the score for the $i$th player by
$$s_i=\sum_{j\in[n]\backslash\{i\}}R_{ij}.$$
Then, the rank estimator $\wh{r}$ is obtained by sorting the scores $\{s_i\}_{i\in[n]}$. 

The whole procedure of  full ranking is summarized as Algorithm \ref{alg:whole}.
\begin{algorithm}
\DontPrintSemicolon
\SetKwInOut{Input}{Input}\SetKwInOut{Output}{Output}
\Input{$\{A_{ij}\bar y_{ij}^{(1)}\}_{1\leq i<j\leq n}$, $\{A_{ij}\bar y_{ij}^{(2)}\}_{1\leq i<j\leq n}$ and $\{A_{ij}\}_{1\leq i<j\leq n}$; $M$ and $h$}
\Output{A rank vector $\wh{r}\in\S_n$} 
\nl Run Algorithm \ref{alg:partition} and obtain the partition $[n]=\uplus_{k=1}^K S_k$.\;
Set $S_0=S_{K+1}=\varnothing$. \;
\nl For $k\in[K-1]$, \;
\qquad compute the local MLE $\wh{\theta}^{(k)}$ according to (\ref{eq:league-MLE}). \;
\qquad For $i\in S_k$ and $j\in S_k\cup S_{k+1}$, \; 
\qquad\qquad set $R_{ij}\leftarrow\mathbb{I}\{\wh{\theta}_i^{(k)}>\wh{\theta}_j^{(k)}\}$. \;
\nl For $k\in[K-2]$ and $l\in[k+2:K]$, \;
\qquad For $(i,j)\in S_k\times S_l$, \;
\qquad\qquad set $R_{ij}\leftarrow 1$. \;
For $i\in[n]$ and $j\in[i+1:n]$, \;
\qquad set $R_{ji}\leftarrow 1-R_{ij}$. \;
\nl For $i\in[n]$, \;
\qquad compute $s_i\leftarrow \sum_{j\in[n]\backslash\{i\}}R_{ij}$. \;
Sort $\{s_i\}_{i\in[n]}$ from high to low and obtain a full rank vector $\wh{r}$.
\caption{A divide-and-conquer full ranking algorithm}
\label{alg:whole}
\end{algorithm}

\subsection{Statistical Properties of Each Step}\label{sec:analysis}

The purpose of this section is to prove the upper bound result of Theorem \ref{thm:BTL-minimax} by analyzing the statistical properties of Algorithm \ref{alg:whole}. The four components of the  algorithm will be analyzed separately. We will first analyze Step 4 in Section \ref{sec:analysis_step4}, then Step 1 in Section \ref{sec:analysis_step1},  followed by Step 3 in Section \ref{sec:analysis_step3} and finally Step 2 in Section \ref{sec:analysis_step2}. 
The results of these individual components will be combined to derive the minimax optimality of Algorithm \ref{alg:whole}, presented in Section \ref{sec:whole}.

\subsubsection{From Pairwise Relations to Full Ranking (Step 4).} \label{sec:analysis_step4}
We first establish a result that clarifies the role of Step 4 of Algorithm \ref{alg:whole}. Consider any matrix $R\in\{0,1\}^{n\times n}$ that satisfies $R_{ij}+R_{ji}=1$ for any $i\neq j$. Let $\wh{r}$ be the rank vector obtained by sorting $\{\sum_{j\in[n]\backslash\{i\}}R_{ij}\}_{i\in[n]}$ from high to low. The error of $\wh{r}$ is controlled by the following lemma.

\begin{lemma}\label{lem:anderson-ineq}
For any $r^*\in\S_n$, define its pairwise relation matrix $R^*$  such that $R_{ij}^*=\mathbb{I}\{r_i^*<r_j^*\}$. Then, we have
$$\k(\wh{r},r^*)\leq \frac{4}{n}\sum_{1\leq i\neq j\leq n}\mathbb{I}\{R_{ij}\neq R_{ij}^*\}.$$
\end{lemma}

Lemma \ref{lem:anderson-ineq} is a deterministic inequality that bounds the error of the rank estimation by the estimation error of pairwise relations. It implies that to accurately rank $n$ players, it is sufficient to accurately estimate the pairwise relations between all pairs. 

\subsubsection{Statistical Properties of League Partition (Step 1).} \label{sec:analysis_step1}
The partition output by Algorithm \ref{alg:partition} satisfies several nice properties that are stated by the following theorem.

\begin{thm}\label{thm:alg-1}
Assume $\theta^*\in\Theta_n(\beta,C_0)$ for some constant $C_0\geq 1$, $\frac{L}{\log n}\rightarrow\infty$ and $\frac{p}{(\beta\vee n^{-1})\log n}\rightarrow\infty$. Let $\{S_k\}_{k\in[K]}$ be the output of Algorithm \ref{alg:partition} with $L_1=\ceil{\sqrt{L\log n}}$, $1\leq M=O(1)$ and $h=\frac{pM}{\beta}$. Then, there exist some constants $C_1,C_2,C_3>0$ only depending on $C_0$ such that the following conclusions hold with probability at least $1-O(n^{-9})$:
\begin{enumerate}
\item \emph{Boundedness:} For any $k\in[K]$ and any $i,j\in S_{k-1}\cup S_k\cup S_{k+1}$, we have $|\theta_{r_i^*}^*-\theta_{r_j^*}^*|\leq C_1M$. Recall the convention that $S_{0}=S_{K+1}=\varnothing$;
\item \emph{Inclusiveness:} For any $k\in[K]$ and any $i\in S_k$, we have $\left\{j\in[n]: |r_i^*-r_j^*|\leq \frac{C_2M}{\beta}\right\}\subset S_{k-1}\cup S_k\cup S_{k+1}$;
\item \emph{Separation:} For any $i\in S_k$ and $j\in S_{l}$ such that $l-k\geq 2$, we have $\theta^*_{r_i^*} > \theta^*_{r_j^*}$; 
\item \emph{Independence:} For any $k\in[K]$, we have $S_k=\check{S}_k$. Here, $\{\check{S}_k\}_{k\in[K]}$ is a partition that is measurable with respect to the $\sigma$-algebra generated by $\{(A_{ij},\bar{y}_{ij}^{(1)}): |\theta_{r_i^*}^*-\theta_{r_j^*}^*|>1.9M\}$;
\item \emph{Continuity: } For any $k\in[K-1]$ and any $i\in S_{k-1}\cup S_{k} \cup S_{k+1}\cup S_{k+2}$, we have $\abs{\left\{j\in[n]: |\theta^*_{r_i^*}-\theta^*_{r_j^*}|\leq \frac{M}{2}\right\} \cap (S_{k-1}\cup S_{k} \cup S_{k+1}\cup S_{k+2})} \geq C_3 \br{\frac{ M}{\beta}\wedge n}$. 
\end{enumerate}
\end{thm}

We give some remarks on each conclusion of Theorem \ref{thm:alg-1}. The first conclusion asserts that the skill parameters of players from the neighboring leagues are close to each other. This property is complemented by the second conclusion that the close opponents of each player are either from the same league, the previous league, or the next league. In other words, for any $k\in[K]$ and any $i\in S_k$, the local graph $\{A_{jk}: j,k\in S_{k-1}\cup S_k\cup S_{k+1}\}$ can be viewed as a data-driven surrogate of $\mathcal{A}_i$ defined in (\ref{eq:nei-i}). Moreover, the second conclusion also implies that $|S_{k-1}\cup S_k\cup S_{k+1}|\gtrsim \frac{1}{\beta}\wedge n$, from which we can deduce the bound $K=O\left(n\beta\vee 1\right)$ that controls the number of iterations Algorithm \ref{alg:partition} needs before it is terminated.\footnote{We can in fact prove a stronger result that $\frac{1}{\beta}\wedge n\lesssim |S_k|\lesssim \frac{1}{\beta}\wedge n$ uniformly for all $k\in[K]$ with probability at least $1-O(n^{-9})$.} 
Conclusion 3 implies that the partition $\{S_k\}_{k\in[K]}$ is roughly correlated with the true rank in the sense that it correctly identifies the comparisons between players who do not belong to neighboring leagues.
Conclusion 4 shows that almost all of the randomness of the partition is from that of $\{(A_{ij},\bar{y}_{ij}^{(1)}): \theta_{r_i^*}^*-\theta_{r_j^*}^*\leq-1.9M\}$. This fact leads to a crucial independence property in the later analysis of the local MLE. Conclusions 1, 2, 4, and 5 are crucial in the analysis of Step 2 in Section \ref{sec:analysis_step2}, while Conclusion 3 will be used in the analysis of Step 3 in Section \ref{sec:analysis_step3}. 

The proof of Theorem \ref{thm:alg-1} is a delicate mathematical induction argument that iteratively explores the asymptotic independence between consecutive constructions of leagues. To be specific, the random variable
$$w_i^{(k+1)}=\sum_{j\in[n]\backslash\left(S_1\cup\cdots\cup S_k\right)}A_{ij}\mathbb{I}\{\bar{y}_{ij}^{(1)}\leq\psi(-2M)\}$$
can be sandwiched between $\underline{w}_i^{(k+1)}$ and $\overline{w}_i^{(k+1)}$. We show that both $\underline{w}_i^{(k+1)}$ and $\overline{w}_i^{(k+1)}$, when conditioning on the previous leagues $S_1,\cdots,S_k$, approximately follow Binomial distributions. Essentially, the $A_{ij}$'s that contribute to the summation of $w_i^{(k+1)}$ are disjoint from the $A_{ij}$'s that lead to the constructions of $S_1,\cdots,S_k$, which then implies an asymptotic independence property between $\big(\underline{w}_i^{(k+1)},\overline{w}_i^{(k+1)}\big)$ and $S_1,\cdots,S_k$.
\begin{figure}[h]
	\centering
	\includegraphics[width=0.6\textwidth]{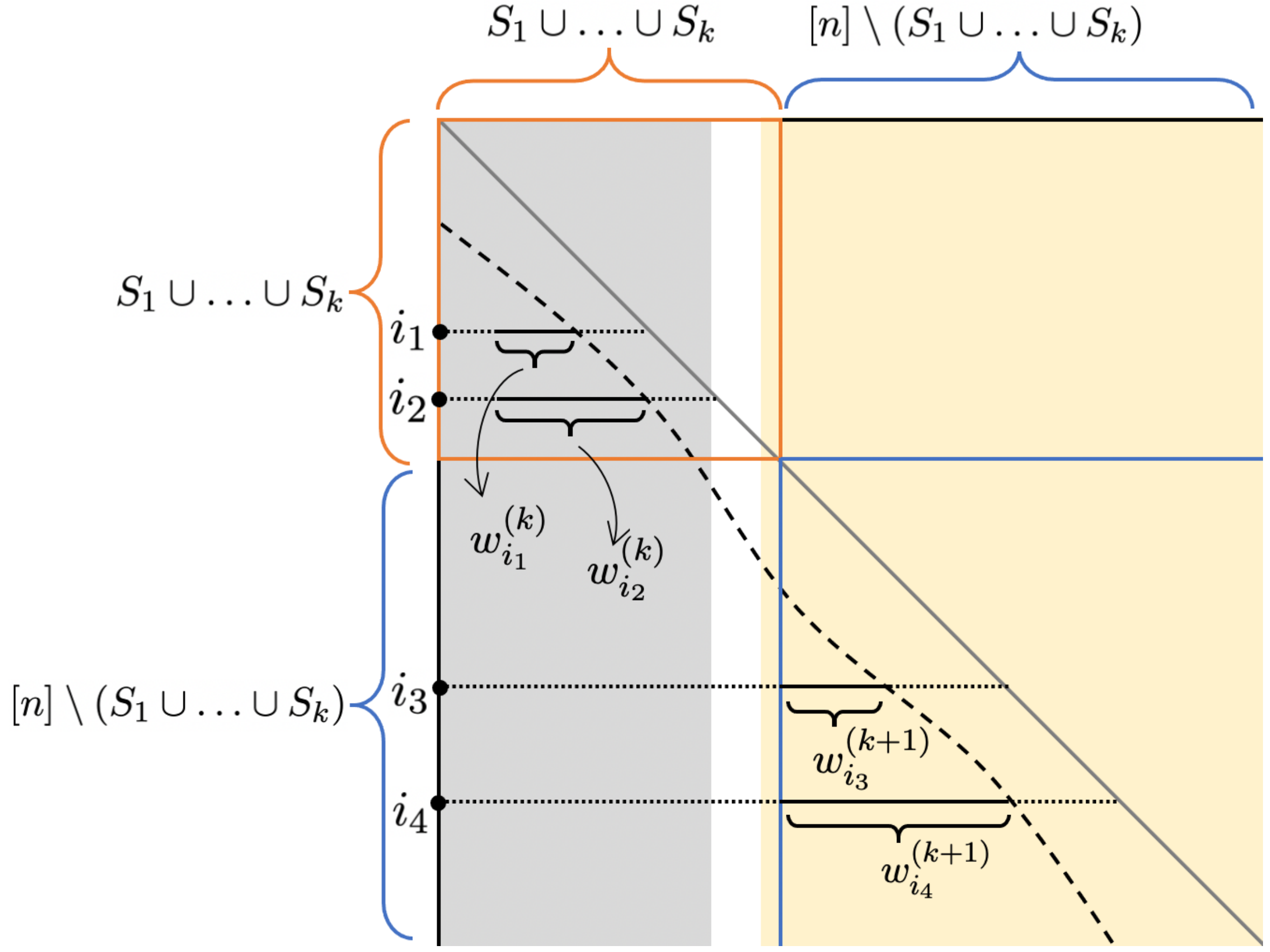}
	\caption{\textsl{Illustration of the independence property of Algorithm \ref{alg:partition}.}}
	\label{fig:thm41}
\end{figure}
This phenomenon is illustrated in Figure \ref{fig:thm41}. In the picture, we use the orange block to denote $S_1\cup\cdots\cup S_k$, the set that has already been partitioned. The next step of the algorithm is to construct the $(k+1)$th league from $[n]\backslash (S_1\cup\cdots\cup S_k)$, which is the blue block. From the positions of $w_i^{(k+1)}$'s, we observe that the construction of $S_{k+1}$ depends on $A_{ij}$'s that are in the yellow area. On the other hand, since the area on the left hand side of the dashed curve satisfies $\bar{y}_{ij}^{(1)}\leq\psi(-2M)$, the construction of the first $k$ leagues only depends on $A_{ij}$'s that are in the grey area. The independence property can be easily seen from the separation between the grey and the yellow areas. A rigorous proof of Theorem \ref{thm:alg-1}, which is based on this argument, will be given in Section \ref{sec:pf-alg-1}.

\subsubsection{Statistical Properties of Cross-League Estimation (Step 3).}  \label{sec:analysis_step3}
 The analysis of Step 3 is quite straightforward following the results from the league partition. Assume the Conclusion 3 of Theorem \ref{thm:alg-1} holds. Then for any $i\in S_k$ and $j\in S_{l}$ such that $l-k\geq 2$, we have $R^*_{ij}=1$. Since $R_{ij}=1$ for all such pairs, we have $\sum_{k\in[K-2]}\sum_{ l\in[k+2:K]} \indc{R_{ij} \neq R^*_{ij}, i\in S_k, j\in S_l}=0$.

\subsubsection{Statistical Properties of Local MLEs (Step 2).} \label{sec:analysis_step2}

The main challenge of analyzing the local MLE is the dependence between the partition $\{S_k\}_{k\in[K]}$ and the likelihood (\ref{eq:league-loss}).
We are going to use Conclusion 4 of Theorem \ref{thm:alg-1} to resolve this issue. Define
$$\check{A}_{ij}=A_{ij}\mathbb{I}\{|\theta_{r_i^*}^*-\theta_{r_j^*}^*|\leq M/2\}+A_{ij}\mathbb{I}\left\{(i,j)\in\mathcal{E}, M/2<|\theta_{r_i^*}^*-\theta_{r_j^*}^*|< 1.1 M\right\},$$
and
$$\check{\ell}^{(k)}(\theta)=\sum_{i,j\in \check{S}_{k-1}\cup \check{S}_k\cup \check{S}_{k+1}\cup \check{S}_{k+2}}\check{A}_{ij}\left[\bar{y}_{ij}^{(2)}\log\frac{1}{\psi(\theta_i-\theta_j)}+(1-\bar{y}_{ij}^{(2)})\frac{1}{1-\psi(\theta_i-\theta_j)}\right].$$
The maximizer of $\check{\ell}^{(k)}(\theta)$ is denoted by
\begin{equation}
\check{\theta}^{(k)}\in \argmin\check{\ell}^{(k)}(\theta).\label{eq:MLE-check}
\end{equation}
The introduction of $\check{\ell}^{(k)}(\theta)$ and $\check{\theta}^{(k)}$ is to disentangle the dependence of the MLE on the league partition. By Theorem \ref{thm:alg-1}, we know that $S_k=\check{S}_k$ for all $k\in[K]$. The concentration of $\{\bar{y}_{ij}^{(1)}\}$ implies that $\{|\theta_{r_i^*}^*-\theta_{r_j^*}^*|\leq M/2\}\subset\{(i,j)\in\mathcal{E}\}\subset\{|\theta_{r_i^*}^*-\theta_{r_j^*}^*|\leq 1.1M\}$ for all $1\leq i<j\leq n$. Therefore, we have
$$\mathbb{I}\left\{(i,j)\in\mathcal{E}\right\}=\mathbb{I}\{|\theta_{r_i^*}^*-\theta_{r_j^*}^*|\leq M/2\}+\mathbb{I}\left\{(i,j)\in\mathcal{E}, M/2<|\theta_{r_i^*}^*-\theta_{r_j^*}^*|< 1.1 M\right\}.$$
We can thus conclude that $\ell^{(k)}(\theta)=\check{\ell}^{(k)}(\theta)$ for all $\theta$ with high probability. The result is formally stated below.

\begin{lemma}\label{lem:MLE-check}
Assume $\theta^*\in\Theta_n(\beta,C_0)$ for some constant $C_0\geq 1$, $\frac{L}{\log n}\rightarrow\infty$ and $\frac{p}{(\beta\vee n^{-1})\log n}\rightarrow\infty$. Let $\{S_k\}_{k\in[K]}$ be the output of Algorithm \ref{alg:partition} with $L_1=\ceil{\sqrt{L\log n}}$, $1\leq M=O(1)$ and $h=\frac{pM}{\beta}$. Then, with probability at least $1-O(n^{-8})$, we have $\ell^{(k)}(\theta)=\check{\ell}^{(k)}(\theta)$ for all $\theta$ and for all $k\in[K]$. As a consequence $\{\wh{\theta}_i^{(k)}\}_{i\in S_k\cup S_{k+1}}$ and $\{\check{\theta}_i^{(k)}\}_{i\in S_k\cup S_{k+1}}$ are equivalent up to a common shift.
\end{lemma}

With Lemma \ref{lem:MLE-check}, it suffices to study (\ref{eq:MLE-check}) for the statistical property of the MLE. Note that $\{\check{A}_{ij}\}$ is measurable with respect to the $\sigma$-algebra generated by $\{(A_{ij},\bar{y}_{ij}^{(1)}): |\theta_{r_i^*}^*-\theta_{r_j^*}^*|<1.1M\}$. Theorem \ref{thm:alg-1} shows that $\{\check{S}_k\}$ is measurable with respect to the $\sigma$-algebra generated by $\{(A_{ij},\bar{y}_{ij}^{(1)}): |\theta_{r_i^*}^*-\theta_{r_j^*}^*|>1.9M\}$. We then reach a very important conclusion that $\{\check{A}_{ij}\}$, $\{\bar{y}_{ij}^{(2)}\}$ and $\{\check{S}_k\}$ are mutually independent, and therefore we can analyze $\check{\theta}^{(k)}$ by conditioning on the partition $\{\check{S}_k\}$. To be more specific, for any $i,j\in \check{S}_k\cup\check{S}_{k+1}$ such that $\theta_{r_i^*}^*>\theta_{r_j^*}^*$, since $R_{ij} = \indc{\check{\theta}_i^{(k)}>\check{\theta}_j^{(k)}} $,
we will provide an upper bound for $\mathbb{P}\left(\check{\theta}_i^{(k)}<\check{\theta}_j^{(k)}\Big|\{\check{S}_k\}_{k\in[K]}\right)$.

To this end, we state a result that characterizes the performance of the MLE under a BTL model with bounded skill parameters. Consider a random graph with independent edges $B_{ij}\sim\text{Bernoulli}(p_{ij})$ for $1\leq i<j\leq m$. For each $B_{ij}=1$, observe i.i.d. $y_{ijl}\sim\text{Bernoulli}(\psi(\eta_i^*-\eta_j^*))$ for $l=1,\cdots,L$. Let $\bar{y}_{ij}=\frac{1}{L}\sum_{l=1}^Ly_{ijl}$, and we define the MLE by
\begin{equation}
\wh{\eta}\in\argmin\sum_{1\leq i<j\leq m}B_{ij}\left[\bar{y}_{ij}\log\frac{1}{\psi(\eta_i-\eta_j)}+(1-\bar{y}_{ij})\log\frac{1}{1-\psi(\eta_i-\eta_j)}\right]. \label{eq:MLE-small}
\end{equation}
\begin{lemma}\label{lem:prev-paper}
Assume $\eta_1^*>\cdots>\eta_m^*$ and $\eta_1^*-\eta_m^*\leq\kappa$. There exists some constant $c\in(0,1)$ such that $p_{ij}=p$ for all $|i-j|\leq cm$ and $p_{ij}\leq p$ otherwise. As long as $\frac{mp}{\log(m+n)}\rightarrow\infty$ and $\kappa=O(1)$, then for any $\delta>0$ that is sufficiently small, there exists a constant $C>0$ such that
$$\mathbb{P}\left(\wh{\eta}_i<\wh{\eta}_j\right)\leq C\left[\exp\left(-\frac{(1-\delta)L(\eta_i^*-\eta_j^*)^2}{2(W_{i}(\eta^*)+W_j(\eta^*))}\right)+n^{-7}\right],$$
for all $1\leq i<j\leq m$, where $W_i(\eta^*)=\frac{1}{\sum_{j\in[m]\backslash\{i\}}p_{ij}\psi'(\eta_i^*-\eta_j^*)}$ for all $i\in[m]$.
\end{lemma}

The proof of Lemma \ref{lem:prev-paper}, which relies on a recently developed leave-one-out technique in the analysis of the BTL model \citep{chen2019spectral,chen2020partial}, will be given in Section \ref{sec:pf-lemmas}.

By conditioning on $\{\check{S}_k\}$, the statistical property of (\ref{eq:MLE-check}) is a direct consequence of Lemma \ref{lem:prev-paper}.
Note that $\mathbb{P}\left(\check{\theta}_i^{(k)}<\check{\theta}_j^{(k)}\Big|\{\check{S}_k\}_{k\in[K]}\right)$ is a function of $\{\check{S}_k\}_{k\in[K]}$, and we will establish a uniform upper bound for this conditional probability for any partition $\{\check{S}_k\}_{k\in[K]}$ satisfying the following conditions:
\begin{enumerate}
\item[(i)] For any $k\in[K]$ and any $i,j\in \check{S}_{k-1}\cup \check{S}_k\cup \check{S}_{k+1}$, we have $|\theta_{r_i^*}^*-\theta_{r_j^*}^*|\leq C_1M$; 
\item[(ii)] For any $k\in[K]$ and any $i\in \check{S}_k$, we have $\left\{j\in[n]: |r_i^*-r_j^*|\leq \frac{C_2M}{\beta}\right\}\subset \check S_{k-1}\cup \check S_k\cup \check S_{k+1}$;
\item[(iii)] For any $k\in[K-1]$ and any $i\in  \check{S}_{k-1}\cup \check{S}_k\cup \check{S}_{k+1}\cup \check{S}_{k+2}$, we have $\Big|\left\{j\in[n]: |\theta^*_{r_i^*}-\theta^*_{r_j^*}|\leq \frac{M}{2}\right\} \cap ( \check{S}_{k-1}\cup \check{S}_k\cup \check{S}_{k+1}\cup \check{S}_{k+2})\Big| \geq C_3 \br{\frac{ M}{\beta}\wedge n}$.
\end{enumerate}
Note that we use the convention $\check{S}_0=\check{S}_{K+1}=\varnothing$ and $C_1,C_2,C_3$ are the same constants in Theorem \ref{thm:alg-1}.
Consider any partition $\{\check{S}_k\}_{k\in[K]}$ satisfying the three conditions above.  When applying Lemma \ref{lem:prev-paper}, by Conditions (i) and (ii), we have $\kappa=2C_1M$ and $m=|\check{S}_{k-1}\cup \check{S}_k\cup \check{S}_{k+1}\cup \check{S}_{k+2}|\asymp \frac{1}{\beta}\wedge n$.  We also know that for any $i,j \in  \check{S}_{k-1}\cup \check{S}_k\cup \check{S}_{k+1}\cup \check{S}_{k+2}$ such that $|\theta^*_{r_i^*}-\theta^*_{r_j^*}|\leq \frac{M}{2}$, we have $\check{A}_{ij}=A_{ij}\sim\text{Bernoulli}(p)$. Then, Condition (iii) implies the existence of a band in $\{(r_i^*,r_j^*):i,j\in  \check{S}_{k-1}\cup \check{S}_k\cup \check{S}_{k+1}\cup \check{S}_{k+2}\}$ with width at least $cm$ for some constant $c>0$, such that $\check{A}_{ij}\sim\text{Bernoulli}(p)$ for all pairs in the band. For any other $(i,j)$, we have $\check{A}_{ij}\sim\text{Bernoulli}(p_{ij})$ with $p_{ij}\leq p$.
%
%
Having checked the conditions of Lemma \ref{lem:prev-paper}, we obtain the following result for the local MLE (\ref{eq:MLE-check}),
\begin{equation}
\mathbb{P}\left(\check{\theta}_i^{(k)}<\check{\theta}_j^{(k)}\Big|\{\check{S}_k\}_{k\in[K]}\right)\leq C\left[\exp\left(-\frac{(1-\delta)npL(\theta_{r_i^*}^*-\theta_{r_j^*}^*)^2}{2(V_{r_i^*}(\theta^*)+V_{r_j^*}(\theta^*))}\right)+n^{-7}\right], \label{eq:conditional-comp}
\end{equation}
for any $i,j\in \check{S}_k\cup\check{S}_{k+1}$ such that $\theta_{r_i^*}^*>\theta_{r_j^*}^*$. Recall the definition of $V_i(\theta^*)$ in (\ref{eq:BTL-var}). The constant $\delta$ in (\ref{eq:conditional-comp}) can be made arbitrarily small with a sufficiently large $M$. To derive (\ref{eq:conditional-comp}) from Lemma \ref{lem:prev-paper}, we only need to show
$$p\sum_{j\in[n]\backslash\{i\}}\psi'(\theta_i^*-\theta_j^*)\leq \left(1+O(e^{-C_2M})\right)\sum_{j\in(\check{S}_{k-1}\cup \check{S}_k\cup \check{S}_{k+1}\cup \check{S}_{k+2})\backslash\{i\}}p_{ij}\psi'(\theta_{r_i^*}^*-\theta_{r_j^*}^*),$$
for all $i\in \check{S}_k\cup\check{S}_{k+1}$.
This is true by a similar argument that leads to (\ref{eq:oracle-fisher-asymp}), together with Condition (ii).
Finally, by Theorem \ref{thm:alg-1}, Conditions (i)-(iii) hold for $\{\check{S}_k\}_{k\in[K]}$ with high probability, and thus (\ref{eq:conditional-comp}) is a high-probability bound.
A similar bound to (\ref{eq:conditional-comp}) also holds for (\ref{eq:league-MLE}) by the conclusion of Lemma \ref{lem:MLE-check}.


\subsection{Analysis of Algorithm \ref{alg:whole}} \label{sec:whole} 
With the help of Lemma \ref{lem:anderson-ineq}, Theorem \ref{thm:alg-1}, Lemma \ref{lem:MLE-check} and Lemma \ref{lem:prev-paper}, we are ready to prove that Algorithm \ref{alg:whole} achieves the minimax rate of full ranking.

\begin{proof}[Proof of Theorem \ref{thm:BTL-minimax} (upper bound)]
Let $\mathcal{G}$ be the event that the conclusions of Theorem \ref{thm:alg-1} and Lemma \ref{lem:MLE-check} hold. We have $\mathbb{P}(\mathcal{G}^c)=O(n^{-8})$. In addition, we use the notation $\check{\mathcal{S}}$ for the event that $\{\check{S}_k\}_{k\in[K]}$ satisfies Conditions (i)-(iii) listed in Section \ref{sec:analysis_step2}.  It is clear that $\mathcal{G}\subset \check{\mathcal{S}}$.

 
 By Lemma \ref{lem:anderson-ineq}, we have
$$\mathbb{E}\k(\wh{r},r^*)\leq \frac{4}{n}\sum_{1\leq i\neq j\leq n}\mathbb{P}(R_{ij}\neq R_{ij}^*).$$
It suffices to give a bound for $\mathbb{P}(R_{ij}\neq R_{ij}^*)$ for every pair $i\neq j$.
Note that we have $\mathbb{P}(R_{ij}\neq R_{ij}^*) \leq \mathbb{P}\left(R_{ij}\neq R_{ij}^*, \mathcal{G}\right) + \mathbb{P}(\mathcal{G}^c)$. Then.
\begin{align*}
\mathbb{P}(R_{ij}\neq R_{ij}^*,\mathcal{G}) & =  \sum_{k=1}^K\sum_{l=1}^K\mathbb{P}(R_{ij}\neq R_{ij}^*,\mathcal{G}, i \in S_k, j \in S_l)\\
& = \sum_{(k,l)\in[K]^2: \abs{k-l}\leq 1} \mathbb{P}(R_{ij}\neq R_{ij}^*,\mathcal{G}, i \in S_k, j \in S_l)\\
&\quad  +  \sum_{(k,l)\in[K]^2: \abs{k-l}\geq  2} \mathbb{P}(R_{ij}\neq R_{ij}^*,\mathcal{G}, i \in S_k, j \in S_l).
\end{align*}
The second term above is zero. This is due to the analysis of Step 3 in Section \ref{sec:analysis_step3} which shows $\sum_{(k,l)\in[K]^2: \abs{k-l}\geq  2} \mathbb{I}\{R_{ij}\neq R_{ij}^*, i \in S_k, j \in S_l\}=0$ under the event $\mathcal{G}$. Hence, we only need to study the first term.
Without loss of generality, consider $\theta_{r_i^*}>\theta_{r_j^*}^*$. Then, the event $\{R_{ij}\neq R_{ij}^*, \mathcal{G},i\in S_k,j\in S_{k}\}$ is equivalent to $\{\wh{\theta}^{(k)}_i<\wh{\theta}^{(k)}_j, \mathcal{G},i\in S_k,j\in S_{k}\}$, which is further equivalent to $\{\check{\theta}^{(k)}_i<\check{\theta}^{(k)}_j, \mathcal{G},i\in \check{S}_k,j\in \check{S}_{k}\}$ by the definition of $\mathcal{G}$. We thus have
\begin{align*}
\mathbb{P}(R_{ij}\neq R_{ij}^*,\mathcal{G})
& = \sum_{(k,l)\in[K]^2: \abs{k-l}\leq 1} \mathbb{P}(\check{\theta}^{(k)}_i<\check{\theta}^{(k)}_j, \mathcal{G},i\in \check{S}_k,j\in \check{S}_{l}) \\
& \leq   \sum_{(k,l)\in[K]^2: \abs{k-l}\leq 1} \mathbb{P}(\check{\theta}^{(k)}_i<\check{\theta}^{(k)}_j, \check{\mathcal{S}},i\in \check{S}_k,j\in \check{S}_{l}) \\
& =  \sum_{(k,l)\in[K]^2: \abs{k-l}\leq 1}  \mathbb{P}\left(\check{\theta}^{(k)}_i<\check{\theta}^{(k)}_j\Big|\check{\mathcal{S}},i\in \check{S}_k,j\in \check{S}_{l}\right)\mathbb{P}\left(\check{\mathcal{S}},i\in \check{S}_k,j\in \check{S}_{l}\right)\\
& \leq C\left[\exp\left(-\frac{(1-\delta)npL(\theta_{r_i^*}^*-\theta_{r_j^*}^*)^2}{2(V_{r_i^*}(\theta^*)+V_{r_j^*}(\theta^*))}\right)+n^{-7}\right] \sum_{(k,l)\in[K]^2: \abs{k-l}\leq 1} \mathbb{P}\left(\check{\mathcal{S}}, i\in \check{S}_k,j\in \check{S}_{l}\right) \\
&\leq C\left[\exp\left(-\frac{(1-\delta)npL(\theta_{r_i^*}^*-\theta_{r_j^*}^*)^2}{2(V_{r_i^*}(\theta^*)+V_{r_j^*}(\theta^*))}\right)+n^{-7}\right],
\end{align*}
for some constant $C>0$ and some $\delta>0$ that is arbitrarily small. The second last inequality above is by Lemma \ref{lem:prev-paper}, or more specifically, (\ref{eq:conditional-comp}), as we show (\ref{eq:conditional-comp}) holds for any $\{\check{S}_k\}_{k\in[K]}$ satisfying Conditions (i)-(iii) listed in Section \ref{sec:analysis_step2}.
Since $\mathbb{P}(\mathcal{G}^c)=O(n^{-8})$, we obtain the bound
\begin{equation}
\mathbb{P}(R_{ij}\neq R_{ij}^*) \leq 2C\left[\exp\left(-\frac{(1-\delta)npL(\theta_{r_i^*}^*-\theta_{r_j^*}^*)^2}{2(V_{r_i^*}(\theta^*)+V_{r_j^*}(\theta^*))}\right)+n^{-7}\right],\label{eq:trueforallpairs}
\end{equation}
for all $i\neq j$.

Summing the bound (\ref{eq:trueforallpairs}) over all $i\neq j$, we have
\begin{eqnarray}
\nonumber \mathbb{E}\k(\wh{r},r^*) &\leq& \frac{8C}{n}\sum_{1\leq i\neq j\leq n}\exp\left(-\frac{(1-\delta)npL(\theta_{r_i^*}^*-\theta_{r_j^*}^*)^2}{2(V_{r_i^*}(\theta^*)+V_{r_j^*}(\theta^*))}\right) + 8Cn^{-6} \\
\label{eq:sum-exp} &=& \frac{8C}{n}\sum_{1\leq i\neq j\leq n}\exp\left(-\frac{(1-\delta)npL(\theta_{i}^*-\theta_{j}^*)^2}{2(V_{i}(\theta^*)+V_{j}(\theta^*))}\right) + 8Cn^{-6}.
\end{eqnarray}
Now it is just a matter of simplifying the expression (\ref{eq:sum-exp}). We consider the following two cases: $\frac{Lp\beta^2}{\beta\vee n^{-1}}\leq 1$ and $\frac{Lp\beta^2}{\beta\vee n^{-1}}> 1$.

First, we consider the case $\frac{Lp\beta^2}{\beta\vee n^{-1}}\leq 1$.
By Lemma \ref{lem:sum-psi-prime} proved in Section \ref{sec:pf-BTL}, there exist constants $c_1,c_2>0$, such that
\begin{equation}
c_1\left(\beta\vee \frac{1}{n}\right)\leq \frac{V_{i}(\theta^*)}{n}\leq c_2\left(\beta\vee \frac{1}{n}\right),\label{eq:order-of-variance}
\end{equation}
for all $\theta^*\in\Theta_n(\beta,C_0)$ and all $i\in[n]$. Then, for each $i\in[n]$,
\begin{eqnarray*}
\sum_{j\in[n]\backslash\{i\}}\exp\left(-\frac{(1-\delta)npL(\theta_{i}^*-\theta_{j}^*)^2}{2(V_{i}(\theta^*)+V_{j}(\theta^*))}\right) &\leq& \sum_{j\in[n]\backslash\{i\}}\exp\left(-\frac{1}{3c_2}(i-j)^2\frac{Lp\beta^2}{\beta\vee n^{-1}}\right) \\
&\leq& \int_0^{\infty}\exp\left(-\frac{1}{3c_2}x^2\frac{Lp\beta^2}{\beta\vee n^{-1}}\right) dx \\
&=& \sqrt{\frac{3\pi c_2}{4}}\sqrt{\frac{\beta\vee n^{-1}}{Lp\beta^2}},
\end{eqnarray*}
and we have $\mathbb{E}\k(\wh{r},r^*)\lesssim \sqrt{\frac{\beta\vee n^{-1}}{Lp\beta^2}}$. The definition of the loss function implies $\mathbb{E}\k(\wh{r},r^*)\leq n$, and thus we obtain the rate $n\wedge\sqrt{\frac{\beta\vee n^{-1}}{Lp\beta^2}}$ when $\frac{Lp\beta^2}{\beta\vee n^{-1}}\leq 1$.

Next, we consider the case $\frac{Lp\beta^2}{\beta\vee n^{-1}}> 1$. 
For any $|i-j|\leq C_0\sqrt{c_2/c_1}$, we have $V_j(\theta^*)\leq (1+\delta')V_i(\theta^*)$ for some $\delta'=o(1)$. This is by the definition of the variance function and the fact that $\sup_x\left|\frac{\psi'(x+\Delta)}{\psi'(x)}-1\right|\lesssim |\Delta|$ for $\Delta=o(1)$. Therefore, we have
\begin{eqnarray*}
&& \sum_{1\leq i\neq j\leq n: |i-j|\leq C_0\sqrt{c_2/c_1}}\exp\left(-\frac{(1-\delta)npL(\theta_{i}^*-\theta_{j}^*)^2}{2(V_{i}(\theta^*)+V_{j}(\theta^*))}\right) \\
&\lesssim& \sum_{i=1}^{n-1}\exp\left(-\frac{(1-2\delta)npL(\theta_i^*-\theta_{i+1}^*)^2}{4V_i(\theta^*)}\right),
\end{eqnarray*}
By (\ref{eq:order-of-variance}), we also have
\begin{eqnarray*}
&& \sum_{1\leq i\neq j\leq n: |i-j|> C_0\sqrt{c_2/c_1}}\exp\left(-\frac{(1-2\delta)npL(\theta_{i}^*-\theta_{j}^*)^2}{2(V_{i}(\theta^*)+V_{j}(\theta^*))}\right) \\
&\lesssim& \sum_{1\leq i\neq j\leq n: |i-j|> C_0\sqrt{c_2/c_1}}\exp\left(-\frac{(1-2\delta)pL\beta^2(i-j)^2}{2c_2(\beta\vee n^{-1})}\right) \\
&\lesssim& n\exp\left(-\frac{(1-2\delta)pL\beta^2C_0^2}{2c_1(\beta\vee n^{-1})}\right) \\
&\lesssim& \sum_{i=1}^{n-1}\exp\left(-\frac{(1-2\delta)npL(\theta_i^*-\theta_{i+1}^*)^2}{4V_i(\theta^*)}\right).
\end{eqnarray*}
The desired bound for $\mathbb{E}\k(\wh{r},r^*)$ immediately follows by summing up the above bounds.
\end{proof}

\subsection{A Data-Driven $h$.} \label{sec:h}
Our proposed algorithm relies on a tuning parameter $h=\frac{pM}{\beta}$ that is unknown in practice. This quantity can be replaced by a data-driven version, defined as
\begin{equation}
\wh{h}=\frac{1}{n}\sum_{1\leq i<j\leq n}A_{ij}\mathbb{I}\{1.2 M\leq |\psi^{-1}(\bar{y}_{ij}^{(1)})| \leq 1.8M\}. \label{eq:global-h}
\end{equation}
A standard concentration result implies that $\wh{h}\asymp \frac{pM}{\beta}$ with high probability. Moreover, by defining
$$\check{h}=\frac{1}{n}\sum_{\substack{1\leq i<j\leq n\\1.1M<|\theta_{r_i^*}^*-\theta_{r_j^*}^*|<1.9M}}A_{ij}\mathbb{I}\{1.2 M\leq |\psi^{-1}(\bar{y}_{ij}^{(1)})| \leq 1.8M\},$$
it can be shown that $\wh{h}=\check{h}$ with high probability. Since $\check{h}$ is measurable with respect to the $\sigma$-algebra generated by $\{(A_{ij},\bar{y}_{ij}^{(1)}): 1.1M<|\theta_{r_i^*}^*-\theta_{r_j^*}^*|<1.9M\}$, we still have the asymptotic independence property between the league partition and local MLE after $h$ being replaced by $\wh{h}$ in Algorithm \ref{alg:partition}. Therefore, with a data-driven $\wh{h}$ being used in the proposed algorithm, the upper bound conclusion of Theorem \ref{thm:BTL-minimax} still holds.

\section{Numerical Results}\label{sec:simulation}

In this section, we conduct numerical experiments to study the statistical and computational properties of Algorithm \ref{alg:whole}.

\paragraph{Simulation Setting.}
In our experiment, we consider $\theta^*\in\mathbb{R}^n$ with $n=1000$. In particular, we set $\theta_i^*=-\beta i$ for all $i\in[n]$ with some $\beta\in[0.001, 0.05]$. The range of $\beta$ implies that the dynamic range $\theta_1^*-\theta_{1000}^*$ takes value in $[0.999, 49.95]$. We assume the true rank is the identity permutation, i.e., $r_i^*=i$ for all $ i\in[n]$. We also consider three different $(L, L_1)$ pairs: (50, 10), (75, 15), (100, 20) in Algorithm \ref{alg:whole}.

\paragraph{Implementation.}
In the implementation of Algorithm \ref{alg:whole}, we set $M = 5$. For the choice of $h$, though the recommended data-driven estimator (\ref{eq:global-h}) works for the theoretical purpose, it may not be a sensible choice for a data set with a moderate size. Note that with $M = 5$, we have $\psi(1.2M)=0.9975274$ and $\psi(1.8M)=0.9998766$, respectively, and thus the indicator $\mathbb{I}\{1.2 M\leq |\psi^{-1}(\bar{y}_{ij}^{(1)})| \leq 1.8M\}$ is usually zero in (\ref{eq:global-h}). To address this issue, we set $h$ by
$$h=0.4\times\frac{1}{n}\sum_{1\leq i<j\leq n}A_{ij}\indc{\psi(-M)\leq\bar{y}_{ij}^{(2)}\leq\psi(M)}.$$
The computation of the local MLE (\ref{eq:league-MLE}) is implemented by the MM algorithm \citep{hunter2004mm}. All simulations are implemented in Python (along with NumPy package, whose backend is written in C) using a 2019 MacBook Pro, 15-inch, 2.6GHz 6-core Intel Core i7. 

\paragraph{Accuracy of League Partition.}
We first study Algorithm \ref{alg:partition}, which is Step 1 of Algorithm \ref{alg:whole}. The purpose of Algorithm \ref{alg:partition} is to divide all players into $K$ leagues. The average value of $K$ from 50 independent experiments is reported in Figure \ref{sim:league-cnt}. This number increases with $\beta$ linearly, which agrees with our theoretical bound $K=O_{\mathbb{P}}(n\beta\vee 1)$. 
\begin{figure}[h]
	\centering
	\includegraphics[width=0.7\textwidth, trim=0 40 0 60, clip]{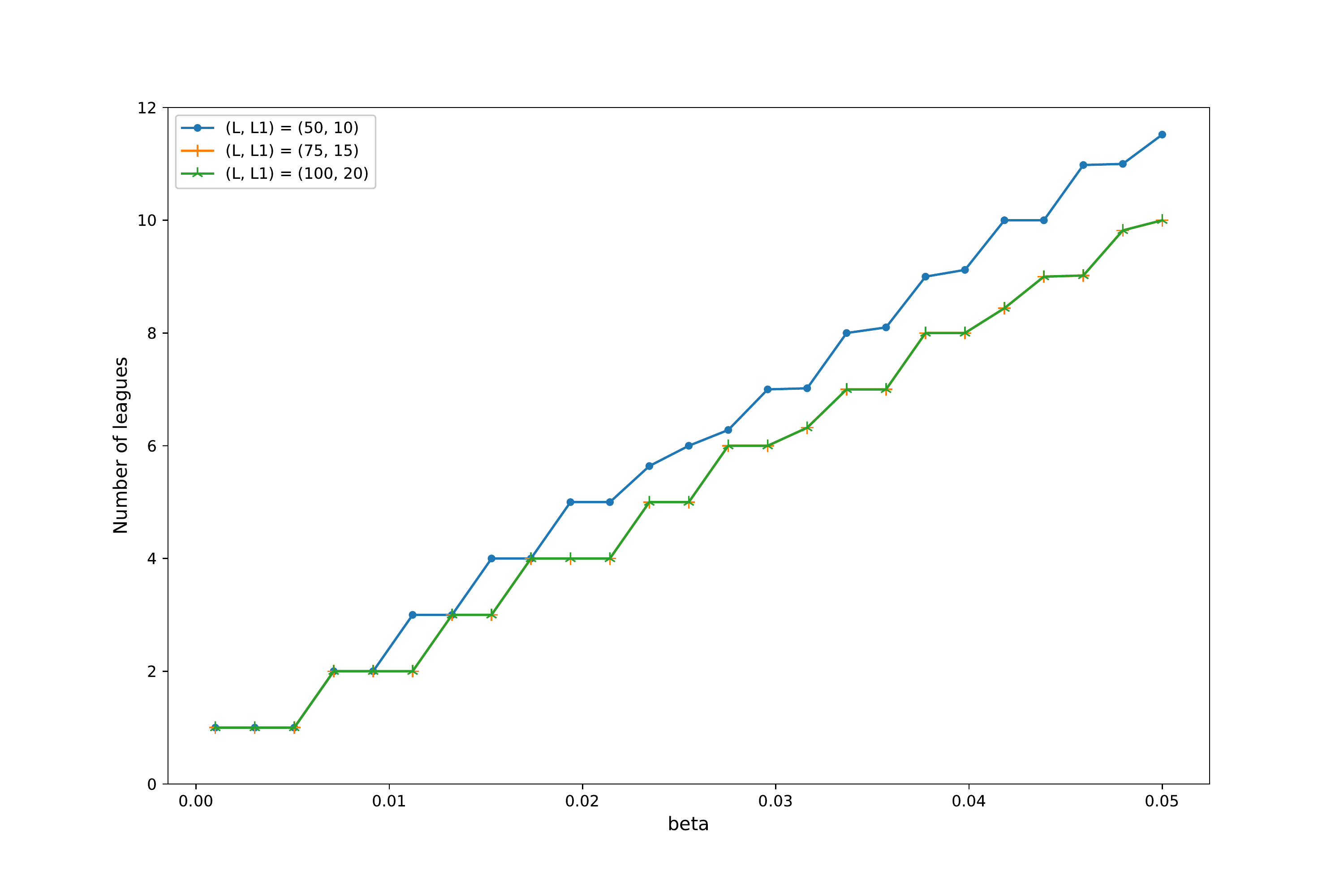}
	\caption{\textsl{The number of leagues obtained by Algorithm \ref{alg:partition}. The orange curve is mostly overlapped by the green curve.}}
	\label{sim:league-cnt}
\end{figure}

To quantify the accuracy of Algorithm \ref{alg:partition}, we define the following metric,
$$E_{partition}=\begin{cases}
\frac{1}{K - 2}\sum_{k=2}^{K-1}\indc{\max\left\{r_i^*:i\in\cup_{k^\prime<k}S_{k^\prime}\right\}>\min\left\{r_i^*:i\in\cup_{k^\prime>k}S_{k^\prime}\right\}}, & K \geq 3,\\
0, & K < 3.
\end{cases}$$
The quantity $E_{partition}$ is essentially designed to verify the Conclusion 3 in Theorem \ref{thm:alg-1}, and we expect that $E_{partition}$ should be $0$ with high probability. Note that Conclusion 3 of Theorem \ref{thm:alg-1} guarantees the correctness of the cross-league pairwise relation estimation, which is Step 3 of Algorithm \ref{alg:whole}. For each combination of $(\beta, L, L_1)$, we generate independent data and repeat the experiments  50 times. It turns out that $E_{partition}$ is always 0, which agrees with the theoretical property of the league partition.

\paragraph{Statistical Error.}

Next, we study the ranking error of the proposed divide-and-conquer algorithm (Algorithm \ref{alg:whole}) under the Kendall's tau distance defined by (\ref{eq:kendall}). For comparison, we also implement the global MLE and the spectral method. The MLE outputs the rank of the entries of $\wh{\theta}$ that maximizes the negative log-likelihood function
\begin{equation}
\sum_{1\leq i<j\leq n}A_{ij}\left[\bar{y}_{ij}\log\frac{1}{\psi(\theta_i-\theta_j)}+(1-\bar{y}_{ij})\log\frac{1}{1-\psi(\theta_i-\theta_j)}\right], \label{eq:global-log-lik}
\end{equation}
where $\bar{y}_{ij}=\frac{1}{L}\sum_{l=1}^Ly_{ijl}$. The spectral method, also known as Rank Centrality, is a ranking algorithm proposed by \cite{negahban2017rank}. Define a matrix $P\in\mathbb{R}^{n\times n}$ by
$$
P_{ij}=\begin{cases}
\frac{1}{d}A_{ij}\bar{y}_{ji}, & i\neq j, \\
1 - \frac{1}{d}\sum_{l\in[n]\backslash\{i\}}A_{il}\bar{y}_{li}, & i=j,
\end{cases} \label{eq:spec-P}
$$
where $d$ is set to be twice the maximum degree of the random graph $A$. Note that $P$ is the transition matrix of a Markov chain. Let $\wh{\pi}$ be the stationary distribution of this Markov chain, and the spectral method outputs the rank of the entries of the vector $\wh{\pi}$.

Both the MLE and the spectral method have been studied for parameter estimation \citep{negahban2017rank,chen2019spectral} and top-$k$ ranking \citep{chen2019spectral,chen2020partial} under the BTL model. However, to the best of our knowledge, the statistical properties of the two methods for full ranking have not been studied in the literature. The recent work \cite{chen2019spectral} has established the estimation errors of the skill parameter for both the MLE and the spectral method. Their results involve a factor of $e^{O(n\beta)}$ in the estimation error under an $\ell_{\infty}$ loss, which suggests that the MLE and the spectral method may not perform well when the dynamic range $n\beta$ diverges.

We implement the MLE, the spectral method, and the divide-and-conquer algorithm for various combinations of $\beta$ and $L$. The results of each setting are computed by averaging across $50$ independent experiments.
\begin{figure}[h]
	\centering
	\includegraphics[width=1.1\textwidth]{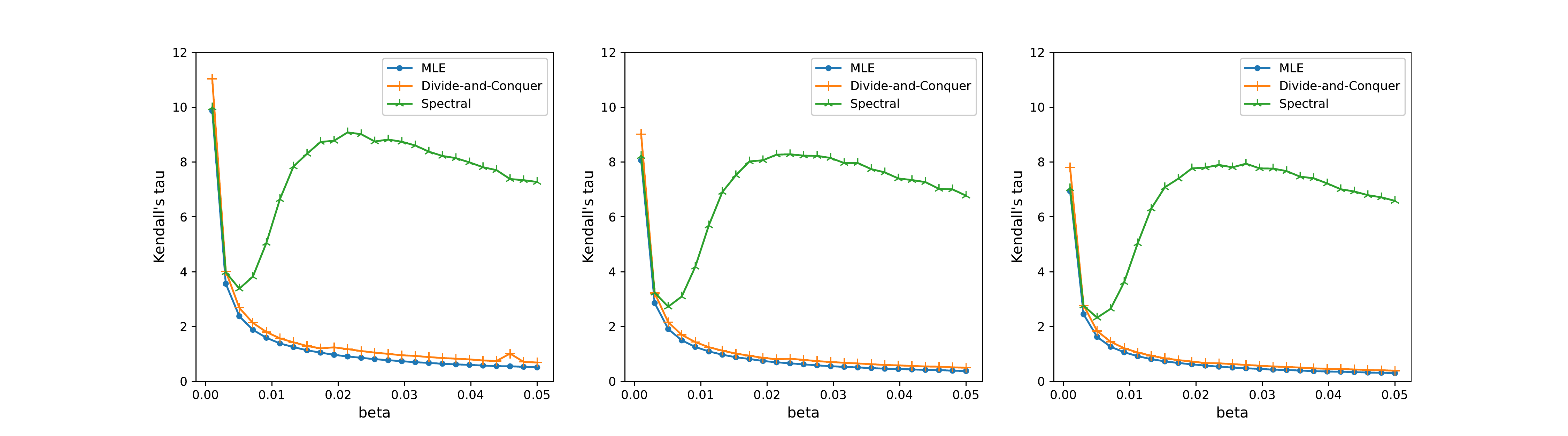}
	\caption{\textsl{Statistical error under Kendall's tau. Left: $(L, L_1)=(50, 10)$; Middle: $(L, L_1)=(75, 15)$; Right: $(L, L_1) = (100, 20)$.}}
	\label{sim:error-plot}
\end{figure}
As shown in Figure \ref{sim:error-plot}, the spectral method is significantly worse than the MLE and the divide-and-conquer algorithm. The performance of the spectral method may be explained by the $e^{O(n\beta)}$ factor in the $\ell_{\infty}$ norm error bound obtained by \cite{chen2019spectral}, though the exact relation between the $\ell_{\infty}$ error and the full ranking error is not clear to us. On the other hand, the error curves of the MLE and the divide-and-conquer algorithm are very close. Since the divide-and-conquer algorithm has been proved to be minimax optimal, the simulation results suggest that the MLE may also enjoy such statistical optimality.

The current analysis of the MLE \citep{chen2019spectral,chen2020partial} crucially depends on the spectral property of the Hessian matrix $H(\theta^*)$ of the objective (\ref{eq:global-log-lik}). It is known that the condition number of $H(\theta^*)$ on the subspace orthogonal to $\mathds{1}_{n}$ is of order $e^{O(n\beta)}$, which explains the $e^{O(n\beta)}$ factor in the $\ell_{\infty}$ estimation error of the MLE \citep{chen2019spectral}. However, our simulation study reveals that the error bound of \cite{chen2019spectral} can be potentially loose. The definition of the Kendall's tau distance suggests that a sharp analysis of the MLE requires a careful study of the random variable $\wh{\theta}_{r_i^*}-\wh{\theta}_{r_j^*}$. We conjecture that the variance of $\wh{\theta}_{r_i^*}-\wh{\theta}_{r_j^*}$ should be approximately proportional to $(e_{r_i^*}-e_{r_j^*})^TH(\theta^*)^{\dagger}(e_{r_i^*}-e_{r_j^*})$, where $e_j$ is the $j$th canonical vector with all entries being 0 except that the $j$th entry is 1. Since $H(\theta^*)$ can be viewed as the graph Laplacian of some random weighted graph, there may exist random matrix tools to study $(e_{r_i^*}-e_{r_j^*})^TH(\theta^*)^{\dagger}(e_{r_i^*}-e_{r_j^*})$ directly without using the naive condition number bound, and we leave this interesting direction as a future project.

In comparison, our divide-and-conquer algorithm does not need to solve the global MLE. Since the objective function of each local MLE is well conditioned (Lemma \ref{lem:bound-hessian}), Algorithm \ref{alg:whole} is provably optimal in addition to its good performance in simulation.

\begin{figure}[h]
	\centering
	\includegraphics[width=1.1\textwidth]{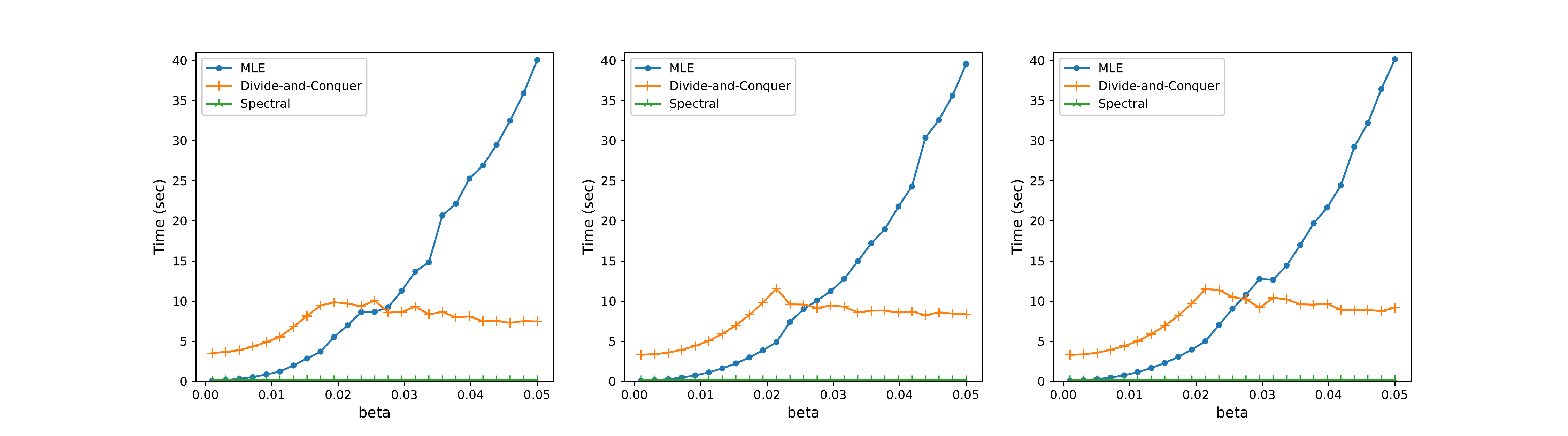}
	\caption{\textsl{Running time comparison. Left: $(L, L_1)=(50, 10)$; Middle: $(L, L_1)=(75, 15)$; Right: $(L, L_1) = (100, 20)$.}}
	\label{sim:time-plot}
\end{figure}
\paragraph{Computational Cost.}
Finally, we compare the computational costs of the three methods. The average time needed to run the three algorithms is given in Figure \ref{sim:time-plot}. The spectral method, though suffers from its unsatisfactory statistical error, is the fastest, partly because finding the stationary distribution is just a single line of code using a NumPy function whose backend is C. The running time of the MLE grows rapidly as $\beta$ increases. This can be explained by the growing condition number of the Hessian matrix $H(\theta^*)$. While the condition number may not affect the statistical error of the MLE, it does have a rather strong effect on its computational cost. On the other hand, the running time for the divide-and-conquer method (Algorithm \ref{alg:whole}) first increases with $\beta$, and then stabilizes. This is the effect of Algorithm \ref{alg:partition}, which divides a large difficult problem into many small sub-problems, and after that each small sub-problem can be conquered efficiently. In fact, we can further improve the computational efficiency by solving the sub-problems in parallel.  The initial increase of the running time of Algorithm \ref{alg:whole} is because of the additional league partition step. Recall that the league partition step divides the players into $K=O_{\mathbb{P}}(n\beta\vee 1)$ subsets. When $\beta$ is small, we have a very small $K$. According to the formula (\ref{eq:league-loss}),  the local MLE is as difficult as the global MLE whenever $K\leq 4$. In this regime, the divide-and-conquer method is more time consuming because of the additional league partition step. On the other hand, as $\beta$ grows, the computational advantage of the divide-and-conquer strategy becomes significant. This makes our proposed algorithm scalable to large data sets, while preserving the statistical optimality, which concludes the divide-and-conquer algorithm as the best overall method among the three.

\section{Discussion}\label{sec:disc}

In this paper, the problem of ranking $n$ players from partial comparison data under the BTL model has been investigated. We have derived the minimax rate with respect to the Kendall's tau distance. A divide-and-conquer algorithm is proposed and is proved to achieve the minimax rate. In this section, we discuss a few directions along which the results of the paper can be extended.

An important condition that we impose throughout the paper is the regularity of the skill parameters $\theta^*\in\Theta_n(\beta,C_0)$. It assumes that $|\theta_i^*-\theta_j^*|\asymp \beta|i-j|$, which roughly describes that players with different skills are evenly distributed in the population. Without this condition, we conjecture that the minimax rate under the Kendall's tau loss should be
$$\inf_{\wh{r}\in\S_n}\sup_{r^*\in\S_n}\mathbb{E}_{(\theta^*,r^*)}\k(\wh{r},r^*)\asymp \frac{1}{n}\sum_{1\leq i<j\leq n}\exp\left(-\frac{(1+o(1))npL(\theta_{i}^*-\theta_{j}^*)^2}{2(V_{i}(\theta^*)+V_{j}(\theta^*))}\right).$$
In fact, this formula has already appeared in the upper bound analysis (\ref{eq:sum-exp}) and can be simplified to the result of Theorem \ref{thm:BTL-minimax} when $\theta^*\in\Theta_n(\beta,C_0)$. Extending the result of Theorem \ref{thm:BTL-minimax} beyond the condition $\theta^*\in\Theta_n(\beta,C_0)$ is possible by some necessary modifications of the league partition step described in Algorithm \ref{alg:partition}. Without $|\theta_i^*-\theta_j^*|\asymp \beta|i-j|$, the partition formula $S_k=\{i\in[n]\backslash(S_1\cup\cdots \cup S_{k-1}): w_i^{(k)}\leq h\}$ should be replaced by $S_k=\{i\in[n]\backslash(S_1\cup\cdots \cup S_{k-1}): w_i^{(k)}\leq h_k\}$ for some sequence $\{h_k\}$ to account for the non-regularity of $\theta^*$. Intuitively, the size of each $|S_k|$ should adaptively depend on the local density of the skill parameters in the neighborhood from which it is selected. Then, the major difficulty is to find a data-driven $\{\wh{h}_k\}$ that estimates the local density. When $|\theta_i^*-\theta_j^*|\asymp \beta|i-j|$, we can just use the global estimator (\ref{eq:global-h}). Without this assumption, estimating $\{h_k\}$ is a much harder problem. In \cite{jadbabaie2020estimation}, it is assumed that the skill parameters $\theta_1^*,\cdots,\theta_n^*$ are i.i.d drawn from some distribution $F$ instead of being fixed parameters, and the authors have studied the problem of estimating $F$, which is called the skill distribution, from the partial pairwise comparison data. Under this formulation, the estimation of the parameters $\{h_k\}$ can be linked to the problem of local bandwidth selection in kernel density estimation \citep{jones1996brief}. We leave this direction of research as one of our future projects.

A restriction of the BTL model is that it can only deal with pairwise comparison. One extension from pairwise comparison to multiple comparison is the popular Plackett-Luce model \cite{plackett1975analysis,luce2012individual}. Suppose there is a subset of $J$ players $S=\{i_1,i_2,\cdots,i_J\}$. Under the Plackett-Luce model, the probability that $j$ is selected among $S$ is given by the formula $\frac{\exp(\theta_j)}{\sum_{i\in S}\exp(\theta_i)}$. Statistical analysis of ranking under the Plackett-Luce model is a problem that has been rarely explored. Both the minimax rate and the construction of optimal algorithms are important open problems.

The ranking problem has also been studied under nonparametric comparison models. For example, a nonparametric stochastically transitive model was proposed by \cite{shah2016stochastically,shah2017simple} and the problems of estimating the mean matrix and top-$k$ ranking have been investigated. However, full ranking is still a problem that has not been well studied under nonparametric models. One of the few works that we are aware of is \cite{mao2018minimax} that assumes $\mathbb{P}(y_{ijl}=1)>\frac{1}{2}+\gamma$ when $r_i^*<r_j^*$. An investigation of full ranking under more general nonparametric settings is another direction to be explored.

\section{Proofs}\label{sec:pf}

\subsection{Proof of Theorem \ref{thm:Gaussian-minimax}}\label{sec:pf-G}

We prove Theorem \ref{thm:Gaussian-minimax} in this section. We first state and prove a few lemmas.

\begin{lemma}\label{lem:A-accurate}
Assume $p\geq\frac{c_0\log n}{n}$ for some sufficiently large $c_0>0$. Then, we have
\begin{equation}
\opnorm{A-\E(A)}\leq C\sqrt{np},\label{eq:A-op}
\end{equation}
\begin{equation}
\opnorm{D-\E(D)}\leq C\sqrt{np\log n}\label{eq:D-op}
\end{equation}
for some constant $C>0$ with probability at least $1-O(n^{-10})$.
\end{lemma}
\begin{proof}
Bound (\ref{eq:A-op}) is a direct consequence of Theorem 5.2 in \cite{lei2015consistency} and Bound (\ref{eq:D-op}) is from standard concentration of sums of i.i.d. Bernoulli random variables.
\end{proof}
\begin{lemma}\label{lem:Laplace-eigen}
Assume $p\geq\frac{c_0\log n}{n}$ for some sufficiently large $c_0>0$. Then, we have
$$np-2C\sqrt{np\log n}\leq\lambda_{\min,\perp}(\mathcal{L}_A)=\min_{u\neq 0, \mathds{1}_n^Tu=0}\frac{u^T\mathcal{L}_A u}{\norm{u}},$$
$$np + 2C\sqrt{np\log n}\geq\lambda_{\max,\perp}(\mathcal{L}_A)= \max_{u\neq 0, \mathds{1}_n^Tu=0}\frac{u^T\mathcal{L}_A u}{\norm{u}}$$
for some constant $C>0$ with probability at least $1-O(n^{-10})$.
\end{lemma}
\begin{proof}
Note the decomposition
$$\mathcal{L}_A=\E\mathcal{L}_A+D-\E D-(A-\E A)$$
and $\lambda_{\min, \perp}(\E\mathcal{L}_A)=\lambda_{\max, \perp}(\E\mathcal{L}_A)=np$. By Lemma \ref{lem:A-accurate}, we have

$$\opnorm{D-\E D-(A-\E A)}\leq 2C\sqrt{np\log n}$$
with probability at least $1-O(n^{-10})$ for some $C>0$. The Lemma can be seen immediately by Weyl's inequality.
\end{proof}

We introduce another notation $r^{*(i,j)}\in\S_n$ to be the element in $\S_n$ having
\begin{align}\label{eqn:r_star_ij_def}
r_k^{*(i,j)}=
\begin{cases}
r_k^*, \text{ if }k\neq i,j \\
r_j^*,\text{ if }k=i\\
r_i^*,\text{ if }k=j\\
\end{cases}.
\end{align}
 That is, $r^{*(i,j)}$ is a permutation by swapping the $i,j$th position in $r^*$ while keeping other positions fixed. 

\begin{lemma}\label{lem:Gaussian-two-point}
Assume $\frac{np}{\log n}\to\infty$. There exists $\delta=o(1)$, such that for any $\theta^*\in\Theta_n(\beta, C_0)$, any $r^*\in\S_n$, any $i,j\in[n], i\neq j$, we have
\begin{align*}
&\inf_{\wh{r}}\frac{\p_{(\theta^*,\sigma^2, r^*)}\left(\wh{r}\neq r^*\right)+\p_{(\theta^*,\sigma^2, r^{*(i,j)})}\left(\wh{r}\neq r^{*(i,j)}\right)}{2}\\
&\gtrsim \min\left\{1, \sqrt{\frac{\sigma^2}{np(\theta_{r_i^*}^*-\theta_{r_j^*}^*)^2}}\exp\left(-\frac{(1+\delta)np(\theta_{r_i^*}^*-\theta_{r_j^*}^*)^2}{4\sigma^2}\right)\right\}
\end{align*}
\end{lemma}
\begin{proof}
Assume $r_i^*=a<r_j^*=b$ and thus $\theta_a^*\geq\theta_b^*$. Let $\mathcal{F}$ be the event about $A$ on which Lemma \ref{lem:A-accurate} holds. We have $\pbr{\mathcal{F}}>1/2$.
To simplify notation, let $\p_A(\cdot)=\p_{(\theta^*,\sigma^2, r^*)}(\cdot|A)$ be the conditional probability. For any $A$, by Neyman-Pearson Lemma, the optimal procedure is given by the likelihood ratio test. Then
\begin{align}
\nonumber&\inf_{\wh{r}}\frac{\p_{(\theta^*,\sigma^2, r^*)}\left(\wh{r}\neq r^*\right)+\p_{(\theta^*,\sigma^2, r^{*(i,j)})}\left(\wh{r}\neq r^{*(i,j)}\right)}{2}\\
\nonumber&\geq \pbr{\mathcal{F}}\inf_{A\in\mathcal{F}}\p_{A}\left(\frac{d\p_{(\theta^*,\sigma^2, r^{*(i,j)})}}{d\p_{(\theta^*,\sigma^2, r^{*})}}\geq1\right)\\
\nonumber&\gtrsim\inf_{A\in\mathcal{F}}\p_{A}\left(-4A_{ij}(\theta_a^*-\theta_b^*+w_{ij})+\sum_{k\neq i,j}-A_{ik}(\theta_a^*-\theta_b^*+2w_{ik})+\sum_{k\neq i,j}A_{jk}(\theta_b^*-\theta_a^*+2w_{jk})\geq0\right)\\
\nonumber&=\inf_{A\in\mathcal{F}}\p_A\left(\mathn(0,\frac{\sigma^2}{D_{ii}+D_{jj}+2A_{ij}})\geq\frac{\abs{\theta_a-\theta_b}}{2}\right)\\
&\gtrsim\min\left\{1, \sqrt{\frac{\sigma^2}{np(\theta_a^*-\theta_b^*)^2}}\exp\left(-\frac{(1+\delta)np(\theta_a^*-\theta_b^*)^2}{4\sigma^2}\right)\right\}\label{eq:Gaussian_lower}
\end{align}
for some $\delta=o(1)$, where (\ref{eq:Gaussian_lower}) comes from standard Gaussian tail bound and Lemma \ref{lem:A-accurate}. 
\end{proof}

Now we are ready to state the proof of Theorem \ref{thm:Gaussian-minimax}.
\begin{proof}[Proof of Theorem \ref{thm:Gaussian-minimax}]
We prove the theorem for any $\theta^*\in\Theta_n(\beta, C_0)$.  Note that conditional on $A$, the solution of the least squares problem (\ref{eq:Gaussian-ls}) can be written as
$$\wh{\theta}=c\mathds{1}_n+\theta_{r^*}^*+Z,$$
where $\theta_{r^*}^*=(\theta_{r_1^*}^*,...,\theta_{r_n^*}^*)^T$, $Z\sim \n(0,\sigma^2\mathcal{L}_A^{\dagger})$ and $c\mathds{1}_n$ is a global shift of the skill parameters. Let $x_{ij}=e_i-e_j$ where $\{e_1,...,e_n\}$ are the standard basis of $\mathbb{R}^n$. Let $\mathcal{F}$ be the event about $A$ when Lemma \ref{lem:Laplace-eigen} holds. Then
\begin{align}
\nonumber&\E_{(\theta^*,\sigma^2, r^*)}\left[\k(\wh{r}, r^*)\right]=\frac{1}{n}\sum_{1\leq i<j\leq n}\p_{(\theta^*,\sigma^2, r^*)}\left(\sgn(\wh{r}_i-\wh{r}_j)\sgn(r_i^*-r_j^*)<0\right)\\
\nonumber&=\frac{1}{n}\sum_{1\leq i<j\leq n}\p_{(\theta^*,\sigma^2, r^*)}\left(\sgn(\wh{\theta}_i-\wh{\theta}_j)\sgn(r_i^*-r_j^*)>0\right)\\
\nonumber&\leq\frac{1}{n}\sum_{1\leq i<j\leq n}\sup_{A\in\mathcal{F}}\p\left(\mathn(0,\sigma^2x_{ij}^T\mathcal{L}_A^{\dagger}x_{ij})>\abs{\theta_{r_i^*}^*-\theta_{r_j^*}^*}|A\right)+O(n^{-9})\\
\nonumber&\leq\frac{1}{n}\sum_{1\leq i<j\leq n}\sup_{A\in\mathcal{F}}\min\left\{1,\sqrt{\frac{\sigma^2x_{ij}^T\mathcal{L}_A^{\dagger}x_{ij}}{2\pi(\theta_{r_i^*}^*-\theta_{r_j^*}^*)^2}}\exp\left(-\frac{(\theta_{r_i^*}^*-\theta_{r_j^*}^*)^2}{2\sigma^2x_{ij}^T\mathcal{L}_A^{\dagger}x_{ij}}\right)\right\}+O(n^{-9})\\
&\leq\frac{1}{n}\sum_{1\leq i<j\leq n}\min\left\{1,\sqrt{\frac{\sigma^2(np-2C\sqrt{np\log n})^{-1}}{\pi(\theta_{r_i^*}^*-\theta_{r_j^*}^*)^2}}\exp\left(-\frac{(np-2C\sqrt{np\log n})(\theta_{r_i^*}^*-\theta_{r_j^*}^*)^2}{4\sigma^2}\right)\right\}\label{eq:Laplace-spec}\\
\nonumber&\quad+O(n^{-9})\\
\nonumber&\lesssim\frac{1}{n}\sum_{1\leq i<j\leq n}\min\left\{1,\sqrt{\frac{\sigma^2}{np(\theta_{i}^*-\theta_{j}^*)^2}}\exp\left(-\frac{(1-\delta_1^\prime)np(\theta_{i}^*-\theta_{j}^*)^2}{4\sigma^2}\right)\right\}+n^{-9}
\end{align}
for some $\delta_1^\prime=o(1)$ independent of $\theta^*$, $\sigma^2$ and $r^*$, where (\ref{eq:Laplace-spec}) is due to Lemma \ref{lem:Laplace-eigen}. 

We first consider the high signal-to-noise ratio regime, where $\frac{np\beta^2}{\sigma^2}>1$. In this scenario,
\begin{align*}
&\sum_{1\leq i<j\leq n}\min\left\{1,\sqrt{\frac{\sigma^2}{np(\theta_{i}^*-\theta_{j}^*)^2}}\exp\left(-\frac{(1-\delta_1^\prime)np(\theta_{i}^*-\theta_{j}^*)^2}{4\sigma^2}\right)\right\}\\
&\leq\sum_{i=1}^{n-1}\sum_{j=i+1}^n\exp\left(-\frac{(1-\delta_1^\prime)np(\theta_{i}^*-\theta_{j}^*)^2}{4\sigma^2}\right)\\
&\leq\sum_{i=1}^{n-1}\exp\left(-\frac{(1-\delta_1^\prime)np(\theta_{i}^*-\theta_{i+1}^*)^2}{4\sigma^2}\right)\sum_{j=i+1}^n\exp\left(-\frac{(1-\delta_1^\prime)np[(\theta_{i}^*-\theta_{j}^*)^2-(\theta_{i}^*-\theta_{i+1}^*)^2]}{4\sigma^2}\right)\\
&\leq\sum_{i=1}^{n-1}\exp\left(-\frac{(1-\delta_1^\prime)np(\theta_{i}^*-\theta_{i+1}^*)^2}{4\sigma^2}\right)\sum_{j=i+1}^n\exp\left(-\frac{(1-\delta_1^\prime)np(j-i-1)\beta^2}{4\sigma^2}\right)\\
&\lesssim\sum_{i=1}^{n-1}\exp\left(-\frac{(1-\delta_1^\prime)np(\theta_{i}^*-\theta_{i+1}^*)^2}{4\sigma^2}\right)
\end{align*}
where the last inequality is due to summation of an exponentially decaying series. This gives the exponential rate in high signal-to-noise ratio regime. 

Now, when $\frac{np\beta^2}{\sigma^2}\leq 1$, 
\begin{align}
\nonumber&\sum_{1\leq i<j\leq n}\min\left\{1,\sqrt{\frac{\sigma^2}{np(\theta_{i}^*-\theta_{j}^*)^2}}\exp\left(-\frac{(1-\delta_1^\prime)np(\theta_{i}^*-\theta_{j}^*)^2}{4\sigma^2}\right)\right\}\\
\nonumber&\leq\sum_{i=1}^{n-1}\sum_{k\geq1}\sum_{\substack{j>i\\(k-1)\sqrt{\frac{\sigma^2}{np\beta^2}}<j-i\leq k\sqrt{\frac{\sigma^2}{np\beta^2}}}}\min\left\{1,\sqrt{\frac{\sigma^2}{np(\theta_i^*-\theta_j^*)^2}}\exp\left(-\frac{(1-\delta^\prime)np(\theta_i^*-\theta_j^*)^2}{4\sigma^2}\right)\right\}\\
\nonumber&\lesssim\sqrt{\frac{\sigma^2}{np\beta^2}}\sum_{i=1}^{n-1}\left(\sum_{k\geq0}\exp\left(-\frac{(1-\delta^\prime)k^2}{4}\right)\right)\lesssim n\sqrt{\frac{\sigma^2}{np\beta^2}}\wedge n^2
\end{align}
where the last inequality also comes from summing an exponentially decaying series and $n^2$ is a trivial upper bound. This finishes the proof of the upper bound.

Now we look at the lower bound. 
For any $r^*\in\S_n$, we have $r^{*(i,j)}\in\S_n$ defined as in (\ref{eqn:r_star_ij_def}). 
Then for any $\theta^*\in\Theta_n(\beta, C_0)$, 
\begin{align}
\nonumber&\inf_{\wh{r}}\sup_{r^*\in\S_n}\E_{(\theta^*, \sigma^2, r^*)}\left[\k(\wh{r}, r^*)\right]\\
\nonumber&\geq\inf_{\wh{r}}\frac{1}{n}\sum_{1\leq i<j\leq n}\frac{1}{n!}\sum_{r^*\in\S_n}\p_{(\theta^*,\sigma^2, r^*)}\left(\sgn(\wh{r}_i-\wh{r}_j)\sgn(r_i^*-r_j^*)<0\right)\\
\nonumber&=\inf_{\wh{r}}\frac{1}{n}\sum_{1\leq i<j\leq n}\frac{1}{n!}\sum_{1\leq a<b\leq n}\sum_{r^*:\{r_i^*,r_j^*\}=\{a,b\}}\p_{(\theta^*,\sigma^2, r^*)}\left(\sgn(\wh{r}_i-\wh{r}_j)\sgn(r_i^*-r_j^*)<0\right)\\
\nonumber&\geq \frac{1}{n}\sum_{1\leq i<j\leq n}\frac{2}{n(n-1)}\sum_{1\leq a<b\leq n}\frac{1}{(n-2)!}\sum_{r^*:r_i^*=a, r_j^*=b}\inf_{\wh{r}}\frac{\p_{(\theta^*,\sigma^2, r^*)}\left(\wh{r}_i\neq a\right)+\p_{(\theta^*,\sigma^2, r^{*(i,j)})}\left(\wh{r}_i\neq b\right)}{2}\\
\nonumber&\gtrsim\frac{1}{n}\sum_{1\leq i<j\leq n}\frac{2}{n(n-1)}\sum_{1\leq a<b\leq n}\\
&\quad\quad\quad\quad\quad\quad\frac{1}{(n-2)!}\sum_{r^*:r_i^*=a, r_j^*=b}\min\left\{1, \sqrt{\frac{\sigma^2}{np(\theta_a^*-\theta_b^*)^2}}\exp\left(-\frac{(1+\delta^\prime)np(\theta_a^*-\theta_b^*)^2}{4\sigma^2}\right)\right\}\label{eq:Gaussian-lower-main}\\
\nonumber&=\frac{1}{n}\sum_{1\leq a<b\leq n}\min\left\{1, \sqrt{\frac{\sigma^2}{np(\theta_a^*-\theta_b^*)^2}}\exp\left(-\frac{(1+\delta^\prime)np(\theta_a^*-\theta_b^*)^2}{4\sigma^2}\right)\right\}
\end{align}
for some $\delta^\prime=o(1)$, where (\ref{eq:Gaussian-lower-main}) comes from Lemma \ref{lem:Gaussian-two-point}. 

We still consider the high signal-to-noise ratio case first. \begin{align}
\nonumber&\sum_{1\leq i<j\leq n}\min\left\{1, \sqrt{\frac{\sigma^2}{np(\theta_i^*-\theta_j^*)^2}}\exp\left(-\frac{(1+\delta^\prime)np(\theta_i^*-\theta_j^*)^2}{4\sigma^2}\right)\right\}\\
\nonumber&\geq\sum_{i=1}^{n-1}\sqrt{\frac{\sigma^2}{np(\theta_i^*-\theta_{i+1}^*)^2}}\exp\left(-\frac{(1+\delta^\prime)np(\theta_i^*-\theta_{i+1}^*)^2}{4\sigma^2}\right)\\
&\gtrsim\sum_{i=1}^{n-1}\exp\left(-\frac{(1+\delta)np(\theta_i^*-\theta_{i+1}^*)^2}{4\sigma^2}\right)\label{eq:Gaussian-absorb}
\end{align}
where $\delta$ in (\ref{eq:Gaussian-absorb}) can be chosen arbitrarily small when $np\beta^2/\sigma^2>1$, which concludes the exponential lower bound.

For the polynomial lower bound when signal-to-noise ratio is small, 
\begin{align*}
&\sum_{1\leq i<j\leq n}\min\left\{1, \sqrt{\frac{\sigma^2}{np(\theta_i^*-\theta_j^*)^2}}\exp\left(-\frac{(1+\delta^\prime)np(\theta_i^*-\theta_j^*)^2}{4\sigma^2}\right)\right\}\\
&\gtrsim\sum_{i=1}^n\sum_{\substack{j\neq i\\\abs{j-i}\leq \sqrt{\frac{\sigma^2}{np\beta^2}}}}\min\left\{1,\sqrt{\frac{\sigma^2}{np(\theta_i-\theta_j)^2}}\exp\left(-\frac{(1+\delta^\prime)np(\theta_i-\theta_j)^2}{4\sigma^2}\right)\right\}\\
&\gtrsim\sum_{i=1}^nn\wedge \left(\sqrt{\frac{\sigma^2}{np\beta^2}}\right)
\end{align*}
which concludes the proof.
\end{proof}

\subsection{Proof of Theorem \ref{thm:BTL-minimax}}\label{sec:pf-BTL}

This section proves Theorem \ref{thm:BTL-minimax}. Since the upper bound part of the proof has already been given in Section \ref{sec:analysis}, we only need to establish the lower bound. First of all, we establish a few lemmas.

\begin{lemma}[Central limit theorem, Theorem 2.20 of \cite{ross2007second}]\label{lem:CLT-stein}
If $Z\sim \n(0,1)$ and $W=\sum_{i=1}^nX_i$ where $X_i$ are independent mean $0$ and $\Var(W)=1$, then
$$\sup_t\left|\mathbb{P}(W\leq t)-\mathbb{P}(Z\leq t)\right| \leq 2\sqrt{3\sum_{i=1}^n\left(\mathbb{E}X_i^4\right)^{3/4}}.$$
\end{lemma}

\begin{lemma}\label{lem:A-bern-2}
Assume $p\geq c_0\frac{\log n}{n}$ for some sufficiently large constant $c_0>0$. For any fixed $\{w_{ijk}\}$, $i,j\in[n], k\in\mathbb{K}$ where $\mathbb{K}$ is a discrete set with cardinality at most $n^{c_1}$ for some constant $c_1>0$. Assume $\max_{i,j\in[n], k\in\mathbb{K}}\abs{w_{ijk}}\leq c_2$ and 
$$p\min_{i\in[n],k\in\mathbb{K}}\sum_{j\in[n]\backslash\{i\}}w_{ijk}^2\geq c_3\log n$$
for some constants $c_2, c_3>0$. Then there exists constants $C_1, C_2>0$, such that for any $i\in[n]$,
$$\max_{k\in\mathbb{K}}\sum_{j\in[n]\backslash\{i\}}(A_{ij}-p)w_{ijk}\leq C_1\sqrt{p\log n\max_{k\in\mathbb{K}}\sum_{j\in[n]}w_{ijk}^2}$$
with probability at least $1-C_2n^{-10}$.
\end{lemma}
\begin{proof}
For any constant $C_1^{\prime}>0$, by Bernstein's inequality, we have
\begin{align*}
&\p\left(\max_{k\in\mathbb{K}}\sum_{j\in[n]\backslash\{i\}}(A_{ij}-p)w_{ijk}> C_1^{\prime}\sqrt{p\log n\max_{k\in\mathbb{K}}\sum_{j\in[n]}w_{ijk}^2}\right)\\
&\leq\abs{\mathbb{K}}\max_{k\in\mathbb{K}}\exp\left(-\frac{C_1^{\prime2}p\log n\max_{k\in\mathbb{K}}\sum_{j\in[n]}w_{ijk}^2}{2p\sum_{j\in[n]\backslash\{i\}}w_{ijk}^2+\frac{2}{3}\max_{i,j\in[n], k\in\mathbb{K}}\abs{w_{ijk}}C_1^{\prime}\sqrt{p\log n\max_{k\in\mathbb{K}}\sum_{j\in[n]}w_{ijk}^2}}\right)\\
&\leq n^{c_1}\exp\left(-\frac{C_1^{\prime2}}{C_2^{\prime}}\log n\right)
\end{align*}
for some constant $C_2^{\prime}>0$. Thus we can set $C_1^{\prime}$ large enough to make the theorem holds.
\end{proof}

\begin{lemma}\label{lem:sum-psi-prime}
Assume $1\leq C_0=O(1)$ and $0<\beta=o(1)$. For any constant $\alpha>0$, there exists constants $C_1, C_2>0$ such that for any $\theta\in\Theta_n(\beta, C_0)$,
$$C_1\frac{1}{\beta\vee1/n}\leq\inf_{\theta_0\in[\theta_n, \theta_1]}\sum_{i=1}^n\psi^{\prime}(\theta_0-\theta_i)^{\alpha}\leq\sup_{\theta_0\in[\theta_n, \theta_1]}\sum_{i=1}^n\psi^{\prime}(\theta_0-\theta_i)^{\alpha}\leq C_2\frac{1}{\beta\vee1/n}$$
for $n$ large enough.
\end{lemma}
\begin{proof}
Define 
\begin{equation}
R_\theta(x, t_1, t_2)=\{i:t_1\leq\abs{\theta_i-x}< t_2\}\label{eq:neighbor-set}
\end{equation}
It is easy to see that there exist constants $C_1^{\prime}, C_2^{\prime}>0$ such that for any  $\theta\in\Theta_n(\beta, C_0)$,
\begin{equation}
\frac{C_1^{\prime}}{\beta\vee1/n}\leq\inf_{x\in[\theta_n, \theta_1]}|R_{\theta}(x, 0, 1)|\label{eq:R-lower}
\end{equation}
and
\begin{equation}
\sup_{t\in\mathbb{N}}\sup_{x\in[\theta_n, \theta_1]}|R_{\theta}(x, t, t+1)|\leq\frac{C_2^{\prime}}{\beta\vee1/n}\label{eq:R-upper}
\end{equation}
Thus
\begin{align*}
&\inf_{\theta_0\in[\theta_n, \theta_1]}\sum_{i=1}^n\psi^{\prime}(\theta_0-\theta_i)^{\alpha}\geq\inf_{\theta_0\in[\theta_n, \theta_1]}\sum_{i\in R_\theta(\theta_0, 0,1)}\psi^{\prime}(\theta_0-\theta_i)^{\alpha}\\
&=\inf_{\theta_0\in[\theta_n, \theta_1]}\sum_{i\in R_\theta(\theta_0, 0,1)}\left[\frac{e^{\theta_0-\theta_i}}{\left(1+e^{\theta_0-\theta_i}\right)^2}\right]^{\alpha}\geq\inf_{\theta_0\in[\theta_n, \theta_1]}\sum_{i\in R_\theta(\theta_0, 0,1)}\frac{1}{4^{\alpha}}e^{-\alpha\abs{\theta_0-\theta_i}}\\
&\geq\inf_{\theta_0\in[\theta_n, \theta_1]}\abs{R_\theta(\theta_0, 0,1)}\frac{1}{4^{\alpha}}e^{-\alpha}\geq\frac{C_3^{\prime}}{\beta\vee1/n}
\end{align*}
for some constant $C_3^\prime>0$. On the other hand, 
\begin{align*}
&\sup_{\theta_0\in[\theta_n, \theta_1]}\sum_{i=1}^n\psi^{\prime}(\theta_0-\theta_i)^{\alpha}=\sup_{\theta_0\in[\theta_n, \theta_1]}\sum_{t\geq0}\sum_{i\in R_{\theta}(\theta_0, t, t+1)}\psi^{\prime}(\theta_0-\theta_i)^{\alpha}\\
&\leq\sup_{\theta_0\in[\theta_n, \theta_1]}\sum_{t\geq0}\sum_{i\in R_{\theta}(\theta_0, t, t+1)}e^{-\alpha\abs{\theta_0-\theta_i}}\leq\sup_{\theta_0\in[\theta_n, \theta_1]}\sum_{t\geq0}\abs{R_\theta(\theta_0, t,t+1)}e^{-\alpha t}\\
&\leq\frac{C_4^{\prime}}{\beta\vee1/n}
\end{align*}
for some constant $C_4^{\prime}>0$, which concludes the proof.
\end{proof}

\begin{lemma}\label{lem:sup-A-u}
Assume $p\geq c_0(\beta\vee\frac{1}{n})\log n$ for some sufficiently large constant $c_0>0$ and $1\leq C_0=O(1)$. For any constant $\alpha>0$, there exist constants $C_1, C_2, C_3>0$ such that for any $r\in\S_n, i\neq j\in[n]$, and $\theta\in\Theta_n(\beta, C_0)$,
\begin{equation}
\inf_{u\in[0,1]}\sum_{k\neq i,j}A_{ik}\psi^{\prime}(u\theta_{r_i}+(1-u)\theta_{r_j}-\theta_{r_k})^{\alpha}\geq C_1\frac{p}{\beta\vee1/n}\label{eq:inf-A-u}
\end{equation}
and
\begin{equation}
\sup_{u\in[0,1]}\sum_{k\neq i,j}A_{ik}\psi^{\prime}(u\theta_{r_i}+(1-u)\theta_{r_j}-\theta_{r_k})^{\alpha}\leq C_2\frac{p}{\beta\vee1/n}\label{eq:sup-A-u}
\end{equation}
with probability at least $1-O(n^{-10})$ for $n$ large enough.
\end{lemma}
\begin{proof}
We remark that $p\geq c_0(\beta\vee\frac{1}{n})\log n$ necessarily implies $0<\beta=o(1)$. We only give the proof of (\ref{eq:sup-A-u}). The inf part (\ref{eq:inf-A-u}) can be proved similarly. For (\ref{eq:sup-A-u}),
\begin{align}
\nonumber&\sup_{u\in[0,1]}\sum_{k\neq i,j}A_{ik}\psi^{\prime}(u\theta_{r_i}+(1-u)\theta_{r_j}-\theta_{r_k})^{\alpha}\\
&\leq\frac{C_1^{\prime}p}{\beta\vee1/n}+\sup_{u\in[0,1]}\sum_{k\neq i,j}(A_{ik}-p)\psi^{\prime}(u\theta_{r_i}+(1-u)\theta_{r_j}-\theta_{r_k})^{\alpha}\label{eq:use-sum-psi-prime}
\end{align}
for some constant $C_1^{\prime}>0$, where (\ref{eq:use-sum-psi-prime}) uses Lemma \ref{lem:sum-psi-prime}. To bound the second term in (\ref{eq:use-sum-psi-prime}), we use standard discretization technique. Let $u_a=\frac{a}{n}, a\in[n]$.   Then for any $u\in[0,1]$, let $a(u)=\arg\min_{a\in[n]}\abs{u-u_a}$. We have $\abs{u-u_{a(u)}}\leq1/n$. Observe that for any $u\in[0,1]$, 
\begin{align}
\nonumber&\abs{\sum_{k\neq i,j}(A_{ik}-p)\left(\psi^{\prime}(u\theta_{r_i}+(1-u)\theta_{r_j}-\theta_{r_k})^{\alpha}-\psi^{\prime}(u_{a(u)}\theta_{r_i}+(1-u_{a(u)})\theta_{r_j}-\theta_{r_k})^{\alpha}\right)}\\
&\leq\alpha\sup_{\xi\in[u\wedge u_{a(u)},u\vee u_{a(u)}]}\sum_{k\neq i,j}\psi^{\prime}(\xi\theta_{r_i}+(1-\xi)\theta_{r_j}-\theta_{r_k})^{\alpha}\abs{u-u_{a(u)}}\abs{\theta_{r_i}-\theta_{r_j}}\label{eq:psi-pp-bound}\\
&\leq\frac{C_2^\prime n\beta}{n}\frac{1}{\beta\vee1/n}\leq C_2^{\prime}\frac{p}{\beta\vee1/n}\label{eq:psi-diff-bound}
\end{align}
for some constant $C_2^{\prime}>0$, where (\ref{eq:psi-pp-bound}) is due to mean value theorem and $\abs{\psi^{\prime\prime}(x)}\leq\psi^{\prime}(x)$ while (\ref{eq:psi-diff-bound}) comes from Lemma \ref{lem:sum-psi-prime}. Therefore,
\begin{align}
\nonumber&\sup_{u\in[0,1]}\sum_{k\neq i,j}A_{ik}\psi^{\prime}(u\theta_{r_i}+(1-u)\theta_{r_j}-\theta_{r_k})^{\alpha}\\
&\leq\frac{C_3^{\prime}p}{\beta\vee1/n}+\max_{a\in[n]}\sum_{k\neq i,j}(A_{ik}-p)\psi^{\prime}(u_a\theta_{r_i}+(1-u_a)\theta_{r_j}-\theta_{r_k})^{\alpha}\label{eq:log-dominate}\\
&\leq\frac{C_3^{\prime}p}{\beta\vee1/n}+C_4^\prime\sqrt{p\log n\max_{a\in[n]}\sum_{k\neq i,j}\psi^{\prime}(u_a\theta_{r_i}+(1-u_a)\theta_{r_j}-\theta_{r_k})^{2\alpha}}\label{eq:simple-concentration}\\
&\leq\frac{C_5^\prime p}{\beta\vee1/n}\label{eq:bound-sum-psi-pp}
\end{align}
for some constants $C_3^\prime, C_4^\prime, C_5^\prime>0$ with probability at least $1-O(n^{-10})$, where (\ref{eq:log-dominate}) is due to (\ref{eq:psi-diff-bound}) and $\frac{p}{\beta\vee1/n}\gtrsim\log n\gg1$. (\ref{eq:simple-concentration}) comes from Lemma \ref{lem:sum-psi-prime}, $\abs{\psi^{\prime}(x)}\leq1/4$ and Lemma \ref{lem:A-bern-2}. (\ref{eq:bound-sum-psi-pp}) is a consequence of Lemma \ref{lem:sum-psi-prime} and $\log n\lesssim\frac{p}{\beta\vee1/n}$, which concludes the proof.
\end{proof}

To proceed with our proof for the lower bound, we define
\begin{align}
G_{i,j,k,\theta,r}(u)=\log\frac{(1+e^{\theta_{r_i}-\theta_{r_k}})^u(1+e^{\theta_{r_j}-\theta_{r_k}})^{1-u}}{1+e^{u\theta_{r_i}+(1-u)\theta_{r_j}-\theta_{r_k}}}.\label{eqn:G_def}
\end{align}
This term is a key ingredient in the exponent of the rate. We first derive some properties of this term. 
\begin{lemma}\label{lem:G-prop}
Assume $1\leq C_0=O(1)$ and $0<\beta=o(1)$. For any constant $C>0$, there exist constants $C_1, C_2, C_3>0$ such that for any $\theta\in\Theta_n(\beta, C_0)$, any $r\in\S_n$ and any $i\neq j\in[n]$ such that $\abs{\theta_{r_i}-\theta_{r_j}}\leq C$, the following hold for $n$ large enough,
\begin{equation}
\sup_{u\in[0,1]}\sup_{k\neq i,j}G_{i,j,k,\theta,r}(u)\leq C_1,\label{eq:G-max}
\end{equation}
\begin{equation}
\sup_{u\in[0,1]}\sum_{k\neq i,j}G_{i,j,k,\theta,r}(u)+G_{i,j,k,\theta, r}(1-u)\leq \sum_{k\neq i,j}\log\frac{(1+e^{\theta_{r_i}-\theta_{r_k}})(1+e^{\theta_{r_j}-\theta_{r_k}})}{\left(1+e^{\frac{\theta_{r_i}+\theta_{r_j}}{2}-\theta_{r_k}}\right)^2},\label{eq:G-sum}
\end{equation}
\begin{equation}
\sup_{u\in[0,1]}\sum_{k\neq i,j}G_{i,j,k,\theta,r}(u)^2+G_{i,j,k,\theta, r}(1-u)^2\leq C_2\frac{(\theta_{r_i}-\theta_{r_j})^4}{\beta\vee1/n},\label{eq:G-square-sum}
\end{equation}
\begin{equation}
C_3\frac{\abs{\theta_{r_i}-\theta_{r_j}}^2}{\beta\vee1/n}\leq\sum_{k\neq i,j}\log\frac{(1+e^{\theta_{r_i}-\theta_{r_k}})(1+e^{\theta_{r_j}-\theta_{r_k}})}{\left(1+e^{\frac{\theta_{r_i}+\theta_{r_j}}{2}-\theta_{r_k}}\right)^2}\leq C_2\frac{\abs{\theta_{r_i}-\theta_{r_j}}^2}{\beta\vee1/n}.\label{eq:G-sum-range}
\end{equation}
\end{lemma}
\begin{proof}
We first look at (\ref{eq:G-max}). Note that
\begin{align*}
&G_{i,j,k,\theta,r}(u)=\log\frac{\psi(u\theta_{r_i}+(1-u)\theta_{r_j}-\theta_{r_k})}{\psi(\theta_{r_i}-\theta_{r_k})^u\psi(\theta_{r_j}-\theta_{r_k})^{1-u}}\\
&\leq\log\frac{\psi(u\theta_{r_i}+(1-u)\theta_{r_j}-\theta_{r_k})}{\psi(\theta_{r_i}-\theta_{r_k})\wedge\psi(\theta_{r_j}-\theta_{r_k})}\\
&=\log\frac{(1+e^{\theta_{r_i}-\theta_{r_k}})}{e^{(1-u)(\theta_{r_i}-\theta_{r_j})}+e^{\theta_{r_i}-\theta_{r_k}}}\vee\log\frac{(1+e^{\theta_{r_j}-\theta_{r_k}})}{e^{-u(\theta_{r_i}-\theta_{r_j})}+e^{\theta_{r_j}-\theta_{r_k}}}\leq C
\end{align*}
where the last inequality comes from $\abs{\theta_{r_i}-\theta_{r_j}}\leq C$.

Now we look at (\ref{eq:G-sum}).
\begin{align*}
&\sup_{u\in[0,1]}\sum_{k\neq i,j}G_{i,j,k,\theta,r}(u)+G_{i,j,k,\theta, r}(1-u)\\
&=\sup_{u\in[0,1]}\sum_{k\neq i,j}\log\frac{(1+e^{\theta_{r_i}-\theta_{r_k}})(1+e^{\theta_{r_j}-\theta_{r_k}})}{1+e^{u\theta_{r_i}+(1-u)\theta_{r_j}-\theta_{r_k}}+e^{(1-u)\theta_{r_i}+u\theta_{r_j}-\theta_{r_k}}+e^{\theta_{r_i}+\theta_{r_j}-2\theta_{r_k}}}\\
&\leq\sup_{u\in[0,1]}\sum_{k\neq i,j}\log\frac{(1+e^{\theta_{r_i}-\theta_{r_k}})(1+e^{\theta_{r_j}-\theta_{r_k}})}{1+2e^{\frac{\theta_{r_i}+\theta_{r_j}}{2}-\theta_{r_k}}+e^{\theta_{r_i}+\theta_{r_j}-2\theta_{r_k}}}=\sum_{k\neq i,j}\log\frac{(1+e^{\theta_{r_i}-\theta_{r_k}})(1+e^{\theta_{r_j}-\theta_{r_k}})}{\left(1+e^{\frac{\theta_{r_i}+\theta_{r_j}}{2}-\theta_{r_k}}\right)^2}.
\end{align*}

To see (\ref{eq:G-square-sum}), we first note that
\begin{align*}
&G_{i,j,k,\theta,r}(u)\\
&=u\log(1+e^{\theta_{r_i}-\theta_{r_k}})+(1-u)\log(1+e^{\theta_{r_j}-\theta_{r_k}})-\log(1+e^{u\theta_{r_i}+(1-u)\theta_{r_j}-\theta_{r_k}})\\
&\geq0
\end{align*}
by Jensen's inequality. Therefore,
\begin{align}
\nonumber&\sup_{u\in[0,1]}\sum_{k\neq i,j}G_{i,j,k,\theta,r}(u)^2+G_{i,j,k,\theta, r}(1-u)^2\leq\sup_{u\in[0,1]}\sum_{k\neq i,j}(G_{i,j,k,\theta,r}(u)+G_{i,j,k,\theta, r}(1-u))^2\\
&\leq\sum_{k\neq i,j}\left[\log\frac{(1+e^{\theta_{r_i}-\theta_{r_k}})(1+e^{\theta_{r_j}-\theta_{r_k}})}{\left(1+e^{\frac{\theta_{r_i}+\theta_{r_j}}{2}-\theta_{r_k}}\right)^2}\right]^2\label{eq:use-G-max}
\end{align}
where (\ref{eq:use-G-max}) can be derived similarly as in the proof of (\ref{eq:G-sum}). To upper bound (\ref{eq:use-G-max}), recall the definition of $R_{\theta}(\cdot,\cdot,\cdot)$ in (\ref{eq:neighbor-set}). We have that for any $k$ such that $r_k\in R_{\theta}(\frac{\theta_{r_i}+\theta_{r_j}}{2}, t, t+1)$,
\begin{align}
\nonumber&\log\frac{(1+e^{\theta_{r_i}-\theta_{r_k}})(1+e^{\theta_{r_j}-\theta_{r_k}})}{\left(1+e^{\frac{\theta_{r_i}+\theta_{r_j}}{2}-\theta_{r_k}}\right)^2}=\log\left(\frac{\cosh(\frac{\theta_{r_i}+\theta_{r_j}}{2}-\theta_{r_k})+\cosh\frac{\theta_{r_i}-\theta_{r_j}}{2}}{\cosh(\frac{\theta_{r_i}+\theta_{r_j}}{2}-\theta_{r_k})+1}\right)\\
&\leq\frac{\cosh\frac{\theta_{r_i}-\theta_{r_j}}{2}-1}{\cosh(\frac{\theta_{r_i}+\theta_{r_j}}{2}-\theta_{r_k})+1}\leq\frac{C_1^{\prime}(\theta_{r_i}-\theta_{r_j})^2}{e^{t}}\label{eq:bound-cosh}
\end{align}
for some constant $C_1^\prime>0$. (\ref{eq:bound-cosh}) can be seen from $\cosh x\leq1+C_2^{\prime}x^2$ for some constant $C_2^{\prime}>0$ when $\abs{x}\leq C/2$.
 and the fact that $t\leq |\frac{\theta_{r_i}+\theta_{r_j}}{2}-\theta_{r_k}|\leq t+1$. 
Therefore, using (\ref{eq:R-upper}) and (\ref{eq:bound-cosh}),
\begin{align*}
&\sum_{k\neq i,j}\left[\log\frac{(1+e^{\theta_{r_i}-\theta_{r_k}})(1+e^{\theta_{r_j}-\theta_{r_k}})}{\left(1+e^{\frac{\theta_{r_i}+\theta_{r_j}}{2}-\theta_{r_k}}\right)^2}\right]^2\\
&\leq\sum_{t\geq0}\sum_{k:r_k\in R_{\theta}(\frac{\theta_{r_i}+\theta_{r_j}}{2}, t, t+1)}\frac{C_1^{\prime2}(\theta_{r_i}-\theta_{r_j})^4}{e^{2t}}\leq\frac{C_3^{\prime}(\theta_{r_i}-\theta_{r_j})^4}{\beta\vee1/n}
\end{align*}
for some constant $C_3^\prime>0$. The upper bound of (\ref{eq:G-sum-range}) can be proved similarly.

Finally, we turn to the lower bound of (\ref{eq:G-sum-range}). Note that we also have $\cosh x\geq1+C_4^{\prime}x^2$ for some constant $C_4^{\prime}>0$ when $\abs{x}\leq C/2$. Therefore, when $r_k\in R_{\theta}(\frac{\theta_{r_i}+\theta_{r_j}}{2}, 0, 1)$,

\begin{align}
\nonumber&\log\frac{(1+e^{\theta_{r_i}-\theta_{r_k}})(1+e^{\theta_{r_j}-\theta_{r_k}})}{\left(1+e^{\frac{\theta_{r_i}+\theta_{r_j}}{2}-\theta_{r_k}}\right)^2}=\log\left(1+\frac{\cosh\frac{\theta_{r_i}-\theta_{r_j}}{2}-1}{\cosh(\frac{\theta_{r_i}+\theta_{r_j}}{2}-\theta_{r_k})+1}\right)\\
&\geq\frac{\cosh\frac{\theta_{r_i}-\theta_{r_j}}{2} - 1}{\cosh\frac{\theta_{r_i}-\theta_{r_j}}{2}+\cosh(\frac{\theta_{r_i}+\theta_{r_j}}{2}-\theta_{r_k})}\geq C_5^{\prime}\abs{\theta_{r_i}-\theta_{r_j}}^2\label{eq:log-lower}
\end{align}
for some constant $C_5^{\prime}>0$, where the first inequality is due to the fact that $\log(1+x) \geq x/(1+x)$ for any $x>-1$. Thus,
\begin{align}
\nonumber&\sum_{k\neq i,j}\log\frac{(1+e^{\theta_{r_i}-\theta_{r_k}})(1+e^{\theta_{r_j}-\theta_{r_k}})}{\left(1+e^{\frac{\theta_{r_i}+\theta_{r_j}}{2}-\theta_{r_k}}\right)^2}\geq\sum_{k:\substack{r_k\in R_{\theta}(\frac{\theta_{r_i}+\theta_{r_j}}{2}, 0, 1)\\k\neq i,j}}\log\frac{(1+e^{\theta_{r_i}-\theta_{r_k}})(1+e^{\theta_{r_j}-\theta_{r_k}})}{\left(1+e^{\frac{\theta_{r_i}+\theta_{r_j}}{2}-\theta_{r_k}}\right)^2}\\
&\geq\left(\abs{R_{\theta}(\frac{\theta_{r_i}+\theta_{r_j}}{2}, 0, 1)}-2\right)C_5^{\prime}\abs{\theta_{r_i}-\theta_{r_j}}^2\geq\frac{C_6^{\prime}\abs{\theta_{r_i}-\theta_{r_j}}^2}{\beta\vee1/n}\label{eq:sum-log-lower}
\end{align}
for some constant $C_6^{\prime}>0$, where (\ref{eq:sum-log-lower}) is a result of (\ref{eq:R-lower}) and (\ref{eq:log-lower}).
\end{proof}

\begin{lemma}\label{lem:sup-A-G}
Assume $\frac{p}{\log n(\beta\vee1/n)}\to\infty$ and $1\leq C_0=O(1)$. For any constant $C_1>0$, there exists $\delta=o(1)$ and constant $C_2>0$,  such that for any $\theta\in\Theta_n(\beta, C_0)$, any $r\in\S_n$ and any $i\neq j\in[n]$ such that $\abs{\theta_{r_i}-\theta_{r_j}}\leq C_1$, the following holds with probability at least $1-O(n^{-10})$ for $n$ large enough,
$$\sup_{u\in[0,1]}\sum_{k\neq i,j}A_{ik}G_{i,j,k,\theta,r}(u)+A_{jk}G_{i,j,k,\theta,r}(1-u)\leq(1+\delta)p\sum_{k\neq i,j}\log\frac{(1+e^{\theta_{r_i}-\theta_{r_k}})(1+e^{\theta_{r_j}-\theta_{r_k}})}{\left(1+e^{\frac{\theta_{r_i}+\theta_{r_j}}{2}-\theta_{r_k}}\right)^2}.$$
\end{lemma}
\begin{proof}
First we have 
\begin{equation}
\psi( a - c )\wedge\psi( b - c )\leq\frac{1}{ a - b }\log\frac{1+e^{ a - c }}{1+e^{ b - c }}\leq\psi( a - c )\vee\psi( b - c ),\label{eq:G-interesting-prop}
\end{equation}
for any $a,b,c\in\mathr$.
To see why (\ref{eq:G-interesting-prop}) holds, let us study the function $f(\delta) = \log(1+\exp(x+\delta)) - \log(1+\exp(x)) - \delta \exp(x)/(1+\exp(x))$ for any $x$. Note that $f'(\delta) = \exp(x+\delta)/(1+\exp(x+\delta)) -  \exp(x)/(1+\exp(x))$ is positive when $\delta >0$ and negative when $\delta<0$.  Since $f(0)=0$,  we have $f(\delta) \geq 0$. As a result, we have $\exp(x)/(1+\exp(x)) \leq \delta^{-1} \log((1+\exp(x+\delta))/(1+\exp(x)))$ when $\delta >0$ and the direction of the inequality is reversed when $\delta <0$. WLOG, we assume $ a  -  b  >0$. Then the first inequality of (\ref{eq:G-interesting-prop})  is proved by taking $x = 	 b  -  c $ and $\delta = a  -  b  $ and second one is proved by  taking $x = 	 a  -  c $ and $\delta =-( a  -  b ) $.


Recall the definition of $G_{i,j,k,\theta,r}$ in (\ref{eqn:G_def}). Then
\begin{align}
\nonumber&\abs{G_{i,j,k,\theta,r}^\prime(u)}=\abs{\theta_{r_i}-\theta_{r_j}}\abs{\frac{1}{\theta_{r_i}-\theta_{r_j}}\log\frac{1+e^{\theta_{r_i}-\theta_{r_k}}}{1+e^{\theta_{r_j}-\theta_{r_k}}}-\frac{e^{u(\theta_{r_i}-\theta_{r_j})+\theta_{r_j}-\theta_{r_k}}}{1+e^{u(\theta_{r_i}-\theta_{r_j})+\theta_{r_j}-\theta_{r_k}}}}\\
&\leq\abs{\theta_{r_i}-\theta_{r_j}}\abs{\psi(\theta_{r_i}-\theta_{r_k})-\psi(\theta_{r_j}-\theta_{r_k})}\label{eq:G-prime-bound-1}\\
&\leq\abs{\theta_{r_i}-\theta_{r_j}}^2.\label{eq:G-prime-bound-2}
\end{align}
Here (\ref{eq:G-prime-bound-1}) is due to the observation that both terms are in the interval $[\psi(\theta_{r_i}-\theta_{r_k})\wedge\psi(\theta_{r_j}-\theta_{r_k}),\psi(\theta_{r_i}-\theta_{r_k})\vee\psi(\theta_{r_j}-\theta_{r_k})]$ for any $u\in[0,1]$, where the first term is due to (\ref{eq:G-interesting-prop}) and the second term is due to the monotonicity of $\exp(x)/(1+\exp(x))$. Hence the difference between these two terms are bounded by $\abs{\psi(\theta_{r_i}-\theta_{r_k})-\psi(\theta_{r_j}-\theta_{r_k})}$ in absolute value.
(\ref{eq:G-prime-bound-2}) is due to $\psi^\prime(x)\leq1/4$.
Following the line of discretization, let $u_a=\frac{a}{n}, a=1,...,n$.   Then for any $u\in[0,1]$, let $a(u)=\arg\min_{a\in[n]}\abs{u-u_a}$. We have $\abs{u-u_{a(u)}}\leq1/n$. Thus,
\begin{align}
\nonumber&\abs{\sum_{k\neq i,j}(A_{ik}-p)(G_{i,j,k,\theta,r}(u)-G_{i,j,k,\theta,r}(u_{a(u)}))+(A_{jk}-p)(G_{i,j,k,\theta,r}(1-u)-G_{i,j,k,\theta,r}(1-u_{a(u)}))}\\
&\leq2\abs{\theta_{r_i}-\theta_{r_j}}^2(n-2)\abs{u-u_{a(u)}}\leq2\abs{\theta_{r_i}-\theta_{r_j}}^2\label{eq:G-discrete}.
\end{align}
Then
\begin{align}
\nonumber&\sup_{u\in[0,1]}\sum_{k\neq i,j}A_{ik}G_{i,j,k,\theta,r}(u)+A_{jk}G_{i,j,k,\theta,r}(1-u)\\
&\leq p\sum_{k\neq i,j}\log\frac{(1+e^{\theta_{r_i}-\theta_{r_k}})(1+e^{\theta_{r_j}-\theta_{r_k}})}{\left(1+e^{\frac{\theta_{r_i}+\theta_{r_j}}{2}-\theta_{r_k}}\right)^2}\label{eq:A-G-sum-bound-1}\\
\nonumber&\quad\quad+\sup_{u\in[0,1]}\sum_{k\neq i,j}(A_{ik}-p)G_{i,j,k,\theta,r}(u)+(A_{jk}-p)G_{i,j,k,\theta,r}(1-u)\\
\nonumber&\leq p\sum_{k\neq i,j}\log\frac{(1+e^{\theta_{r_i}-\theta_{r_k}})(1+e^{\theta_{r_j}-\theta_{r_k}})}{\left(1+e^{\frac{\theta_{r_i}+\theta_{r_j}}{2}-\theta_{r_k}}\right)^2}\\
&\quad\quad+2\abs{\theta_{r_i}-\theta_{r_j}}^2+\max_{a\in[n]}\sum_{k\neq i,j}(A_{ik}-p)G_{i,j,k,\theta,r}(u_{a})+(A_{jk}-p)G_{i,j,k,\theta,r}(1-u_{a})\label{eq:A-G-sum-bound-2}\\
\nonumber&\leq p\sum_{k\neq i,j}\log\frac{(1+e^{\theta_{r_i}-\theta_{r_k}})(1+e^{\theta_{r_j}-\theta_{r_k}})}{\left(1+e^{\frac{\theta_{r_i}+\theta_{r_j}}{2}-\theta_{r_k}}\right)^2}+2\abs{\theta_{r_i}-\theta_{r_j}}^2\\
&\quad\quad+C_1^{\prime}\sqrt{p\log n\max_{a\in[n]}\sum_{k\neq i,j}G_{i,j,k,\theta,r}(u_a)^2+G_{i,j,k,\theta,r}(1-u_a)^2}\label{eq:A-G-sum-bound-3}\\
&\leq p\sum_{k\neq i,j}\log\frac{(1+e^{\theta_{r_i}-\theta_{r_k}})(1+e^{\theta_{r_j}-\theta_{r_k}})}{\left(1+e^{\frac{\theta_{r_i}+\theta_{r_j}}{2}-\theta_{r_k}}\right)^2}+2\abs{\theta_{r_i}-\theta_{r_j}}^2+C_2^\prime\abs{\theta_{r_i}-\theta_{r_j}}^2\sqrt{\frac{p\log n}{\beta\vee1/n}}\label{eq:A-G-sum-bound-4}\\
&=(1+\delta)p\sum_{k\neq i,j}\log\frac{(1+e^{\theta_{r_i}-\theta_{r_k}})(1+e^{\theta_{r_j}-\theta_{r_k}})}{\left(1+e^{\frac{\theta_{r_i}+\theta_{r_j}}{2}-\theta_{r_k}}\right)^2}\label{eq:A-G-sum-bound-5}
\end{align}
with probability at least $1-O(n^{-10})$ for some constants $C_1^\prime, C_2^\prime>0$ and $\delta=o(1)$. (\ref{eq:A-G-sum-bound-1}) is due to Lemma \ref{lem:G-prop}. (\ref{eq:A-G-sum-bound-2}) comes from (\ref{eq:G-discrete}). (\ref{eq:A-G-sum-bound-3}) and (\ref{eq:A-G-sum-bound-4}) are a consequence of Lemma \ref{lem:A-bern-2} and Lemma \ref{lem:G-prop}. (\ref{eq:A-G-sum-bound-5}) is because of Lemma \ref{lem:G-prop} and 
$$p\frac{1}{\abs{\theta_{r_i}-\theta_{r_j}}^2}\sum_{k\neq i,j}\log\frac{(1+e^{\theta_{r_i}-\theta_{r_k}})(1+e^{\theta_{r_j}-\theta_{r_k}})}{\left(1+e^{\frac{\theta_{r_i}+\theta_{r_j}}{2}-\theta_{r_k}}\right)^2}\gtrsim \frac{p}{\beta\vee1/n}\gg\sqrt{\frac{p\log n}{\beta\vee1/n}}\gg1,$$
which concludes the proof.

\end{proof}

\begin{lemma}\label{lem:BTL-two-point}
 Assume $\frac{p}{\log n(\beta\vee1/n)}\to\infty$ and $1\leq C_0=O(1)$. For any constant $C>0$. there exist constants $C_1,C_2>0$, $\delta=o(1)$ such that for any $\theta^*\in\Theta_n(\beta, C_0)$, any  $r^*\in\S_n$ and $i\neq j\in[n]$ such that $\abs{\theta_{r_i^*}^*-\theta_{r_j^*}^*}\leq C$, we have
\begin{align*}
&\inf_{\wh{r}}\frac{\p_{(\theta^*, r^*)}\left(\wh{r}\neq r^*\right)+\p_{(\theta^*, r^{*(i,j)})}\left(\wh{r}\neq r^{*(i,j)}\right)}{2}\\
&\geq C_1\exp\left(-\sqrt{\frac{C_2Lp(\theta_{r_i^*}^*-\theta_{r_j^*}^*)^2}{\beta\vee1/n}}-(1+\delta)2Lp\sum_{k\neq i,j}G_{i,j,k,\theta^*, r^*}(1/2)\right)
\end{align*}
for $n$ large enough. Here $r^{*(i,j)}$ is defined as in (\ref{eqn:r_star_ij_def}).
\end{lemma}
\begin{proof}
By Neyman-Pearson Lemma, the optimal procedure is the likelihood ratio test:
\begin{align*}
&\inf_{\wh{r}}\frac{\p_{(\theta^*, r^*)}\left(\wh{r}\neq r^*\right)+\p_{(\theta^*, r^{*(i,j)})}\left(\wh{r}\neq r^{*(i,j)}\right)}{2}\\
&=\frac{\p_{(\theta^*, r^*)}\left(\ell_n(\theta^*, r^*)\geq\ell_n(\theta^*, r^{*(i,j)})\right)+\p_{(\theta^*, r^{*(i,j)})}\left(\ell_n(\theta^*, r^*)\leq\ell_n(\theta^*, r^{*(i,j)})\right)}{2}
\end{align*}
We only need to lower bound $\p_{(\theta^*, r^*)}\left(\ell_n(\theta^*, r^*)\geq\ell_n(\theta^*, r^{*(i,j)})\right)$ and the other term can be bounded similarly. WLOG, assume $i<j$ and $r_i^*=a<r_j^*=b$. Let
$$Z_{kl}=y_{ikl}\log\frac{\psi(\theta_b^*-\theta_{r_k^*}^*)}{\psi(\theta_a^*-\theta_{r_k^*}^*)}+(1-y_{ikl})\log\frac{1-\psi(\theta_b^*-\theta_{r_k^*}^*)}{1-\psi(\theta_a^*-\theta_{r_k^*}^*)}, k\neq i,j,$$
$$\bar{Z}_{kl}=y_{jkl}\log\frac{\psi(\theta_a^*-\theta_{r_k^*}^*)}{\psi(\theta_b^*-\theta_{r_k^*}^*)}+(1-y_{jkl})\log\frac{1-\psi(\theta_a^*-\theta_{r_k^*}^*)}{1-\psi(\theta_b^*-\theta_{r_k^*}^*)}, k\neq i,j$$
and
$$Z_{0l}=y_{ijl}\log\frac{\psi(\theta_b^*-\theta_{a}^*)}{\psi(\theta_a^*-\theta_{b}^*)}+(1-y_{ijl})\log\frac{1-\psi(\theta_b-\theta_{a}^*)}{1-\psi(\theta_a^*-\theta_{b}^*)}.$$

To simplify notation, we use $\p_A(\cdot)$ as $\p_{(\theta^*, r^*)}(\cdot|A)$ and $\E_{A}[\cdot]$ as $\E_{(\theta^*, r^*)}[\cdot|A]$.
Then
\begin{align*}
&\p_{A}\left(\ell_n(\theta^*, r^*)\geq\ell_n(\theta^*, r^{*(i,j)})\right)=\p_{A}\left(\sum_{l=1}^L\left(A_{ij}Z_{0l}+\sum_{k\neq i,j}A_{ik}Z_{kl}+A_{jk}\bar{Z}_{kl}\right)\geq0\right).
\end{align*}

Let $\mu_{i^\prime j^\prime}=\psi(\theta_{r_{i^\prime}^*}^*-\theta_{r_{j^\prime}^*}^*)$ for any $i^\prime\neq j^\prime$.
Define 
\begin{align*}
&\nu_{r^*}(u)=\log\E_{A}\left\{\exp\left[u\left(A_{ij}Z_{01}+\sum_{k\neq i,j}A_{ik}Z_{k1}+A_{jk}\bar{Z}_{k1}\right)\right]\right\}\\
&=A_{ij}\nu_{0,r^*}(u)+\sum_{k\neq i,j}A_{ik}\nu_{k,r^*}(u)+\sum_{k\neq i,j}A_{jk}\bar{\nu}_{k,r^*}(u)
\end{align*}
where
$$\nu_{0,r^*}(u)=\log\left[\mu_{ij}^u(1-\mu_{ij})^{1-u}+\mu_{ij}^{1-u}(1-\mu_{ij})^u\right]=-\log\frac{1+e^{\theta_a^*-\theta_b^*}}{e^{u(\theta_a^*-\theta_b^*)}+e^{(1-u)(\theta_a^*-\theta_b^*)}}$$
$$\nu_{k,r^*}(u)=\log\left[\mu_{jk}^u\mu_{ik}^{1-u}+(1-\mu_{jk})^u(1-\mu_{ik})^{1-u}\right]=-G_{i,j,k,\theta^*,r^*}(1-u)$$
$$\bar{\nu}_{k,r^*}(u)=\log\left[\mu_{ik}^u\mu_{jk}^{1-u}+(1-\mu_{ik})^u(1-\mu_{jk})^{1-u}\right]=-G_{i,j,k,\theta^*,r^*}(u).$$
$\nu_{r^*}(u)$ is the conditional cumulant generating function of $A_{ij}Z_{01}+\sum_{k\neq i,j}A_{ik}Z_{kl}+A_{jk}\bar{Z}_{kl}$. We also have $\nu_{0,r^*}(u), \nu_{k,r^*}(u),\bar{\nu}_{k,r^*}(u)$ as the cumulant generating functions of $Z_{01}, Z_{k1}, \bar{Z}_{k1}$ respectively. Define
$$u_{r^*}^*=\arg\min_{u\geq0}\nu_{r^*}(u).$$
Since cumulant generating functions are convex and  $\nu_{r^*}(0) = \nu_{r^*}(1)=0$,
it can be seen easily that $u_{r^*}^*\in(0,1)$ and depends on $A$. 
Following the change-of-measure argument in the proof of Lemma 8.4 and Lemma 9.3 of \cite{chen2020partial}, we have
\begin{align}
\nonumber&\p_{A}\left(\sum_{l=1}^L\left(A_{ij}Z_{0l}+\sum_{k\neq i,j}A_{ik}Z_{kl}+A_{jk}\bar{Z}_{kl}\right)\geq0\right)\\
&\geq\exp\left(-u_{r^*}^*T+L\nu_{r^*}(u_{r^*}^*)\right)\mathbb{Q}_A\left(0\leq\sum_{l=1}^L\left(A_{ij}Z_{0l}+\sum_{k\neq i,j}A_{ik}Z_{kl}+A_{jk}\bar{Z}_{kl}\right)\leq T\right)\label{eq:change-measure}
\end{align}
for any $T$ in (\ref{eq:change-measure}) to be determined later and $\mathbb{Q}_A$ is a measure under which $Z_{0l}, Z_{kl}, \bar{Z}_{kl}, l\in[L], k\neq i,j$ are all independent given $A$ and follow
$$\mathbb{Q}_A(Z_{0l}=s)=e^{u_{r^*}^*s-\nu_{0,r^*}(u_{r^*}^*)}\p_A\left(Z_{0l}=s\right),$$
$$\mathbb{Q}_A(Z_{kl}=s)=e^{u_{r^*}^*s-\nu_{k,r^*}(u_{r^*}^*)}\p_A\left(Z_{kl}=s\right),k\neq i,j,k\in[n],$$
$$\mathbb{Q}_A(\bar{Z}_{kl}=s)=e^{u_{r^*}^*s-\bar{\nu}_{k,r^*}(u_{r^*}^*)}\p_A\left(\bar{Z}_{kl}=s\right), k\neq i,j, k\in[n].$$
Furthermore, by definition of $u_{r^*}^*$, the expectation of $A_{ij}Z_{0l}+\sum_{k\neq i,j}A_{ik}Z_{kl}+A_{jk}\bar{Z}_{kl}$ under $\mathbb{Q}_A$ is 0. 

We can compute the 2nd and 4th moments under $Q_A$, denoted as $\Var_{\mathbb{Q}_A}(\cdot)$ and $\kappa_{\mathbb{Q}_A}(\cdot)$ respectively:
\begin{align}
\nonumber&\Var_{\mathbb{Q}_A}(Z_{0l})=\nu_{0,r^*}^{\prime\prime}(u_{r^*}^*)=4\mu_{ij}(1-\mu_{ij})\frac{(\theta_a^*-\theta_b^*)^2e^{2u_{r^*}^*(\theta_a^*-\theta_b^*)}}{((1-\mu_{ij})e^{2u_{r^*}^*(\theta_a^*-\theta_b^*)}+\mu_{ij})^2}\\
&=4(\theta_a^*-\theta_b^*)^2\psi^{\prime}\left((1-2u_{r^*}^*)(\theta_a^*-\theta_b^*)\right),\label{eq:Q-2nd-1}
\end{align}
\begin{align}
\nonumber&\Var_{\mathbb{Q}_A}(Z_{kl})=\nu_{k,r^*}^{\prime\prime}(u_{r^*}^*)=\mu_{ik}(1-\mu_{ik})\frac{(\theta_a^*-\theta_b^*)^2e^{u_{r^*}^*(\theta_a^*-\theta_b^*)}}{((1-\mu_{ik})e^{u_{r^*}^*(\theta_a^*-\theta_b^*)}+\mu_{ik})^2}\\
&=(\theta_a^*-\theta_b^*)^2\psi^{\prime}\left((1-u_{r^*}^*)\theta_a^*+u_{r^*}^*\theta_b^*-\theta_{r_k^*}^*\right), k\neq i,j, k\in[n],\label{eq:Q-2nd-2}
\end{align}
\begin{align}
\nonumber&\Var_{\mathbb{Q}_A}(\bar{Z}_{kl})=\bar{\nu}_{k,r^*}^{\prime\prime}(u_{r^*}^*)=\mu_{jk}(1-\mu_{jk})\frac{(\theta_a^*-\theta_b^*)^2e^{-u_{r^*}^*(\theta_a^*-\theta_b^*)}}{((1-\mu_{jk})e^{-u_{r^*}^*(\theta_a^*-\theta_b^*)}+\mu_{jk})^2}\\
&=(\theta_a^*-\theta_b^*)^2\psi^{\prime}\left(u_{r^*}^*\theta_a^*+(1-u_{r^*}^*)\theta_b^*-\theta_{r_k^*}^*\right),k\neq i,j, k\in[n]\label{eq:Q-2nd-3}
\end{align}
and
\begin{align}
\nonumber&\kappa_{\mathbb{Q}_A}(Z_{0l})=\mathbb{Q}_A\left((Z_{0l}-\mathbb{Q}_A(Z_{0l}))^4\right)=\nu_{0,r^*}^{\prime\prime\prime\prime}(u_{r^*}^*)+3\nu_{0,r^*}^{\prime\prime}(u_{r^*}^*)^2\\
\nonumber&\leq16\mu_{ij}(1-\mu_{ij})\frac{(\theta_a^*-\theta_b^*)^4e^{2u_{r^*}^*(\theta_a^*-\theta_b^*)}}{[(1-\mu_{ij})e^{2u_{r^*}^*(\theta_a^*-\theta_b^*)}+\mu_{ij}]^2}+3\nu_{0,r^*}^{\prime\prime}(u_{r^*}^*)^2\\
&=4(\theta_a^*-\theta_b^*)^2\nu_{0,r^*}^{\prime\prime}(u_{r^*}^*)+3\nu_{0,r^*}^{\prime\prime}(u_{r^*}^*)^2\leq7(\theta_a^*-\theta_b^*)^4\psi^{\prime}\left((1-2u_{r^*}^*)(\theta_a^*-\theta_b^*)\right),\label{eq:Q-4th-1}
\end{align}
\begin{align}
\nonumber&\kappa_{\mathbb{Q}_A}(Z_{kl})=\mathbb{Q}_A\left((Z_{kl}-\mathbb{Q}_A(Z_{kl}))^4\right)\leq(\theta_a^*-\theta_b^*)^2\nu_{k,r^*}^{\prime\prime}(u_{r^*}^*)+3\nu_{k,r^*}^{\prime\prime}(u_{r^*}^*)^2\\
&\leq4(\theta_a^*-\theta_b^*)^4\psi^{\prime}\left((1-u_{r^*}^*)\theta_a^*+u_{r^*}^*\theta_b^*-\theta_{r_k^*}^*\right),k\neq i,j, k\in[n]\label{eq:Q-4th-2}
\end{align}
\begin{align}
\nonumber&\kappa_{\mathbb{Q}_A}(\bar{Z}_{kl})=\mathbb{Q}_A\left((\bar{Z}_{kl}-\mathbb{Q}_A(\bar{Z}_{kl}))^4\right)\leq(\theta_a^*-\theta_b^*)^2\bar{\nu}_{k,r^*}^{\prime\prime}(u_{r^*}^*)+3\bar{\nu}_{k,r^*}^{\prime\prime}(u_{r^*}^*)^2\\
&\leq4(\theta_a^*-\theta_b^*)^4\psi^{\prime}\left(u_{r^*}^*\theta_a^*+(1-u_{r^*}^*)\theta_b^*-\theta_{r_k^*}^*\right),k\neq i,j, k\in[n].\label{eq:Q-4th-3}
\end{align}

Let $\mathcal{F}_1$ be the event on which the following holds:
$$
\inf_{u\in[0,1]}\sum_{k\neq i,j}A_{ik}\psi^\prime((1-u)\theta_a^*+u\theta_b^*-\theta_{r_k^*}^*)+A_{jk}\psi^\prime(u\theta_a^*+(1-u)\theta_b^*-\theta_{r_k^*}^*)\geq C_1^\prime\frac{p}{\beta\vee1/n},$$
$$
\sup_{u\in[0,1]}\sum_{k\neq i,j}A_{ik}\psi^\prime((1-u)\theta_a^*+u\theta_b^*-\theta_{r_k^*}^*)^{3/4}+A_{jk}\psi^\prime(u\theta_a^*+(1-u)\theta_b^*-\theta_{r_k^*}^*)^{3/4}\leq C_2^\prime\frac{p}{\beta\vee1/n}
$$
for some constants $C_1^\prime, C_2^\prime>0$. We shall choose $C_1^\prime, C_2^\prime$ to make $\mathcal{F}_1$ happen with probability at least $1-O(n^{-10})$ by Lemma \ref{lem:sup-A-u}. Therefore, we shall choose $T$ as
\begin{align*}
&T=\sqrt{L\left(A_{ij}\Var_{\mathbb{Q}_A}(Z_{01})+\sum_{k\neq i,j}A_{ik}\Var_{\mathbb{Q}_A}(Z_{k1})+A_{jk}\Var_{\mathbb{Q}_A}(\bar{Z}_{k1})\right)}\\
&\leq\sqrt{C_3^{\prime}L\frac{p(\theta_a^*-\theta_b^*)^2}{\beta\vee1/n}}
\end{align*}
on $\mathcal{F}_1$ for some constant $C_3^\prime>0$ using (\ref{eq:Q-2nd-1})-(\ref{eq:Q-2nd-3}). With this choice of $T$, the $\mathbb{Q}_A$ measure can be lower bounded by some constant $C_4^\prime>0$ on $\mathcal{F}_1$. This can be seen by bounding the 4th moment approximation bound using Lemma \ref{lem:CLT-stein}
:
\begin{align}
\nonumber&\sqrt{L\frac{A_{ij}\kappa_{\mathbb{Q}_A}(Z_{01})^{3/4}+\sum_{k\neq i,j}A_{ik}\kappa_{\mathbb{Q}_A}(Z_{kl})^{3/4}+A_{jk}\kappa_{\mathbb{Q}_A}(\bar{Z}_{kl})^{3/4}}{\left(LA_{ij}\Var_{\mathbb{Q}_A}(Z_{01})+L\sum_{k\neq i,j}A_{ik}\Var_{\mathbb{Q}_A}(Z_{k1})+A_{jk}\Var_{\mathbb{Q}_A}(\bar{Z}_{k1})\right)^{3/2}}}\\
&\leq\sqrt{C_5^\prime L\frac{\sum_{k\neq i,j}A_{ik}\psi^{\prime}\left((1-u_{r^*}^*)\theta_a^*+u_{r^*}^*\theta_b^*-\theta_{r_k^*}^*\right)^{3/4}+A_{jk}\psi^{\prime}\left(u_{r^*}^*\theta_a^*+(1-u_{r^*}^*)\theta_b^*-\theta_{r_k^*}^*\right)^{3/4}}{\left(L\sum_{k\neq i,j}A_{ik}\psi^\prime((1-u_{r^*}^*)\theta_a^*+u_{r^*}^*\theta_b^*-\theta_{r_k^*}^*)+A_{jk}\psi^\prime(u_{r^*}^*\theta_a^*+(1-u_{r^*}^*)\theta_b^*-\theta_{r_k^*}^*)\right)^{3/2}}}\label{eq:Q-clt-bound-1}\\
&\leq C_6^{\prime}\left(L\frac{p}{\beta\vee1/n}\right)^{-1/4}\label{eq:Q-clt-bound-2}
\end{align}
on $\mathcal{F}_1$ for some constants $C_5^\prime, C_6^\prime>0$  and this bound tends to 0. (\ref{eq:Q-clt-bound-1}) is due to (\ref{eq:Q-2nd-1})-(\ref{eq:Q-2nd-3}) and (\ref{eq:Q-4th-1})-(\ref{eq:Q-4th-3}). (\ref{eq:Q-clt-bound-2}) is a consequence of Lemma \ref{lem:sup-A-u}. 

Now we turn to $L\nu_{r^*}(u_{r^*}^*)$. Let $\mathcal{F}_2$ be the event on which the following holds:
\begin{align*}
&\sup_{u\in[0,1]}\sum_{k\neq i,j}(A_{ik}G_{i,j,k,\theta^*,r^*}(1-u)+A_{jk}G_{i,j,k,\theta^*,r^*}(u))\\
&\leq(1+\delta_1^\prime)2p\sum_{k\neq i,j}G_{i,j,k,\theta^*, r^*}(1/2).
\end{align*}
By Lemma \ref{lem:sup-A-G}, there exists $\delta_1^\prime=o(1)$ independent of $i,j,\theta^*,r^*$ such that $\mathcal{F}_2$ holds with probability at least $1-O(n^{-10})$. Then, on this event,
\begin{align}
\nonumber&\nu_{r^*}(u_{r^*})\geq-\sup_{u\in[0,1]}\left(-A_{ij}\nu_{0,r^*}(u)-\sum_{k\neq i,j}(A_{ik}\nu_{k,r^*}(u)+A_{jk}\bar{\nu}_{k,r^*}(u))\right)\\
\nonumber&\geq-A_{ij}\sup_{u\in[0,1]}\log\frac{1+e^{\theta_a^*-\theta_b^*}}{e^{u(\theta_a^*-\theta_b^*)}+e^{(1-u)(\theta_a^*-\theta_b^*)}}-\sup_{u\in[0,1]}\sum_{k\neq i,j}(A_{ik}G_{i,j,k,\theta^*,r^*}(1-u)+A_{jk}G_{i,j,k,\theta^*,r^*}(u))\\
&\geq-C_7^\prime\abs{\theta_a^*-\theta_b^*}^2-(1+\delta_1^\prime)2p\sum_{k\neq i,j}G_{i,j,k,\theta^*, r^*}(1/2)\label{eq:ij-bound}\\
&\geq-(1+\delta_2^\prime)2p\sum_{k\neq i,j}G_{i,j,k,\theta^*, r^*}(1/2)\label{eq:ij-absorb}
\end{align}
for some $\delta_2^\prime=o(1)$. (\ref{eq:ij-bound}) comes from 
$$\log\frac{1+e^{\theta_a^*-\theta_b^*}}{e^{u(\theta_a^*-\theta_b^*)}+e^{(1-u)(\theta_a^*-\theta_b^*)}}\leq\log\cosh\frac{\theta_a^*-\theta_b^*}{2}\leq\cosh\frac{\theta_a^*-\theta_b^*}{2} - 1\leq C_7^\prime\abs{\theta_a^*-\theta_b^*}^2$$
for some constant $C_7^\prime>0$ when $\abs{\theta_a^*-\theta_b^*}\leq C$. (\ref{eq:ij-absorb}) is because of Lemma \ref{lem:G-prop} and $\frac{p}{\beta\vee1/n}\gg1$. Note that $\delta_2^\prime$ can also be chosen independent of $i,j,\theta^*,r^*$.

Thus, we can further lower bound (\ref{eq:change-measure}) on $\mathcal{F}_1\cap\mathcal{F}_2$:
\begin{align}
\nonumber&\p_{A}\left(\sum_{l=1}^L\left(A_{ij}Z_{0l}+\sum_{k\neq i,j}A_{ik}Z_{kl}+A_{jk}\bar{Z}_{kl}\right)\geq0\right)\\
\nonumber&\geq C_4^\prime\exp\left(-\sqrt{C_3^\prime Lp\frac{(\theta_a^*-\theta_b^*)^2}{\beta\vee1/n}}+L\nu_{r^*}(u_{r^*}^*)\right)\\
\nonumber&\geq C_4^\prime\exp\left(-\sqrt{C_3^\prime Lp\frac{(\theta_a^*-\theta_b^*)^2}{\beta\vee1/n}}-L\sup_{u\in[0,1]}\left(-A_{ij}\nu_{0,r^*}(u)-\sum_{k\neq i,j}(A_{ik}\nu_{k,r^*}(u)+A_{jk}\bar{\nu}_{k,r^*}(u))\right)\right)\\
\nonumber&\geq C_4^\prime\exp\left(-\sqrt{C_3^\prime Lp\frac{(\theta_a^*-\theta_b^*)^2}{\beta\vee1/n}}-(1+\delta_2^\prime)2Lp\sum_{k\neq i,j}G_{i,j,k,\theta^*, r^*}(1/2)\right).
\end{align}
which finishes the proof.
\end{proof}

Now we are ready to prove the lower bound part of Theorem \ref{thm:BTL-minimax}.
\begin{proof}[Proof of Theorem \ref{thm:BTL-minimax} (lower bound)]
We remark that $p\geq c_0(\beta\vee\frac{1}{n})\log n$ necessarily implies $0<\beta=o(1)$.  It also implies $n\wedge \beta^{-1}\gg 1$ and $\frac{\beta\vee1/n}{Lp}=o(1)$ which will be useful in the proof. Recall the definition of  $r^{*(i,j)}$ in (\ref{eqn:r_star_ij_def}) for any $r^{*}\in\S_n$ and $i,j\in[n]$ such that $i\neq j$.

For any $\theta^*\in\Theta_n(\beta, C_0)$, we have
\begin{align}
\nonumber&\inf_{\wh{r}}\sup_{r^*\in\S_n}\E_{(\theta^*,r^*)}\left[\k(\wh{r},r)\right]\\
\nonumber&\geq\inf_{\wh{r}}\frac{1}{n!}\sum_{r^*\in\S_n}\frac{1}{n}\sum_{1\leq i< j\leq n}\p_{(\theta^*,r^*)}\left(\wh{r}_i<\wh{r}_j,r_i^*>r_j^*\right)+\p_{(\theta^*,r^*)}\left(\wh{r}_i>\wh{r}_j,r_i^*<r_j^*\right)\\
\nonumber&=\inf_{\wh{r}}\frac{1}{n}\sum_{1\leq i<j\leq n}\frac{1}{n!}\sum_{r^*\in\S_n}\p_{(\theta^*,r^*)}\left(\wh{r}_i<\wh{r}_j,r_i^*>r_j^*\right)+\p_{(\theta^*,r^*)}\left(\wh{r}_i>\wh{r}_j,r_i^*<r_j^*\right)\\
\nonumber&\geq\frac{1}{n}\sum_{1\leq a<b\leq n}\frac{2}{n(n-1)}\sum_{1\leq i<j\leq n}\\
\nonumber&\quad\quad\quad\quad\quad\quad\quad\frac{1}{(n-2)!}\sum_{r^*:r_i^*=a,r_j^*=b}\inf_{\wh{r}}\frac{\p_{(\theta^*,r^*)}(\wh{r}\neq r^*)+\p_{(\theta^*,r^{*(i,j)})}(\wh{r}\neq r^{*(i.j)})}{2}\\
\nonumber&\geq\frac{1}{2n}\sum_{a=1}^n\frac{2}{n(n-1)}\sum_{1\leq i<j\leq n}\sum_{b\in[n]\backslash\{a\}:\abs{a-b}\leq 1\vee \sqrt{\frac{C_1^\prime(\beta\vee n^{-1})}{Lp\beta^2}}}\\
\nonumber&\quad\quad\quad\quad\quad\quad\quad\frac{1}{(n-2)!}\sum_{r^*:r_i^*=a,r_j^*=b}\inf_{\wh{r}}\frac{\p_{(\theta^*,r^*)}(\wh{r}\neq r^*)+\p_{(\theta^*,r^{*(i,j)})}(\wh{r}\neq r^{*(i.j)})}{2},
\end{align}
where $C_1^\prime>0$ is a constant. Note that for any $a,b\in[n]$ such that $\abs{a-b}\leq 1\vee \sqrt{\frac{C_1^\prime(\beta\vee n^{-1} )}{Lp\beta^2}}$, we have $\abs{\theta^*_a - \theta^*_b}\leq C_0 \br{\beta  \vee \sqrt{\frac{C_1^\prime(\beta\vee n^{-1} )}{Lp}}}=o(1)$. Then  by Lemma \ref{lem:BTL-two-point}, we have
\begin{align}
\nonumber&\inf_{\wh{r}}\sup_{r^*\in\S_n}\E_{(\theta^*,r^*)}\left[\k(\wh{r},r)\right]\geq\frac{1}{2n}\sum_{a=1}^n\frac{2}{n(n-1)}\sum_{1\leq i<j\leq n}\sum_{b\in[n]\backslash\{a\}:\abs{a-b}\leq 1\vee \sqrt{\frac{C_1^\prime(\beta\vee n^{-1} )}{Lp\beta^2}}}\\
&\quad\quad C_3^\prime\exp\left(-\sqrt{\frac{C_2^\prime Lp(\theta_a^*-\theta_b^*)^2}{\beta\vee n^{-1} }}-(1+\delta_1^\prime)Lp\sum_{k\neq a, b}\log\frac{(1+e^{\theta_a^*-\theta_k^*})(1+e^{\theta_{b}^*-\theta_k^*})}{\left(1+e^{\frac{\theta_a^*+\theta_{b}^*}{2}-\theta_k^*}\right)^2}\right),\label{eq:use-BTL-two-point}
\end{align}
for some constant $C_2^\prime,C_3^\prime>0$ and some $\delta_1^\prime=o(1)$. We are going to simplify the second term in the exponent. We have
\begin{align}
\sum_{k\neq a, b}\log\frac{(1+e^{\theta_a^*-\theta_k^*})(1+e^{\theta_{b}^*-\theta_k^*})}{\left(1+e^{\frac{\theta_a^*+\theta_{b}^*}{2}-\theta_k^*}\right)^2}&\leq\sum_{k\neq a, b}\frac{e^{\theta_a^*-\theta_k^*}+e^{\theta_b^*-\theta_k^*}-2e^{\frac{\theta_a^*+\theta_{b}^*}{2}-\theta_k^*}}{\left(1+e^{\frac{\theta_a^*+\theta_{b}^*}{2}-\theta_k^*}\right)^2}\label{eq:log-ineq}\\
\nonumber&=2\sum_{k\neq a, b}\frac{(\cosh\frac{\theta_a^*-\theta_{b}^*}{2}-1)e^{\frac{\theta_a^*+\theta_{b}^*}{2}-\theta_k^*}}{\left(1+e^{\frac{\theta_a^*+\theta_{b}^*}{2}-\theta_k^*}\right)^2}\\
&=(1+\delta_2^\prime)\frac{(\theta_a^*-\theta_{b}^*)^2}{4}\sum_{k\neq a, b}\psi^{\prime}\left(\frac{\theta_a^*+\theta_{b}^*}{2}-\theta_k^*\right)\label{eq:cosh-expand}\\
&=(1+\delta_3^\prime)\frac{(\theta_a^*-\theta_{b}^*)^2}{4}\sum_{k\neq a}\psi^{\prime}(\theta_a^*-\theta_k^*)\label{eq:a-a+1-close}
\end{align}
for some $\delta_2^\prime=o(1), \delta_3^\prime=o(1)$. Here (\ref{eq:log-ineq}) uses $\log(1+x)\leq x$. In  (\ref{eq:cosh-expand}) we use $\theta_a^*-\theta_{b}^*=o(1)$ and $\cosh x-1=(1+O(x))\frac{x^2}{2}$ when $x=o(1)$.
 From Lemma \ref{lem:sum-psi-prime}  we know $\sum_{k\neq a}\psi^{\prime}(\theta_a^*-\theta_k^*)\asymp n\wedge \beta^{-1}\gg1$. Then using this and the fact     $\sup_x\abs{\frac{\psi^\prime(x+ t )}{\psi^\prime(x)}-1}=O( t )$ when $ t =o(1)$, we obtain (\ref{eq:a-a+1-close}). Using $\sum_{k\neq a}\psi^{\prime}(\theta_a^*-\theta_k^*)\asymp n\wedge \beta^{-1}$ again and the fact $\abs{\theta^*_a - \theta^*_b} \geq \beta$,  there exists a constant $C_4^\prime>0$ such that
$$\frac{\sqrt{\frac{C_2^\prime Lp(\theta_a^*-\theta_{b}^*)^2}{\beta\vee n^{-1} }}}{\frac{Lp(\theta_a^*-\theta_{b}^*)^2}{4}\sum_{k\neq a}\psi^{\prime}(\theta_a^*-\theta_k^*)}\leq \frac{C_4^\prime}{\sqrt{\frac{Lp\beta^2}{\beta\vee n^{-1} }}}.$$
Therefore, for an arbitrarily small constant $\delta>0$,  we have constant $C_5^\prime>0$, such that (\ref{eq:use-BTL-two-point}) can be lower bounded by
\begin{align}
\nonumber&C_3^{\prime}\exp\left(-\sqrt{\frac{C_2^\prime Lp(\theta_a^*-\theta_{b}^*)^2}{\beta\vee n^{-1} }}-(1+\delta_3^\prime)\frac{Lp(\theta_a^*-\theta_{b}^*)^2}{4}\sum_{k\neq a}\psi^{\prime}(\theta_a^*-\theta_k^*)\right)\\
\nonumber&\geq C_3^\prime\exp\left(-\left(1+\delta_3^\prime+\frac{C_4^\prime}{\sqrt{\frac{Lp\beta^2}{\beta\vee n^{-1} }}}\right)\frac{Lp(\theta_a^*-\theta_{b}^*)^2}{4V_a(\theta^*)}\right)\\
&\geq C_5^\prime\exp\left(-(1+\delta)\frac{Lp(\theta_a^*-\theta_{b}^*)^2}{4V_a(\theta^*)}\right).\label{eqn:noname7}
\end{align}
So far, we obtain
\begin{align}
\nonumber&\inf_{\wh{r}}\sup_{r^*\in\S_n}\E_{(\theta^*,r^*)}\left[\k(\wh{r},r)\right]\\
&\geq\frac{1}{2n}\sum_{a=1}^n\frac{2}{n(n-1)}\sum_{1\leq i<j\leq n}\sum_{b\in[n]\backslash\{a\}:\abs{a-b}\leq 1\vee \sqrt{\frac{C_1^\prime(\beta\vee n^{-1} )}{Lp\beta^2}}} C_5^\prime\exp\left(-(1+\delta)\frac{Lp(\theta_a^*-\theta_{b}^*)^2}{4V_a(\theta^*)}\right)\label{eqn:noname8}\\
\nonumber &\geq\frac{1}{2n}\sum_{a=1}^n\frac{2}{n(n-1)}\sum_{1\leq i<j\leq n}\sum_{b=a+1} C_5^\prime\exp\left(-(1+\delta)\frac{Lp(\theta_a^*-\theta_{b}^*)^2}{4V_a(\theta^*)}\right) \\
\nonumber &\geq \frac{C_5^\prime}{2n}\sum_{a=1}^n \exp\left(-(1+\delta)\frac{Lp(\theta_a^*-\theta_{a+1}^*)^2}{4V_a(\theta^*)}\right).
\end{align}
Hence, we obtain the exponential rate. 

In the following we are going to derive  the polynomial rate for the regime $\frac{Lp\beta^2}{\beta\vee n^{-1}}\leq 1$. 
Note that for any $a,b\in[n]$ such that $\abs{a-b}\leq 1\vee \sqrt{\frac{C_1^\prime(\beta\vee n^{-1})}{Lp\beta^2}}$, we have
\begin{align*}
\frac{Lp(\theta_a^*-\theta_{b}^*)^2}{4V_a(\theta^*)} \lesssim \frac{Lp\beta^2}{n \wedge \beta^{-1}} \br{1\vee \frac{\beta \vee n^{-1}}{Lp\beta^2}} \lesssim 1.
\end{align*}
Then from (\ref{eqn:noname8}), there exist some constant $C_6^\prime,C_7^\prime>0$ such that
\begin{align*}
\inf_{\wh{r}}\sup_{r^*\in\S_n}\E_{(\theta^*,r^*)}\left[\k(\wh{r},r)\right]&\geq\frac{1}{2n}\sum_{a=1}^n\frac{2}{n(n-1)}\sum_{1\leq i<j\leq n}\sum_{b\in[n]\backslash\{a\}:\abs{a-b}\leq 1\vee \sqrt{\frac{C_1^\prime(\beta\vee n^{-1})}{Lp\beta^2}}} C_6^\prime\\
&\geq \frac{C_6^\prime}{2}  \br{1\vee \sqrt{\frac{C_1^\prime(\beta\vee n^{-1})}{Lp\beta^2}}} \\
& \geq C_7^\prime  \left(n\wedge\sqrt{\frac{\beta\vee n^{-1}}{Lp\beta^2}}\right),
\end{align*}
where the last inequality is due to $\frac{Lp\beta^2}{\beta\vee n^{-1}}\leq 1$ and the fact that the loss is at most $n$. 
\end{proof}

\subsection{Proof of Theorem \ref{thm:alg-1}}\label{sec:pf-alg-1}

The following two lemmas are needed for the proof of Theorem \ref{thm:alg-1}.
Recall that $\bar{y}_{ij}^{(1)}=\frac{1}{L_1}\sum_{l=1}^{L_1}y_{ijl}, i\neq j\in[n]$.
\begin{lemma}\label{lem:ave-concentration}
There exists a constant $C_1>0$ such that for any $\theta^*\in\Theta_n(\beta, C_0)$ and $r^*\in\S_n$, 
$$\max_{i\in[n], j\in[n], i\neq j}\abs{\bar{y}_{ij}^{(1)}-\psi(\theta_{r_i^*}^*-\theta_{r_j^*}^*)}\leq C_1\sqrt{\frac{\log n}{L_1}}$$
holds with probability at least $1-O(n^{-10})$.
\end{lemma}
\begin{proof}
This can be seen directly by standard Hoeffding's inequality and union bound argument.
\end{proof}

\begin{lemma}\label{cor:ave-con-cor}
For $L_1$ such that $\frac{L_1}{\log n}\to\infty$ and constant $M\geq1$, there exists $0<\delta_0=o(1)$ and $0<\delta_1=o(1)$ such that for any $\theta^*\in\Theta_n(\beta, C_0)$, any $r^*\in\S_n$, 
$$\max_{i\in[n], j\in[n], i\neq j}\abs{\bar{y}_{ij}^{(1)}-\psi(\theta_{r_i^*}^*-\theta_{r_j^*}^*)}\leq \delta_0$$
and
$$\bigcap_{i=1}^n\left\{\underline{\mathcal{E}_{1,i}}\subset\mathcal{E}_{1,i}\subset\overline{\mathcal{E}_{1,i}}\right\}$$
hold with probability at least $1-O(n^{-10})$, where
\begin{eqnarray*}
\mathcal{E}_{1,i} &=& \left\{j\in[n]: \bar{y}_{ij}^{(1)}\leq\psi(-2M)\right\}, \\
\underline{\mathcal{E}_{1,i}} &=& \left\{j\in[n]: \theta_{r_j^*}^*\geq\theta_{r_i^*}^*+2M+\delta_1\right\}, \\
\overline{\mathcal{E}_{1,i}} &=& \left\{j\in[n]: \theta_{r_j^*}^*\geq\theta_{r_i^*}^*+2M-\delta_1\right\}.
\end{eqnarray*}
\end{lemma}
\begin{proof}
This is a direct consequence of Lemma \ref{lem:ave-concentration} and $M=O(1)$.
\end{proof}

Now we are ready to prove Theorem \ref{thm:alg-1}.
\begin{proof}[Proof of Theorem \ref{thm:alg-1}]

Let $\mathcal{F}^{(0)}$ be the event on which Lemma \ref{cor:ave-con-cor} holds. We will always work on this high probability event throughout the proof. Also, we will assume the regime $n\beta\to\infty$. The case $\beta\lesssim1/n$ is trivial since we only have one league $S_1=[n]$ if $M$ is chosen to be a large enough constant. 

To start the exposition, we define a series of quantities iteratively  for all $k\in[K-1]$, with the base case $\underline{S_0}=\overline{S_0}=S_0^\prime=\wt{S}_0=\emptyset, \underline{u^{(0)}}=\overline{u^{(0)}}=0$. Let
$$\underline{t_i^{(k)}}=\abs{\left\{j\in[n]\backslash\wt{S}_{k-1}:j\in\underline{\mathcal{E}_{1,i}}\right\}},$$
$$ \overline{t_i^{(k)}}=\abs{\left\{j\in[n]\backslash\wt{S}_{k-1}:j\in\overline{\mathcal{E}_{1,i}}\right\}},$$
\begin{align}
& \underline{S_k}=\left\{i\in[n]\backslash\wt{S}_{k-1}:\left(1+\frac{0.11}{C_0^2}\right)p\overline{t_i^{(k)}}\leq h\right\},\label{eqn:def_underline_S_k} \\
&  \overline{S_k}=\left\{i\in[n]\backslash\wt{S}_{k-1}:\left(1-\frac{0.11}{C_0^2}\right)p\underline{t_i^{(k)}}\leq h\right\}, \label{eqn:def_overline_S_k}
\end{align}
$$\overline{u^{(k)}}=\max\left\{r_i^*:i\in[n]\backslash\wt{S}_{k-1}, \overline{t_i^{(k)}}\leq\frac{M}{\left(1-\frac{0.12}{C_0^2}\right)\beta}\right\},$$
$$\underline{u^{(k)}}=\max\left\{r_i^*:i\in[n]\backslash\wt{S}_{k-1}, \overline{t_i^{(k)}}\leq\frac{M}{\left(1+\frac{0.11}{C_0^2}\right)\beta}\right\},$$
$$w_i^{(k)\prime}=\sum_{j\in\underline{S_k}\cap\overline{\mathcal{E}_{1,i}}}A_{ij}\indc{j\in\mathcal{E}_{1,i}},$$
$$S_k^\prime = \left\{i\in[n]\backslash\wt{S}_{k-1}:w_i^{(k)\prime}\leq h\right\}, $$
$$\wt{S}_{k}=\uplus_{m=1}^{k}S_m^\prime.$$

We make several remarks about these definition. The above definitions have essentially constructed another partition $S_1^\prime, S_2^\prime, ...$ using $w_i^{(k)\prime}$ comparing to Algorithm \ref{alg:partition} using $w_i^{(k)}$. The relationship between $S_k$ and $S_k^\prime$ will be made clear during the exposition. In fact, they will be equal with high probability. We should keep in mind that the partition using $w_i^{(k)\prime}$ is not a bona fide one since the definition uses $\overline{\mathcal{E}_{1,i}}$ and $\underline{S_k}$ which involve the knowledge of $\theta^*$. However, this can be used in theoretical exploration. Our strategy is to show certain properties hold for partitions $S_k^\prime$, then $S_k=S_k^\prime$ with high probability and  thus inherits those properties. 

We start with some simple but crucial facts which will act as building blocks in the proof.
\begin{itemize}
\item
$\overline{t_i^{(k)}}$ and $\underline{t_i^{(k)}}$ has the following monotonicity property: for any $i,j\in[n]\backslash\wt{S}_{k-1}$ such that $r_i^*\leq r_j^*$, 
\begin{equation}
\overline{t_i^{(k)}}\leq\overline{t_j^{(k)}}, \underline{t_i^{(k)}}\leq\underline{t_j^{(k)}}.\label{eq:t-mono-k}
\end{equation}
This is direct from the definition.
\item
For any $i\in[n]\backslash\wt{S}_{k-1}$,
\begin{equation}
0\leq \overline{t_i^{(k)}}-\underline{t_i^{(k)}}\leq\frac{2\delta_1}{\beta}+1,\label{eq:t-upper-lower-diff-k}
\end{equation}
which comes from
$\overline{t_i^{(k)}}-\underline{t_i^{(k)}}=\abs{\left\{j\in[n]\backslash\wt{S}_{k-1}:2M-\delta_1\leq\theta_{r_j^*}^*-\theta_{r_i^*}^*<2M+\delta_1\right\}}.$
\item We have
\begin{equation}
\left\{i\in[n]\backslash\wt{S}_{k-1}:r_i^*\leq\underline{u^{(k)}}\right\}=\underline{S_k}\subset\overline{S_k}\subset\left\{i\in[n]\backslash\wt{S}_{k-1}:r_i^*\leq\overline{u^{(k)}}\right\}.\label{eq:S-u-relation-k}
\end{equation}
Here $\underline{S_k}\subset\overline{S_k}$ is due to monotonicity (\ref{eq:t-mono-k}) and $\underline{t_i^{(k)}}\leq\overline{t_i^{(k)}}$ by definition.  Recall $h=pM/\beta$.  Using (\ref{eq:t-upper-lower-diff-k}), for any $i\in \overline{S_k}$, we have $\overline{t_i^{(k)}} \leq \underline{t_i^{(k)}} + \frac{2\delta_1}{\beta}+1 \leq \frac{h}{\left(1-\frac{0.11}{C_0^2}\right)p}+ \frac{2\delta_1}{\beta}+1 \leq \frac{M}{\left(1-\frac{0.12}{C_0^2}\right)\beta}$. Hence, we have $\overline{S_k}\subset\left\{i\in[n]\backslash\wt{S}_{k-1}:r_i^*\leq\overline{u^{(k)}}\right\}$.
\item
$\underline{t_i^{(k)}}$, $\overline{t_i^{(k)}}$, $\underline{S_k}$, $\overline{S_k}, \underline{u^{(k)}}, \overline{u^{(k)}}$ are measurable with respect to the $\sigma$-algebra generated by $\wt{S}_{k-1}$. This is direct from the definition.
\end{itemize}
Now, we will prove the following statements by induction on $k$:
\begin{itemize}
\item
With probability at least $1-O(kn^{-10})$, 
\begin{equation}
\underline{S_{k^\prime}}\subset S_{k^\prime}\subset\overline{S_{k^\prime}}
\label{eq:u-contain-S-k}
\end{equation}
for all $0\leq k^\prime\leq k$.
\item
With probability at least $1-O(kn^{-10})$, 
\begin{equation}
\abs{\underline{S_{k^\prime}}}\geq\left(\frac{1.7}{C_0}+\frac{1}{1+\frac{0.11}{C_0^2}}\right)\frac{M}{\beta}\label{eq:S-k-card-lower}
\end{equation}
for all $1\leq k^\prime\leq k$ and $\abs{S_0}=0$.
\item
With probability at least $1-O(kn^{-10})$, 
\begin{equation}
\abs{\overline{S_{k^\prime}}\backslash\underline{S_{k^\prime}}}\leq\overline{u^{(k^\prime)}}-\underline{u^{(k^\prime)}}\leq\frac{0.29M}{C_0\beta}\label{eq:S-upper-lower-diff-k}
\end{equation}
for all $0\leq k^\prime\leq k$.
\item
With probability at least $1-O(kn^{-10})$, 
\begin{equation}
\abs{\overline{S_{k^\prime}}}\leq\left(2+\frac{0.29}{C_0}+\frac{1}{1-\frac{0.12}{C_0^2}}\right)\frac{M}{\beta}\label{eq:S-k-card-upper}
\end{equation}
for all $0\leq k^\prime\leq k$.
\item
With probability at least $1-O(kn^{-10})$,
\begin{equation}
S_{k^\prime}=S_{k^\prime}^\prime\label{eq:S-S-prime-equal-k}
\end{equation}
for all $0\leq k^\prime\leq k$.
\end{itemize}
Now, suppose (\ref{eq:u-contain-S-k}) - (\ref{eq:S-S-prime-equal-k}) hold until $k-1$, which is the case for $k=1$. 
 In the following, we are going to establish (\ref{eq:u-contain-S-k}) - (\ref{eq:S-S-prime-equal-k})  for $k$ one by one.

~\\
\emph{(Establishment of  (\ref{eq:u-contain-S-k})).}
Recall that we assume $\mathcal{F}^{(0)}$ holds. On the intersection of all high probability events before $k$, we have $\wt{S}_{k-1}=S_1\cup \ldots\cup  S_{k-1}$.  
We  sandwich $w_i^{(k)}$ by 
$$\underline{w_i^{(k)}}=\sum_{j\in[n]\backslash\wt{S}_{k-1}}A_{ij}\indc{j\in\underline{\mathcal{E}_{1,i}}}\leq w_i^{(k)}\leq\sum_{j\in[n]\backslash\wt{S}_{k-1}}A_{ij}\indc{j\in\overline{\mathcal{E}_{1,i}}}=\overline{w_i^{(k)}}.$$
Recall the definition of $S_k$ in Algorithm \ref{alg:partition}. Then we have $ S_k =   \left\{i\in[n]\backslash\wt{S}_{k-1}: w_i^{(k)}\leq h\right\}$.
Hence, $ \left\{i\in[n]\backslash\wt{S}_{k-1}: \overline{w_i^{(k)}}\leq h\right\} \subset S_k\subset  \left\{i\in[n]\backslash\wt{S}_{k-1}: \underline{w_i^{(k)}}\leq h\right\}$. To prove (\ref{eq:u-contain-S-k}), by the definitions in  (\ref{eqn:def_underline_S_k}) and (\ref{eqn:def_overline_S_k}), we only need to show
\begin{align*}
 &\left\{i\in[n]\backslash\wt{S}_{k-1}: \underline{w_i^{(k)}}\leq h\right\} \subset \left\{i\in[n]\backslash\wt{S}_{k-1}:\left(1-\frac{0.11}{C_0^2}\right)p\underline{t_i^{(k)}}\leq h\right\},\\
 &\left\{i\in[n]\backslash\wt{S}_{k-1}:\left(1+\frac{0.11}{C_0^2}\right)p\overline{t_i^{(k)}}\leq h\right\} \subset  \left\{i\in[n]\backslash\wt{S}_{k-1}: \overline{w_i^{(k)}}\leq h\right\},
\end{align*}
a sufficient condition of which is the following event:
\begin{align*}
&\mathcal{F}^{(k)}=\left\{\forall i\in[n]\backslash\wt{S}_{k-1} \text{ such that }p\underline{t_i^{(k)}}\leq  \frac{h}{2} :  \underline{w_i^{(k)}}\leq h\right\}\\
&\quad\quad\bigcap\left\{\forall i\in[n]\backslash\wt{S}_{k-1} \text{ such that }p\underline{t_i^{(k)}}>  \frac{h}{2} : \left(1-\frac{0.11}{C_0^2}\right)p\underline{t_i^{(k)}} \leq \underline{w_i^{(k)}}\right\} \\
&\quad\quad\bigcap\left\{\forall i\in[n]\backslash\wt{S}_{k-1} \text{ such that }p\overline{t_i^{(k)}}\leq  \frac{h}{2} :  \overline{w_i^{(k)}}\leq h\right\}\\
&\quad\quad\bigcap\left\{\forall i\in[n]\backslash\wt{S}_{k-1} \text{ such that }p\overline{t_i^{(k)}}>  \frac{h}{2} :  \overline{w_i^{(k)}}\leq\left(1+\frac{0.11}{C_0^2}\right)p\overline{t_i^{(k)}}\right\}.
\end{align*}
Hence to prove (\ref{eq:u-contain-S-k}), we only need to analyze $\pbr{\mathcal{F}^{(k)}}$.


Note that for any $j\in[n]\backslash\wt{S}_{k-1}$ we have  $r^*_j>\underline{u^{(k-1)}}$ according to the definition of $\underline{S_{k-1}}$ in  (\ref{eq:S-u-relation-k}).
Thus
$$\underline{w_i^{(k)}}=\sum_{\substack{j\in[n]\backslash\wt{S}_{k-1}\\r^*_j>\underline{u^{(k-1)}}}}A_{ij}\indc{\theta_j^*\geq\theta_i^*+2M+\delta_1},$$
$$\overline{w_i^{(k)}}=\sum_{\substack{j\in[n]\backslash\wt{S}_{k-1}\\r^*_j>\underline{u^{(k-1)}}}}A_{ij}\indc{\theta_j^*\geq\theta_i^*+2M-\delta_1}.$$
On the other hand, recall that $w_i^{(k-1)\prime}=\sum_{j\in \underline{S_{k-1}}\cap\overline{\mathcal{E}_{1,i}}}A_{ij}\indc{j\in\mathcal{E}_{1,i}}$ which only  involves $A_{ij}$ such that $r^*_j\leq \underline{u^{(k-1)}}$ due to (\ref{eq:S-u-relation-k}). By  (\ref{eq:S-u-relation-k}) and induction hypothesis of (\ref{eq:u-contain-S-k}) we further know $\underline{u^{(1)}}\leq ...\leq \underline{u^{(k-1)}}$. As a result, $w_i^{(1)\prime},...,w_i^{(k-1)\prime}$ are independent of $\underline{w_i^{(k)}}, \overline{w_i^{(k)}}$. Since $\wt{S}_{k-1}$ is determined by $w_i^{(1)\prime},...,w_i^{(k-1)\prime}$, it is also independent of $\underline{w_i^{(k)}}, \overline{w_i^{(k)}}$.


Therefore, conditional on $\wt{S}_{k-1}$,  we have 
$$\underline{w_i^{(k)}}|\wt{S}_{k-1}\sim\text{Binomial}(\underline{t_i^{(k)}}, p),$$
$$\overline{w_i^{(k)}}|\wt{S}_{k-1}\sim\text{Binomial}(\overline{t_i^{(k)}}, p).$$
Recall that $C_0\geq 1$ is a constant and  $h=pM/\beta \gg \log n$ since $p/(\beta \log n)\rightarrow\infty$ by assumption. By Bernstein inequality for the Binomial distributions together with a union bound argument, we have $\pbr{\mathcal{F}^{(k)}|\wt{S}_{k-1}} \geq 1-O(n^{-10}).$ Since this holds for all $\wt{S}_{k-1}$, we have
\begin{align*}
\pbr{\mathcal{F}^{(k)}} \geq 1-O(n^{-10}).
\end{align*}
Therefore, we have proved  (\ref{eq:u-contain-S-k}). 

~\\
\emph{(Establishment of (\ref{eq:S-k-card-lower})).} We first present a simple fact from induction hypothesis:
\begin{align}
\left\{i\in[n], r_i^*\leq\underline{u^{(k-1)}}\right\}\subset\wt{S}_{k-1}\subset\left\{i\in[n], r_i^*\leq\overline{u^{(k-1)}}\right\}.\label{eqn:tilde_S_sandwitch}
\end{align}
The first containment  is because (\ref{eq:S-u-relation-k}) and (\ref{eq:u-contain-S-k}) hold up to $k-1$. To prove the second containment, we only need to show $\overline{u^{(1)}} \leq \ldots \leq \overline{u^{(k-1)}}$.
Notice that from (\ref{eq:S-k-card-lower}) and (\ref{eq:S-upper-lower-diff-k}) for $k-1$, we have $\abs{\underline{S_{k-1}}} \geq \overline{u^{(k-2)}}-\underline{u^{(k-2)}}$. On the other hand, from (\ref{eq:u-contain-S-k}) for $k-1$, we have $\abs{\underline{S_{k-1}}} \leq \abs{ \cbr{i\in[n]: r_i^*> \underline{u^{(k-2)}}, r^*_i \leq \overline{u^{(k-1)}}}}  \leq \overline{u^{(k-1)}} -  \underline{u^{(k-2)}}$. Hence, we have $\overline{u^{(k-1)}}  \geq \overline{u^{(k-2)}}$ and similarly we can show  $\overline{u^{(l+1)}}  \geq \overline{u^{(l)}}$ for any $l\leq k-2$, which proves $\overline{u^{(1)}} \leq \ldots \leq \overline{u^{(k-1)}}$.

Using (\ref{eqn:tilde_S_sandwitch}), we have 
\begin{align*}
&\abs{\underline{S_k}}=\abs{\left\{i\in[n]\backslash\wt{S}_{k-1}: \overline{t_i^{(k)}}\leq\frac{M}{\left(1+\frac{0.11}{C_0^2}\right)\beta}\right\}}\\
&\geq\abs{\left\{i\in[n]: r_i^*>\overline{u^{(k-1)}}, \abs{\left\{j\in[n]: r_j^*>\underline{u^{(k-1)}},\theta_{r_j^*}^*\geq\theta_{r_i^*}^*+2M-\delta_1\right\}}\leq\frac{M}{\left(1+\frac{0.11}{C_0^2}\right)\beta}\right\}}.
\end{align*}
For any $i\in[n]$, since $\theta^*\in \Theta_n(\beta,C_0)$, we have
\begin{align*}
\abs{\left\{j\in[n]: r_j^*>\underline{u^{(k-1)}},\theta_{r_j^*}^*\geq\theta_{r_i^*}^*+2M-\delta_1\right\}}&\leq  r_i^* -\left\lfloor \frac{2M-\delta_1}{C_0\beta}\right\rfloor - \underline{u^{(k-1)}}.
\end{align*}
Hence,
\begin{align*}
\abs{\underline{S_k}} &\geq\abs{\left\{i\in[n]: r_i^*>\overline{u^{(k-1)}}, r_i^*\leq\frac{M}{\left(1+\frac{0.11}{C_0^2}\right)\beta}+ \left\lfloor \frac{2M-\delta_1}{C_0\beta}\right\rfloor + \underline{u^{(k-1)}} \right\} } \\
& \geq  \frac{M}{\left(1+\frac{0.11}{C_0^2}\right)\beta}+ \left\lfloor \frac{2M-\delta_1}{C_0\beta}\right\rfloor + \underline{u^{(k-1)}}  - \overline{u^{(k-1)}}\\
&\geq\left(\frac{1.7}{C_0}+\frac{1}{1+\frac{0.11}{C_0^2}}\right)\frac{M}{\beta}.
\end{align*}

~\\
\emph{(Establishment of (\ref{eq:S-upper-lower-diff-k})).} From (\ref{eq:u-contain-S-k}), we have $\abs{\overline{S_k}\backslash\underline{S_k}}\leq\overline{u^{(k)}}-\underline{u^{(k)}}$. Hence, we only need to show $\overline{u^{(k)}}-\underline{u^{(k)}}\leq\frac{0.29M}{C_0\beta}$. 

We are going to prove 
\begin{align}
\theta^*_{ \overline{u^{(k-1)}}} \geq  \theta^*_{ \underline{u^{(k)}}} + 2M -\delta_1.\label{eqn:noname1}
\end{align}
First, by (\ref{eq:S-upper-lower-diff-k}) for $k-1$, (\ref{eq:S-k-card-lower}), and (\ref{eq:S-u-relation-k}), we have  $\abs{\left\{i\in [n]: \underline{u^{(k-1)}}\leq r_i^*\leq\underline{u^{(k)}}\right\}}\geq |\underline{S_k}|$ which leads to $\underline{u^{(k)}} \geq \overline{u^{(k-1)}}$. Let $b\in[n]$ be the index such that $r^*_b = \underline{u^{(k)}} + 1 $. Then it means $b \in [n]\backslash\wt{S}_{k-1}$ and $\overline{t^{(k)}_{b}}> \frac{M}{\left(1+\frac{0.11}{C_0^2}\right)\beta}$. By the definition of $\overline{t^{(k)}_i}$, for any $i\in[n]\backslash\wt{S}_{k-1}$, we have
\begin{align*}
\abs{\left\{j\in[n]: r^*_j \geq \underline{u^{(k-1)}},  \theta^*_{r^*_j} > \theta^*_{r^*_i} + 2M -\delta_1\right\}} & \geq \abs{\left\{j\in[n]\backslash\wt{S}_{k-1}:j\in\overline{\mathcal{E}_{1,i}}\right\}} = \overline{t^{(k)}_i},
\end{align*}
which implies $\theta^*_{\underline{u^{(k-1)}} +  \overline{t^{(k)}_i}} >\theta^*_{r^*_i} + 2M -\delta_1  $. This means
\begin{align*}
\theta^*_{\underline{u^{(k-1)}} } >\theta^*_{r^*_i} + 2M -\delta_1  +\overline{t^{(k)}_i}\beta.
\end{align*}
Considering the $b$ index here, we have
\begin{align}
 \theta^*_{ \underline{u^{(k-1)}}}   \geq  \theta^*_{\underline{u^{(k)}} + 1} + 2M -\delta_1 +\frac{M}{\left(1+\frac{0.11}{C_0^2}\right)}. \label{eqn:noname2}
\end{align}
Then using (\ref{eq:S-upper-lower-diff-k}) for $k-1$, we have
\begin{align}
 \theta^*_{ \overline{u^{(k-1)}}}   \geq  \theta^*_{\underline{u^{(k)}} + 1} + 2M -\delta_1 +\frac{M}{\left(1+\frac{0.11}{C_0^2}\right)} - 0.29M \geq \theta^*_{\underline{u^{(k)}} } + 2M -\delta_1,\label{eqn:noname3}
\end{align}
which proves (\ref{eqn:noname1}). Then for any $i,j\in[n]\backslash\wt{S}_{k-1}$ such that $\underline{u^{(k)}}\leq  r_i^*< r_j^*$, we have
\begin{align*}
\overline{t_j^{(k)}} - \overline{t_i^{(k)}} & =  \abs{\left\{l\in[n]\backslash\wt{S}_{k-1}:\theta^*_{r^*_l} \geq \theta^*_{r^*_j} + 2M -\delta_1\right\}} -  \abs{\left\{l\in[n]\backslash\wt{S}_{k-1}:\theta^*_{r^*_l} \geq \theta^*_{r^*_i} + 2M -\delta_1\right\}} \\
& =  \abs{\left\{l\in[n]\backslash\wt{S}_{k-1}: \theta^*_{r^*_i} + 2M -\delta_1 >\theta^*_{r^*_l} \geq \theta^*_{r^*_j} + 2M -\delta_1\right\}} \\
& \geq \abs{\left\{r_l^*\geq \overline{u^{(k-1)}}: \theta^*_{r^*_i} + 2M -\delta_1 >\theta^*_{r^*_l} \geq \theta^*_{r^*_j} + 2M -\delta_1\right\}}\\
&\geq \abs{\left\{l\in[n]: \theta^*_{r^*_i} + 2M -\delta_1 >\theta^*_{r^*_l} \geq \theta^*_{r^*_j} + 2M -\delta_1\right\}}\\
& \geq \frac{\theta^*_{r^*_i} -\theta^*_{r^*_j}  }{C_0 \beta}\\
&\geq \frac{r^*_j - r^*_i}{C_0},
\end{align*}
where in the first inequality we use (\ref{eqn:tilde_S_sandwitch}) and in the second  inequality we use  (\ref{eqn:noname1}).  The last two inequalities are due to $\theta^*\in \Theta_n(\beta,C_0)$. As a result,
\begin{align}\label{eqn:noname5}
\overline{u^{(k)}}-\underline{u^{(k)}}  \leq  \frac{\frac{M}{\left(1-\frac{0.12}{C_0^2}\right)\beta} - \frac{M}{\left(1+\frac{0.11}{C_0^2}\right)\beta}}{C_0} \leq \frac{0.29M}{C_0\beta}.
\end{align}

~\\
\emph{(Establishment of (\ref{eq:S-k-card-upper})).}
We first have
$$\abs{\overline{S_k}}\leq\overline{u^{(k)}}-\underline{u^{(k-1)}}\leq\overline{u^{(k)}}-\left(\overline{u^{(k-1)}}-\frac{0.29M}{C_0\beta}\right)$$
due to induction hypothesis on (\ref{eq:S-upper-lower-diff-k}) for $k-1$ and $\left\{i\in[n]: r_i^*\leq\underline{u^{(k-1)}}\right\}\subset\wt{S}_{k-1}$. By the definition of $\overline{u^{(k)}}$, similar to the proof of (\ref{eqn:noname2}), we can show
\begin{align*}
\theta^*_{\overline{u^{(k-1)}}} \leq \theta_{\overline{u^{(k)}}}^* + 2M -\delta_1 +  \frac{M}{\left(1-\frac{0.12}{C_0^2}\right)}
\end{align*}
which implies
$$\overline{u^{(k)}}-\overline{u^{(k-1)}}\leq\frac{2M}{\beta}+\frac{M}{\left(1-\frac{0.12}{C_0^2}\right)\beta}.$$
Therefore,
$$\abs{\overline{S_k}}\leq\left(2+\frac{0.29}{C_0}+\frac{1}{1-\frac{0.12}{C_0^2}}\right)\frac{M}{\beta}.$$

~\\
\emph{(Establishment of (\ref{eq:S-S-prime-equal-k})).}
Define
$$\mathcal{F}^{(k)\prime}=\left\{\min_{i\in[n]:r_i^*>\overline{u^{(k)}}}\sum_{j\in\underline{S_k}}A_{ij}\indc{j\in\underline{\mathcal{E}_{1,i}}}>h\right\}.$$
We are going to show the event $\mathcal{F}^{(k)\prime}$ is a sufficient condition for (\ref{eq:S-S-prime-equal-k}). By definition, since $\underline{S_k} \subset [n]\backslash \wt{S}_{k-1}$, we have  $w_i^{(k)\prime}\leq w_i^{(k)}$ which implies $S_k\subset S_k^\prime$. We only need to show $S_k^\prime\subset S_k$. Note that for any $i$ such that $r_i^*>\overline{u^{(k)}}$, we have
\begin{align*}
w_i^{(k)\prime} = \sum_{j\in\underline{S_k}}A_{ij}\indc{j\in\mathcal{E}_{1,i}} \geq \sum_{j\in\underline{S_k}}A_{ij}\indc{j\in\underline{\mathcal{E}_{1,i}}}>h,
\end{align*}
which means $i\notin S_k^\prime$ as $\mathcal{F}^{(k)\prime}$ is assumed to be true. Hence to show $S_k^\prime\subset S_k$, we only need to show $S_k^\prime \cap \{i\in[n]: r_i^*\leq \overline{u^{(k)}}\}\subset S_k$. 
Note that due to (\ref{eq:S-upper-lower-diff-k}), for any $i,j\in[n]$, such that $r_i^*\leq \overline{u^{(k)}}$ and $r_j^*> \underline{u^{(k)}}$, we have $r_j^*>\underline{u^{(k)}}$, $\theta_{r_i^*}^*-\theta_{r_j^*}^*\geq \theta_{\overline{u^{(k)}}}^*-\theta_{\underline{u^{(k)}}}^*\geq -0.29M$. Then for any $i$ such that $r_i^*\leq \overline{u^{(k)}}$, we have
\begin{align*}
w_i^{(k)\prime} - w_i^{(k)} &=  \sum_{j\in\underline{S_k}}A_{ij}\indc{j\in\mathcal{E}_{1,i}} -  \sum_{j\in[n]\backslash \wt{S}_{k-1}}A_{ij}\indc{j\in\mathcal{E}_{1,i}}  \\
&\geq - \sum_{j\in[n]:r_j^*> \underline{u^{(k)}}}\indc{j\in\mathcal{E}_{1,i}}\\
& \geq  - \sum_{j\in[n]:r_j^*> \underline{u^{(k)}}}\indc{j\in\overline{\mathcal{E}_{1,i}}}\\
& =  - \sum_{j\in[n]:r_j^*> \underline{u^{(k)}}} \indc{\theta_{r_j^*}^*\geq\theta_{r_i^*}^*+2M-\delta_1} \\
& = 0,
\end{align*}
where  first inequality is due to (\ref{eq:S-u-relation-k}). Hence we have $S_k^\prime \cap \{i\in[n]: r_i^*\leq \overline{u^{(k)}}\}\subset S_k$ which leads to $S_k = S_k^\prime$. As a result, to establish (\ref{eq:S-S-prime-equal-k}), we only need to analyze $\pbr{\mathcal{F}^{(k)\prime}}$.

The analysis of  $\pbr{\mathcal{F}^{(k)\prime}}$ is similar to that of $\pbr{\mathcal{F}^{(k)}}$ in the establishment of  (\ref{eq:u-contain-S-k}). By a similar independence argument, we have 
\begin{align*}
\br{\sum_{j\in\underline{S_k}}A_{ij}\indc{j\in\underline{\mathcal{E}_{1,i}}}}\Bigg| \wt{S}_{k-1} \sim\text{Binomial}\left(\abs{\underline{S_k}\cap\underline{\mathcal{E}_{1,i}}},p\right)
\end{align*}
for any $i\in[n]$ such that $r_i^*>\overline{u^{(k)}}$. From (\ref{eqn:noname3}), we have
\begin{align}
\underline{u^{(k)}} - \overline{u^{(k-1)}} \geq  \frac{2M - \delta_1}{C_0 \beta}.\label{eqn:noname4}
\end{align}
Together with (\ref{eq:S-u-relation-k}) and (\ref{eqn:tilde_S_sandwitch}), we have 
\begin{align*}
\abs{\underline{S_k}\cap\underline{\mathcal{E}_{1,i}}} &\geq \abs{\cbr{j\in[n]: \overline{u^{(k-1)}} \leq r_j^*  \leq \underline{u^{(k)}} , \theta_{r_j^*}^*\geq\theta_{r_i^*}^*+2M+\delta_1}}  \geq \underline{u^{(k)}} - \overline{u^{(k-1)}} \geq  \frac{2M - \delta_1}{C_0 \beta}.
\end{align*}
Recall that $h=pM/\beta$ and $p/(\beta \log n)\rightarrow\infty$. By Bernstein inequality, we have
\begin{align*}
\pbr{\mathcal{F}^{(k)\prime} | \wt{S}_{k-1}} = \pbr{\min_{i\in[n]:r^*_i \geq \overline{u^{(k)}}}\br{\sum_{j\in\underline{S_k}}A_{ij}\indc{j\in\underline{\mathcal{E}_{1,i}}}} > h\Bigg| \wt{S}_{k-1}} \geq 1-O(n^{-10}).
\end{align*}
Since this holds for all $\wt{S}_{k-1}$, we have $\pbr{\mathcal{F}^{(k)\prime}} \geq 1-O(n^{-10})$.

~\\
\emph{(Establishment of (\ref{eq:u-contain-S-k}) - (\ref{eq:S-S-prime-equal-k}) for $K$).}
We have (\ref{eq:u-contain-S-k}) - (\ref{eq:S-S-prime-equal-k}) hold for each $k\in[K-1]$ with probability at least $1-O(n^{-9})$. For the last partition, $S_K=[n]\backslash\wt{S}_{K-1}$. Let $S_{K,1}$ be the set obtained by Algorithm \ref{alg:partition} before the terminating condition $[n] - \abs{S_1}+...+\abs{S_{K,1}}\leq\abs{S_{K,1}}/2$ is met. $\underline{S_{K,1}}, \overline{S_{K,1}}$ can be similarly defined and (\ref{eq:u-contain-S-k}) - (\ref{eq:S-S-prime-equal-k}) should also be satisfied by $S_{K,1}$. Therefore, 
$$\abs{S_K}\leq\frac{3\abs{S_{K, 1}}}{2}\leq\frac{3}{2}\abs{\overline{S_{K, 1}}}\leq\frac{3}{2}\left(2+\frac{0.29}{C_0}+\frac{1}{1-\frac{0.12}{C_0^2}}\right)\frac{M}{\beta},$$
$$\abs{S_K}\geq\abs{S_{K, 1}}\geq\abs{\underline{S_{K,1}}}\geq\left(\frac{1.7}{C_0}+\frac{1}{1+\frac{0.11}{C_0^2}}\right)\frac{M}{\beta}$$
and
$$\left\{i\in[n]: r_i^*>\overline{u^{(K-1)}}\right\}\subset S_K\subset\left\{i\in[n]: r_i^*>\underline{u^{(K-1)}}\right\}.$$

~\\
\indent So far, we have establish (\ref{eq:u-contain-S-k}) - (\ref{eq:S-S-prime-equal-k}) for any $k\in[K]$.
Now we are ready to use them to prove the conclusions in Theorem \ref{thm:alg-1}.
\begin{enumerate}
\item Conclusion 1  is a consequence of (\ref{eq:u-contain-S-k}) and (\ref{eq:S-k-card-upper}).

\item For Conclusion 2, by (\ref{eqn:noname4}) we have $\overline{u^{(k-2)}} < \underline{u_{(k-1)}} < \overline{u^{(k-1)}} < \underline{u_{(k)}} < \overline{u^{(k)}} < \underline{u_{(k+1)}} $. Together with (\ref{eq:u-contain-S-k}) and (\ref{eqn:tilde_S_sandwitch}), we have 
$$\left\{i\in[n]: \overline{u^{(k-2)}}<r_i^*\leq\underline{u^{(k+1)}}\right\}\subset S_{k-1}\cup S_k\cup S_{k + 1}\subset\left\{i\in[n]: \underline{u^{(k-2)}}<r_i^*\leq\overline{u^{(k+1)}}\right\}.$$
Therefore, using (\ref{eqn:noname4}), for any $i$ such that $\underline{u^{(k-1)}}<r_i^*\leq\overline{u^{(k)}}$, 
\begin{align*}
&\left\{j\in[n]: \abs{r_i^*-r_j^*}\leq\frac{1.51M}{C_0\beta}\right\}\\
&\subset\left\{j\in[n]: \underline{u^{(k-1)}}-\frac{1.51M}{C_0\beta}\leq r_j^*\leq\overline{u^{(k)}}+\frac{1.51M}{C_0\beta}\right\}\\
&\subset\left\{j\in[n]: \overline{u^{(k-2)}}<r_j^*\leq\underline{u^{(k+1)}}\right\}\subset S_{k-1}\cup S_k\cup S_{k + 1}.
\end{align*}
For $k=1$ or $K$, only oneside needs to be considered and the property still holds due to the gap between $\underline{u^{(2)}}$ and $\overline{u^{(1)}}$ as well as the gap between $\underline{u^{(K-1)}}$ and $\overline{u^{(K-2)}}$.

\item For Conclusion 3, by (\ref{eq:u-contain-S-k}) and (\ref{eqn:tilde_S_sandwitch}), we have 
\begin{align}
\left\{i\in[n]: \overline{u^{(k-1)}}<r_i^*\leq\underline{u^{(k)}}\right\}\subset S_k\subset\left\{i\in[n]: \underline{u^{(k-1)}}<r_i^*\leq\overline{u^{(k)}}\right\}.\label{eqn:noname6}
\end{align}
Using (\ref{eqn:noname4}), we have 
$$\max\left\{r_i^*: i\in S_k\right\}\leq\overline{u^{(k)}}<\underline{u^{(k+1)}}<\min\left\{r_i^*: i\in S_{k+2}\right\}.$$
Same results can be established for $\max\left\{r_i^*: i\in S_k\right\}<\min\left\{r_i^*: i\in S_{l}\right\}$ for any $l>k+2$.

\item For Conclusion 4, for any $k$ and any $i$, the definition of $w_i^{(k)\prime}$ only involves $j$ such that $j\in\overline{\mathcal{E}_{1,i}}$. This implies that the definition of $S_k^\prime$ only involves information of $(A_{ij}, \bar{y}_{ij}^{(1)})$ such that $\theta_{r_j^*}^*-\theta_{r_i^*}^*\geq2M-\delta_1$. Thus $S_k^\prime$ can be used as the $\check{S}_k$ in Theorem \ref{thm:alg-1}.

\item For Conclusion 5,  note that  for any $k\in[K]$ and $i\in S_k$, we have
\begin{align*}
\abs{\left\{j\in[n]: |\theta^*_{r_i^*}-\theta^*_{r_j^*}|\leq \frac{M}{2}\right\} \cap S_k} &\geq  \abs{\left\{j\in[n]: |r_i^*-r_j^*|\leq \frac{M}{2C_0\beta}\right\} \cap S_k} \\
& \geq \abs{\left\{j\in[n]: |r_i^*-r_j^*|\leq \frac{M}{2C_0\beta} ,  \overline{u^{(k-1)}}<r_j^*\leq\underline{u^{(k)}}\right\}}.
\end{align*}
where the last inequality is by (\ref{eqn:noname6}). Again by (\ref{eqn:noname6}), we have $\underline{u^{(k-1)}}<r_i^*\leq\overline{u^{(k)}}$.
From   (\ref{eqn:noname4}) we know $\underline{u^{(k)}} - \overline{u^{(k-1)}} > M/(2C_0\beta)$.  Then we have
\begin{align*}
\abs{\left\{j\in[n]: |\theta^*_{r_i^*}-\theta^*_{r_j^*}|\leq \frac{M}{2}\right\} \cap S_k}  &\geq \frac{M}{2C_0\beta} - \max\cbr{\overline{u^{(k-1)}} - \underline{u^{(k-1)}}, \overline{u^{(k)}} - \underline{u^{(k)}}}\\
&\geq  \frac{0.21M}{C_0\beta}
\end{align*}
where the last inequality is  by (\ref{eq:S-upper-lower-diff-k}).
\end{enumerate}
The proof is complete.
\end{proof}

\subsection{Proofs of Lemma \ref{lem:anderson-ineq}, Lemma \ref{lem:MLE-check} and Lemma \ref{lem:prev-paper}}\label{sec:pf-lemmas}

We first prove Lemma \ref{lem:anderson-ineq} below.
\begin{proof}[Proof of Lemma \ref{lem:anderson-ineq}]
Recall that $\wh{r}$ is obtained by sorting $\left\{\sum_{j\in[n]\backslash\{i\}}R_{ij}\right\}_{i\in[n]}$. Define 
$$\wh{s}_i=\sum_{j\in[n]\backslash\{i\}}R_{ij},$$
$$\wh{R}_{ij}=\indc{\wh{s}_i>\wh{s}_j}=\indc{\wh{r}_i<\wh{r}_j}$$
and
$$s_i^*=\sum_{j\in[n]\backslash\{i\}}R_{ij}^*.$$
Observe that
$$r_i^*=n-s_i^*,$$
we have
$$\wh{R}_{ij}=\indc{\wh{s}_i>\wh{s}_j}=\indc{\wh{s}_i-s_i^*+s_i^*>\wh{s}_j-s_j^*+s_j^*}=\indc{\wh{s}_i-s_i^*-(\wh{s}_j-s_j^*)>r_i^*-r_j^*}.$$
Thus
\begin{align*}
&\k(\wh{r}, r^*)=\frac{1}{n}\sum_{1\leq i<j\leq n}\indc{\wh{R}_{ij}\neq R_{ij}^*}\\
&=\frac{1}{n}\sum_{1\leq i<j\leq n}\abs{\indc{\wh{s}_i-s_i^*-(\wh{s}_j-s_j^*)>r_i^*-r_j^*}-\indc{0>r_i^*-r_j^*}}\\
&\leq\frac{1}{n}\sum_{1\leq i<j\leq n}\indc{\abs{\wh{s}_i-s_i^*-(\wh{s}_j-s_j^*)}\geq\abs{r_i^*-r_j^*}}\\
&\leq\frac{1}{n}\sum_{1\leq i<j\leq n}\indc{\abs{\abs{r_i^*-r_j^*}\leq\abs{\wh{s}_i-s_i^*}+\abs{\wh{s}_j-s_j^*}}}\\
&=\frac{1}{n}\sum_{k=1}^{n-1}\sum_{\substack{1\leq i<j\leq n\\\abs{r_i^*-r_j^*}=k}}\indc{k\leq\abs{\wh{s}_i-s_i^*}+\abs{\wh{s}_j-s_j^*}}\\
&\leq\frac{1}{n}\sum_{k=1}^{n-1}\sum_{\substack{1\leq i<j\leq n\\\abs{r_i^*-r_j^*}=k}}\indc{\frac{k}{2}\leq\abs{\wh{s}_i-s_i^*}}+\indc{\frac{k}{2}\leq\abs{\wh{s}_j-s_j^*}}\\
&\leq\frac{2}{n}\sum_{i=1}^{n}\sum_{k=1}^{n-1}\indc{\frac{k}{2}\leq\abs{\wh{s}_i-s_i^*}}\leq\frac{4}{n}\sum_{i=1}^{n}\sum_{k=1}^n\indc{k\leq\abs{\wh{s}_i-s_i^*}}\\
&=\frac{4}{n}\sum_{i=1}^n\abs{\wh{s}_i-s_i^*}\leq\frac{4}{n}\sum_{i=1}^n\sum_{j\in[n]\backslash\{i\}}\abs{R_{ij}-R_{ij}^*}=\frac{4}{n}\sum_{1\leq i\neq j\leq n}\indc{R_{ij}\neq R_{ij}^*}
\end{align*}
which completes the proof.
\end{proof}

Next, we prove Lemma \ref{lem:MLE-check}.
\begin{proof}[Proof of Lemma \ref{lem:MLE-check}]
Recall $\mathcal{E}=\left\{(i,j): 1\leq i<j\leq n, \psi(-M)\leq\overline{y}_{ij}^{(1)}\leq\psi(M)\right\}$. Then on the event where Lemma \ref{cor:ave-con-cor} holds, $\mathcal{E}$ can be written as 
$$\mathcal{E}=\left\{(i,j):1\leq i<j\leq n, \abs{\theta_{r_i^*}-\theta_{r_j^*}}\leq M/2\right\}\uplus\left(\mathcal{E}\cap\left\{(i,j):M/2<\abs{\theta_{r_i^*}-\theta_{r_j^*}}< 1.1M\right\}\right)$$
which implies $\check{A}_{ij}=A_{ij}\indc{(i,j)\in\mathcal{E}}$. Moreover, on the event where Theorem \ref{thm:alg-1} holds, $\check{S}_k=S_k, k\in[K]$ by Conclusion 4. This proves $\ell^{(k)}(\theta)=\check{\ell}^{(k)}(\theta)$ with probability at least $1- O(n^{-8})$. $\check{\theta}^{(k)}$ and $\wh{\theta}^{(k)}$ are equivalent up to a common shift since the Hessian in local MLE is well conditioned with probability at least $1- O(n^{-8})$ due to Lemma \ref{lem:bound-hessian}. 
\end{proof}

Finally, we need to prove Lemma \ref{lem:prev-paper}, which requires us to first establish a few extra lemmas.

\begin{lemma}\label{lem:band_laplacian_eigenvalue}
For any integer constant $C\geq1$, define a matrix $M\in\cbr{0,1}^{n\times n}$ such that $M_{ij}=\indc{\abs{i-j}\leq n/C}$. Let $\mathcal{L}_M$ be its Laplacian matrix such that
\begin{align*}
\sbr{\mathcal{L}_M}_{ij} =\begin{cases}
-M_{ij},\quad \text{if }i\neq j\\
\sum_{l}M_{il},\quad \text{if }i= j.
\end{cases}
\end{align*}
Denote $\lambda_{\min,\perp}(\mathcal{L}_M)$ to be the second smallest eigenvalue of $\mathcal{L}_M$, i.e., $\lambda_{\min,\perp}(\mathcal{L}_M)=\min_{u\neq 0, \mathds{1}_n^Tu=0}\frac{u^T\mathcal{L}_M u}{\norm{u}}$. Then there exists another constant $C'>0$ that only depends on $C$ such that
\begin{align*}
\frac{1}{n}\lambda_{\min,\perp}(\mathcal{L}_M) =\inf_{\substack{x\in\mathbb{R}^n\\\mathds{1}_n^Tx=0\\\norm{x}=1}}\frac{\sum_{\abs{i-j}\leq\frac{n}{C}}(x_i-x_j)^2}{n} \geq  C'.
\end{align*}
\end{lemma}
\begin{proof}
We partition $[n]$ into $4C$ consecutive blocks such that each block contains either $\lceil n/4C\rceil$ or $\lfloor n/4C\rfloor$ consecutive indices. 
Let these blocks be a sequence of disjoint sets $B_1,...,B_{4C}$ such that $\max_{i\in B_k} i <\min_{j\in B_l} j$ if $k<l$. 
The idea is to lower bound the summation over the diagonal region by a sequence of square regions. 
Thus, for any $x\in\mathbb{R}^n, \mathds{1}_n^Tx=0, \norm{x}=1$, we have
\begin{align*}
\frac{1}{n}x^T\mathcal{L}_M x& = \frac{\sum_{\abs{i-j}\leq\frac{n}{C}}(x_i-x_j)^2}{n} \\
&\geq\frac{1}{n}\left[\sum_{k,l\in[4C]:\abs{k-l}\leq 1}\sum_{i\in B_{k}, j\in B_{l}}(x_i-x_j)^2\right]\\
& = \sum_{k,l\in[4C]:\abs{k-l}\leq 1} \br{\frac{\abs{B_l}}{n} \sum_{i\in B_k} x_i^2 + \frac{\abs{B_k}}{n} \sum_{i\in B_l} x_i^2 - 2\br{\sum_{i\in B_k} \frac{x_i}{\sqrt{n}}}\br{\sum_{i\in B_l} \frac{x_i}{\sqrt{n}}} }\\
& =  \sum_{k,l\in[4C]:\abs{k-l}\leq 1} \br{p_l z_k+ p_k z_l -2y_ky_l},
\end{align*}
where we denote
$$y_k=\sum_{i\in B_k}\frac{x_i}{\sqrt{n}}, z_k=\sum_{i\in B_k}x_i^2, p_k=\frac{\abs{B_k}}{n}.$$

For any $k\in[4C]$, we define
\begin{align}\label{eqn:w_def}
w_{2k-1} = \frac{y_k +\sqrt{p_k z_k - y_k^2}}{2p_k},\text{ and }w_{2k} = \frac{y_k -\sqrt{p_k z_k - y_k^2}}{2p_k}.
\end{align}
Note that for any $k,l\in[4C]$, we have
\begin{align*}
&p_l p_k \br{\br{w_{2k-1} - w_{2l-1}}^2 + \br{w_{2k-1} - w_{2l}}^2 + \br{w_{2k} - w_{2l-1}}^2 + \br{w_{2k} - w_{2l}}^2} \\
&= p_l p_k \br{2\br{w_{2k-1}^2 + w_{2k}^2 + w_{2l-1}^2 + w_{2k}^2} - 2\br{w_{2k-1} + w_{2k}}\br{w_{2l-1} + w_{2l}} } \\
& =  p_l p_k  \br{\frac{z_k}{p_k} + \frac{z_l}{p_l} -2 \frac{y_ky_l}{p_kp_l}}\\
&= p_l z_k + p_k z_l -2y_ky_l.
\end{align*}
Then we have
\begin{align*}
 &\sum_{k,l\in[4C]:\abs{k-l}\leq 1} \br{p_l z_k+ p_k z_l -2y_ky_l}\\
  &=   \sum_{k,l\in[4C]:\abs{k-l}\leq 1}   p_l p_k \br{\br{w_{2k-1} - w_{2l-1}}^2 + \br{w_{2k-1} - w_{2l}}^2 + \br{w_{2k} - w_{2l-1}}^2 + \br{w_{2k} - w_{2l}}^2}.
\end{align*}
Note that $w$ is a function of $y,z,p$ which by definition satisfy: $\sum_{k=1}^{4C} y_k=0$, $\sum_{k=1}^{4C} z_k=1$, $\min_{k\in[4C]} p_k \geq 1/(5C)$, $\sum_{k=1}^{4C} p_k=1$, and $y_k^2\leq p_k z_k$ for all $k\in[4C]$. Define a parameter space $T$:
\begin{align*}
T = \cbr{(y,z,p):\sum_{k=1}^{4C} y_k=0,\sum_{k=1}^{4C} z_k=1,\min_{k\in[4C]} p_k \geq 1/(5C),\sum_{k=1}^{4C} p_k=1, \text{ and }y_k^2\leq p_k z_k,\forall k\in[4C]}.
\end{align*}
Then we have
\begin{align}
\nonumber&\frac{1}{n}\lambda_{\min,\perp}(\mathcal{L}_M) \\
&\geq \inf_{(y,z,p)\in T}  \sum_{k,l\in[4C]:\abs{k-l}\leq 1}   p_l p_k \br{\br{w_{2k-1} - w_{2l-1}}^2 + \br{w_{2k-1} - w_{2l}}^2 + \br{w_{2k} - w_{2l-1}}^2 + \br{w_{2k} - w_{2l}}^2},\label{eqn:w_lower}
\end{align}
where $w$ is defined in (\ref{eqn:w_def}).

Since $T$ only depends on $C$, the quantity (\ref{eqn:w_lower}) also only depends on $C$. Then, (\ref{eqn:w_lower}) is equal to some constant $C'\geq 0$ only depending on $C$. We are going to show $C'>0$. Otherwise, let the infimum of (\ref{eqn:w_lower}) be achieved at some $w$ with $(y,z,p)\in T$. Then, we must have $w_{2k-1} = w_{2l-1} = w_{2k} = w_{2l}$ for any $k,l\in[4C]$ such that $\abs{k-l}\leq 1$. This has two immediately implications. First, 
 for any $k\in[4C]$, since $w_{2k-1}=w_{2k}$, we have $y_k^2 =p_k z_k$ and $w_k = y_k/(2p_k)$,  Second, since $w_{2k} = w_{2(k+1)}$ for any $k\in[4C-1]$, there exists some $c$ such that $y_k/p_k = c$ for all $k\in[4C]$. Together with $y_k^2 =p_k z_k$, we obtain $c^2 p_k = z_k$ for all $k\in[C]$. Since $\sum_{k=1}^{4C} z_k=1$ and $\sum_{k=1}^{4C} p_k=1$, we conclude $c=\pm 1$. Then using $y_k/p_k = c$, we have $\sum_{k=1}^{4C} y_k = c \sum_{k=1}^{4C} p_k = c \neq 0$, which is a contradiction with $\sum_{k=1}^{4C} y_k=0$.  As a result, we obtain $\frac{1}{n}\lambda_{\min,\perp}(\mathcal{L}_M) \geq C' >0$.
 \end{proof}

\begin{lemma}\label{lem:bound-hessian}
Under the assumptions in Lemma \ref{lem:prev-paper}, 
$$\lambda_{\min,\perp}(H(\eta^*))=\min_{u\neq 0: \mathds{1}_{m}^Tu=0}\frac{u^TH(\eta^*)u}{\|u\|^2} \gtrsim mp$$
with probability at least $1-O(n^{-10})$, where $H(\eta^*)$ is the Hessian matrix of the objective (\ref{eq:MLE-small}), defined by
$$H_{ij}(\eta^*)=\begin{cases}
\sum_{l\in[m]\backslash\{i\}}B_{il}\psi'(\eta_i^*-\eta_l^*), & i=j, \\
-B_{ij}\psi'(\eta_i^*-\eta_j^*), & i\neq j.
\end{cases}$$
\end{lemma}
\begin{proof}
We can decompose $H(\eta^*)$ into stochastic part $H(\eta^*)-\E(H(\eta^*))$ ans deterministic part $\E(H(\eta^*))$ and bound them separately. We first look at the stochastic part. Note that
$$H(\eta^*)-\E(H(\eta^*))=D - \E(D)-\sum_{i<j}(B_{ij}-p_{ij})\psi^\prime(\eta_i^*-\eta_j^*)(E_{ij}+E_{ji})$$
where $D=diag\{D_1,...,D_m\}=diag\{\sum_{j\neq 1}B_{ij}\psi^\prime(\eta_1^*-\eta_j^*),...,\sum_{j\neq m}B_{mj}\psi^\prime(\eta_m^*-\eta_j^*)\}$; $E_{ij}$ is an $m\times m$ matrix and has 1 on the entry $(i,j)$ and 0 otherwise. We also have $\opnorm{(B_{ij}-p_{ij})\psi^\prime(\eta_i^*-\eta_j^*)(E_{ij}+E_{ji})}\leq 1$ and $\opnorm{\sum_{i<j}(B_{ij}-p_{ij})^2\psi^\prime(\eta_i^*-\eta_j^*)^2(E_{ij}+E_{ji})^2}\leq mp$. By matrix Bernstein inequality in \cite{tropp2015introduction}, we have
$$\p\left(\opnorm{\sum_{i<j}(B_{ij}-p_{ij})(E_{ij}+E_{ji})}>t\right)\leq2m\exp\left(-\frac{t^2/2}{mp+\frac{t}{3}}\right).$$
Taking $t=C_1^\prime\sqrt{mp\log n}$ for some large enough constant $C_1^\prime>0$, we have 
$$\opnorm{\sum_{i<j}(B_{ij}-p_{ij})\psi^\prime(\eta_i^*-\eta_j^*)(E_{ij}+E_{ji})}\leq C_1^\prime\sqrt{mp\log n}$$
with probability at least $1-O(n^{-10})$. Standard concentration using Bernstein inequality also yields 
$$\opnorm{D-\E(D)}\leq C_2^\prime\sqrt{mp\log n}$$
for some constant $C_2^\prime>0$ with probability at least $1-O(n^{-10})$. Thus the stochastic part
\begin{equation}
\opnorm{H(\eta^*)-\E(H(\eta^*))}\leq(C_1^\prime+C_2^\prime)\sqrt{mp\log n}=o(mp)\label{eq:stoc-small}
\end{equation}
with probability at least $1-O(n^{-10})$. 

For the deterministic part, we first choose a constant integer $C^\prime>0$ such that for any $\abs{i-j}\leq\frac{n}{C^\prime}$, $p_{ij}=p$. Thus for any unit vector $x\in\mathbb{R}^m$ such that $\mathds{1}_m^Tx=0$, 
\begin{align}
\nonumber&\frac{x^T\E(H(\eta^*))x}{m}=\frac{\sum_{i<j}p_{ij}\psi^\prime(\eta_i^*-\eta_j^*)(x_i-x_j)^2}{m}\\
\nonumber&\geq\frac{\sum_{i<j, \abs{i-j}\leq\frac{m}{C^\prime}}p\psi^\prime(\eta_i^*-\eta_j^*)(x_i-x_j)^2}{m}\\
&\gtrsim p\frac{\sum_{i<j, \abs{i-j}\leq\frac{m}{C^\prime}}(x_i-x_j)^2}{m}\label{eq:use-bounded}\\
&\gtrsim p\label{eq:use-band-mat}
\end{align}
where (\ref{eq:use-bounded}) uses the boundedness of $\eta_1^*-\eta_m^*$; (\ref{eq:use-band-mat}) is a consequence of Lemma \ref{lem:band_laplacian_eigenvalue} and $C^\prime$ is a constant independent of $m$ and $n$. Combing (\ref{eq:stoc-small}) and (\ref{eq:use-band-mat}) concludes the proof.
\end{proof}

The proof of Lemma \ref{lem:prev-paper} is given below.
\begin{proof}[Proof of Lemma \ref{lem:prev-paper}]
Since $\frac{L(\eta_i^*-\eta_j^*)^2}{2(W_{i}(\eta^*)+W_j(\eta^*))}\asymp mpL(\eta_i^*-\eta_j^*)^2$, we only need to consider the situation where $mpL(\eta_i^*-\eta_j^*)^2$ is greater than a sufficiently large constant, since otherwise we can use the trivial bound $\mathbb{P}\left(\wh{\eta}_i<\wh{\eta}_j\right)\leq 1$. Define
$$\wt{\eta}_j=\eta_j^*-\frac{\sum_{l\in[m]\backslash\{j\}}B_{jl}(\bar{y}_{jl}-\psi(\eta_j^*-\eta_l^*))}{\sum_{l\in[m]\backslash\{j\}}B_{jl}\psi'(\eta_j^*-\eta_l^*)}.$$
Following the same argument used in the proof of Theorem 3.2 of \cite{chen2020partial}, we have
\begin{equation}
|\wh{\eta}_i-\wt{\eta}_i|\vee|\wh{\eta}_j-\wt{\eta}_j|\leq \delta\Delta,\label{eq:bias-very-small}
\end{equation}
with probability at least $1-O(n^{-7})-\exp(-\Delta^{3/2}Lmp)-\exp\left(-\Delta^2mpL\frac{mp}{\log(n+m)}\right)$,
where $\Delta=\min\left(\eta_i^*-\eta_j^*,\left(\frac{\log(n+m)}{mp}\right)^{1/4}\right)$ and $\delta>0$ is some sufficiently small constant. In fact, the bound (\ref{eq:bias-very-small}) has only been established in \cite{chen2020partial} with a random graph that satisfies $p_{ij}=p$ for all $1\leq i<j\leq m$. To establish (\ref{eq:bias-very-small}) under the more general setting of interest, we first have
\begin{equation}
\lambda_{\min,\perp}(H(\eta^*))=\min_{u\neq 0: \mathds{1}_{m}^Tu=0}\frac{u^TH(\eta^*)u}{\|u\|^2} \gtrsim mp, \label{eq:Hessian-more-general-graph}
\end{equation}
with high probability, where $H(\eta^*)$ is the Hessian matrix of the objective (\ref{eq:MLE-small}). This is established in Lemma \ref{lem:bound-hessian}.
Note that (\ref{eq:Hessian-more-general-graph}) is the only difference between the proofs of the current setting and the setting in \cite{chen2020partial}. With (\ref{eq:bias-very-small}), we have
\begin{eqnarray*}
\mathbb{P}\left(\wh{\eta}_i<\wh{\eta}_j\right) &\leq& \mathbb{P}\left(\wt{\eta}_j-\eta_j^*-(\wt{\eta}_i-\eta_i^*)>(1-\delta)\Delta\right) \\
&& + O(n^{-7})+\exp(-\Delta^{3/2}Lmp)+\exp\left(-\Delta^2mpL\frac{mp}{\log(n+m)}\right).
\end{eqnarray*}
Define
$$\mathcal{B}=\left\{B: \left|\frac{\sum_{l\in[m]\backslash\{j\}}p_{jl}\psi'(\eta_j^*-\eta_l^*)}{\sum_{l\in[m]\backslash\{j\}}B_{jl}\psi'(\eta_j^*-\eta_l^*)}-1\right|\leq\delta,\left|\frac{\sum_{l\in[m]\backslash\{i\}}p_{il}\psi'(\eta_i^*-\eta_l^*)}{\sum_{l\in[m]\backslash\{i\}}B_{il}\psi'(\eta_i^*-\eta_l^*)}-1\right|\leq\delta'\right\}.$$
By Bernstein's inequality, we have $\mathbb{P}(B\in \mathcal{B}^c)\leq O(n^{-7})$ for some $\delta'=o(1)$. We then have
\begin{eqnarray*}
&& \mathbb{P}\left(\wt{\eta}_j-\eta_j^*-(\wt{\eta}_i-\eta_i^*)>(1-\delta)\Delta\right) \\
&\leq& \sup_{B\in\mathcal{B}}\mathbb{P}\left(-\frac{\sum_{l\in[m]\backslash\{j\}}B_{jl}(\bar{y}_{jl}-\psi(\eta_j^*-\eta_l^*))}{\sum_{l\in[m]\backslash\{j\}}B_{jl}\psi'(\eta_j^*-\eta_l^*)}\right. \\
&& \left.+\frac{\sum_{l\in[m]\backslash\{i\}}B_{il}(\bar{y}_{il}-\psi(\eta_i^*-\eta_l^*))}{\sum_{l\in[m]\backslash\{i\}}B_{il}\psi'(\eta_i^*-\eta_l^*)}>(1-\delta)\Delta\Big|B\right) + O(n^{-7}) \\
&\leq& \exp\left(-\frac{(1-2\delta)L(\eta_i^*-\eta_j^*)^2}{2(W_{i}(\eta^*)+W_j(\eta^*))}\right) + O(n^{-7}).
\end{eqnarray*}
Since
$$\exp(-\Delta^{3/2}Lmp)+\exp\left(-\Delta^2mpL\frac{mp}{\log(n+m)}\right)\lesssim \exp\left(-\frac{(1-2\delta)L(\eta_i^*-\eta_j^*)^2}{2(W_{i}(\eta^*)+W_j(\eta^*))}\right) + O(n^{-7}),$$
we obtain the desired conclusion.

\end{proof}

\bibliographystyle{dcu}
\bibliography{reference}

\end{document}